\documentclass[12pt]{article}

\usepackage{graphicx}
\usepackage{CJK}
\usepackage{authblk}
\usepackage{bm}
\usepackage{tablefootnote}
\usepackage{threeparttable}
\usepackage{footnote}
\usepackage{amsmath, amssymb, amsthm, enumerate}
\usepackage{algorithm}
\usepackage{algpseudocode}
\usepackage{multirow}
\usepackage{array}
\usepackage{graphicx} 
\usepackage{epstopdf}
\newcommand{\C}[1]{{\cal {#1}}}
\newcommand{\tr}{^{\sf T}}
\newcommand{\m}[1]{{\bf{#1}}}
\newtheorem{theorem}{Theorem}[section]
\newtheorem{lemma}[theorem]{Lemma}
\newtheorem{corollary}[theorem]{Corollary}
\newtheorem{remark}{Remark}
\newtheorem{assum}{Assumption}
\usepackage{float}
\usepackage{subfig}

\usepackage{amsmath,amssymb,amsthm,enumerate,mathrsfs}

\theoremstyle{definition}

\usepackage[left=2cm,right=2cm]{geometry}

\numberwithin{equation}{section}

\begin{document}
	\makeatletter

	\title{A Decentralized Primal-Dual Method with Quasi-Newton Tracking}
	\begin{CJK}{UTF8}{gbsn}

		\author[a]{Liping Wang\thanks{E-mail address: wlpmath@nuaa.edu.cn}}
		\author[a]{Hao Wu \thanks{E-mail address: wuhoo104@nuaa.edu.cn}
		}
		\author[b]{Hongchao Zhang\thanks{E-mail address: hozhang@math.lsu.edu}}
		
		\affil[a]{School of Mathematics, Nanjing University of Aeronautics and Astronautics.}
		
		\affil[b]{Department of Mathematics, Louisiana State University.}
		\maketitle

	\end{CJK}
	
	\vspace{2mm}

	\footnotesize{
		\noindent\begin{minipage}{18cm}
			{\bf Abstract:}
			This paper considers the decentralized optimization problem of minimizing a finite sum of strongly convex and twice continuously differentiable functions over a fixed-connected undirected network. A fully decentralized primal-dual method(DPDM) and its generalization(GDPDM), which allows for multiple primal steps per iteration, are proposed. 
			In our methods, both primal and dual updates use second-order information obtained by quasi-Newton techniques which only involve matrix-vector multiplication.
			Specifically, the primal update applies a Jacobi relaxation step using the BFGS approximation for both computation and communication efficiency. 
			The dual update employs a new second-order correction step. 
			We show that the decentralized local primal updating direction on each node asymptotically approaches the centralized quasi-Newton direction. 
			Under proper choice of parameters, GDPDM including DPDM has global linear convergence 
			for solving strongly convex decentralized optimization problems. 
			Our numerical results show both GDPDM and DPDM are very efficient compared with other state-of-the-art methods
			for solving decentralized optimization.
		\end{minipage}
		\\[5mm]
		

		\hbox to14cm{\hrulefill}\par


\section{INTRODUCTION}
In this paper, we consider the following decentralized optimization problem over an undirected and connected network containing $n$ nodes
\begin{equation}\label{obj_fun1}
	{\m{z}^*} = \arg \mathop {\min }\limits_{\m{z} \in {R^p}} \sum\limits_{i = 1}^n {{f_i}(\m{z})},
\end{equation}
where the local objective function $f_i:R^{p} \rightarrow R$, $i=1,\ldots,n$, on each node is strongly convex and twice continuously differentiable. We denote two nodes as neighbors if they are connected by an
edge. In a decentralized setting, there does not exist one central server to gather local information from all nodes, compute shared global information, and broadcast it back to the whole network. 
Each local function $f_i$ is only known to node $i$. All the nodes collaborate with their neighbors through information exchange (i.e., communication) to finally
obtain the global solution $\m{z}^*$. Decentralized optimization has wide applications including decentralized resources control\cite{fusco2021decentralized}, wireless networks\cite{jeong2022asynchronous}, decentralized machine learning\cite{zhang2022distributed}, power systems\cite{chen2020fully}, federated learning\cite{pillutla2022robust}, etc..

As more important practical applications surge up, the decentralized optimization methods have been extensively studied in recent years, where first-order methods 
gain attention very rapidly due to their simple iterative schemes and low computational cost per iteration.
The decentralized gradient descent(DGD)\cite{nedic2009distributed,yuan2016convergence,zeng2018nonconvex} methods are a class of the most well-known first-order methods for solving decentralized optimization.
However, DGD methods usually converge to the optimal solution only with a diminishing step size, while a small constant stepsize often leads the iterate to a neighborhood of the minimizer\cite{yuan2016convergence}. 
There are many recent works devoted to a constant step size with guaranteed convergence. 
By significantly increasing the number of communication steps, the DGD method \cite{berahas2018balancing} achieves global convergence with R-linear rate. 
To ensure global convergence, EXTRA\cite{shi2015extra,li2020revisiting} as well as its modification, called  NIDS\cite{li2019decentralized},
adopts different mixing matrices at odd and even iterations. 
The method given in  \cite{yuan2018exact} combines the diffusion strategy with an important subclass of left-stochastic
matrices to obtain wider stability region and enhanced performance than EXTRA.
Gradient Tracking(GT) methods\cite{xu2015augmented,qu2017harnessing,nedic2017achieving,zhang2020distributed,song2023optimal} track the global average gradient
to design local search directions at any node.
An effective approach to achieve global convergence is to design the methods in the primal-dual domain.
A general method given in \cite{jakovetic2018unification} unified EXTRA and GT into a primal-dual framework.
A more flexible first-order primal-dual framework is proposed in \cite{mansoori2021flexpd}, where multiple primal steps per iteration are allowed. 
In addition, there are also some special classes of primal-dual methods based on alternating direction approaches \cite{zhu2016quantized,zhang2018admm,mancino2023decentralized}.

Although first-order methods are more simple and easily implementable, their asymptotic convergence speed is often slow for more accurate solutions. 
So, many second-order methods have been proposed recently to accelerate the convergence rate.
Some methods focus on penalized approaches for solving a constrained problem, where a consensus constraint is introduced to reformulate the problem \eqref{obj_fun1}.
NN\cite{mokhtari2016network} uses the Newton's method to solve the penalty problem. DQN\cite{bajovic2017newton} develops a diagonal correction technique to overcome the challenge 
that the Hessian inverse of the penalty function can not be computed in a decentralized way. 
DBFGS\cite{eisen2017decentralized} is a decentralized quasi-Newton method which only uses the local neighbor information.
However, NN, DQN, and DBFGS are inexact penalty methods in the sense that the penalty parameter needs to go to infinity for ensuring global convergence.
ESOM\cite{mokhtari2016decentralized}, PMM-DQN\cite{bajovic2017newton} and PD-QN\cite{eisen2019primal} improve NN, DQN and DBFGS, respectively, in the primal-dual framework.
Newton Tracking(NT)\cite{zhang2021newton} is also developed in the primal-dual domain, where the local search directions asymptotically track the global centralized Newton direction.
A decentralized ADMM \cite{li2022communication}  incorporates BFGS quasi-Newton techniques to improve computation efficiency, and 
the quadratically approximated ADMM given by \cite{mokhtari2016dqm} was improved in \cite{zhang2023decentralized} to promote communication efficiency. 

Among all the previously developed primal-dual methods, there are rarely methods applying second-order information in the dual update.
Moreover, when combining the quasi-Newton techniques into the decentralized methods in the primal-dual framework, some critical issues may arise.
First, the local quasi-Newton matrices generated by the BFGS approximation are not necessarily positive definite 
although the augmented Lagrangian function across all nodes is strongly convex. 
Second, when the primal problem is solved inexactly, the exact dual gradient is not obtained. 
Extracting useful Hessian information from inexact dual gradients also needs to be studied.  
Our goal of this paper is to explore second-order information in both primal and dual domains, and 
propose a new fully decentralized primal-dual quasi-Newton method with both theoretical global linear convergence
and numerical efficiency. 
Our main contributions are as follows. 

\begin{itemize}
	\item [1.]  In the primal domain, we combine the Jacobi relaxation technique with the BFGS approximation for updating the primal iterates.
	Multiple adaptive primal updates per iteration are also allowed in our approach to balance the accuracy obtained in both the primal and dual domains.
	\item [2.] In the dual domain, by applying the Newton's method to the dual problem and making some critical reformulations,
	we obtain a novel dual updating step, which can be viewed as adding a second-order correction term to the usual dual ascent step. 
	To maintain computational efficiency, we apply BB-approximation techniques to capture the spectral information of the dual Hessian. To the best of our knowledge,
	we are the first to give a primal-dual method using quasi-Newton dual update with guaranteed convergence.
	\item [3.] Our proposed method is a quasi-Newton tracking method, where the local search direction on each node tracks the global centralized quasi-Newton direction. 
	Our quasi-Newton tracking can be considered as a generalization of the Newton tracking \cite{zhang2021newton}. 
	In addition, our approaches only involve matrix-vector multiplications to avoid inverting a matrix. 
	The numerical results show that our methods are very efficient compared with other state-of-the-art methods
	for solving decentralized optimization.
\end{itemize}

The paper organized is as follows: In Section~2, we reformulate the problem as a constrained decentralized optimization, and develop our new methods. 
Global convergence as well as linear convergence rate are established in Section~3.
Numerical experiments are performed in Section~4 to compare our method with other well-established first and second-order methods for solving decentralized optimization.
Finally, we draw some conclusions in Section~5.
\subsection{Notation}
We use uppercase and lowercase boldface letters to denote matrices and vectors, respectively. We let $\m{x}_i$ denote the local copy of the global variable $\m{z}$ at node $i$ and define $\mathcal{N}_i$ as the set consisting of the neighbors of node $i$ (we treat node $i$ itself as one of its neighbors for convenience.). $\bar{\m{x}}$ denotes $\frac{1}{n} \sum_{i=1}^n \m{x}_i$. $\mbox{Null}(\m{M})$ denotes the null space of $\m{M}$. We let $\operatorname{span}(\m{v})$ denote the linear subspace spanned by $\m{v}$. Kronecker Product is denoted as $\otimes$. $\operatorname{Proj}_{[a,b]} (\cdot)$ is the projection operator onto interval $[a,b]$. 
Given a symmetric matrix $\m{M}$, ${\lambda _{\min }(\m{M})}$ and ${\lambda _{\max }(\m{M})}$ denote its smallest eigenvalue and largest eigenvalue, respectively, and  
$\lambda_{2}(\m{M})$ denotes its second smallest eigenvalue. The trace and determinant of $\m{M}$ is denoted as  $\operatorname{tr}(\m{M})$ and  $\operatorname{det}(\m{M})$, respectively. 
For any vector $\m{v}$, $\left\|\m{v}\right\|_{\m{M}}^2$ denotes $\m{v}\tr \m{M} \m{v}$.
$\m{A}\tr$ denotes transpose of the matrix $\m{A}$ and $\m{A}^{\dagger}$ denotes its pseudo inverse. 
$\m{M}_1 \succeq \m{M}_2$ means $\m{M}_1 - \m{M}_2$ is positive definite.
We use $\m{I}_p$ to denote the $p \times p$ identity matrix. Especially, $\m{I}$ denotes $\m{I}_{np}$.  
We let $\operatorname{diag}\{a_1,\ldots,a_p\}$ denote a $p \times p$ diagonal matrix whose diagonal elements are $a_1,\ldots,a_p$,
and $\operatorname{log}(\cdot)$ denote $\operatorname{log}_{10}(\cdot)$. 
For any scalar $\theta \in R$, if $\theta \ne 0$, $I_{R \setminus \{0\}} = 1$; otherwise, $I_{R \setminus \{0\}} = 0$.

	\section{PROBLEM FORMULATION AND ALGORITHM DEVELOPMENT}
	\subsection{Problem Formulation}
	
	Since the network is connected, the constraints $\m{x}_i=\m{x}_j, i,j = 1,\ldots,n$, are equivalent to $\m{x}_i=\m{x}_j, j \in \mathcal{N}_i, i = 1,\ldots,n$. Therefore, the problem \eqref{obj_fun1} can be reformulated as the following constrained optimization 
	\begin{align}\label{obj_fun2}
		\left\{ {\m{x}_i^*} \right\}_{i = 1}^n = &\arg \mathop {\min }\limits_{\left\{ {\m{x}_i} \right\}_{i = 1}^n } \sum\limits_{i = 1}^n {{f_i}({\m{x}_i})},  \\
		&s.t.{\rm{ }}{\m{x}_i} = {\m{x}_j},\forall j \in {\mathcal{N}_i},i = 1,\ldots,n.\notag
	\end{align}
	The consensus constraints ensure the solution of problem \eqref{obj_fun2} is equal to the solution of \eqref{obj_fun1} and $\m{x}_1^*=\m{x}_2^*= \ldots =\m{z}^*$.
	
	To reveal the network structure, we introduce a mixing matrix  $\tilde{\m{W}} \in R^{n \times n}$ which
	has the following standard features.\\
	$\mathbf{1}$. $\tilde{\m{W}}$ is nonnegative and $\tilde{W}_{ij}$ characterizes the active link $(i,j)$, i.e., $\tilde{{W}}_{ij}>0$ if $j \in \mathcal{N}_i$,  $\tilde{{W}}_{ij}=0$ otherwise.\\
	$\mathbf{2}$. $\tilde{\m{W}}$ is symmetric and doubly stochastic, i.e., $\tilde{\m{W}}=\tilde{\m{W}}\tr$ and $\tilde{\m{W}}\m{1}_n=\m{1}_n$.
	
	It has a few common choices for the mixing matrix $\tilde{\m{W}}$, such as Laplacian-based constant edge weight matrix \cite{sayed2014diffusion} and Metropolis constant edge weight matrix \cite{xiao2007distributed}. 
	From the second feature, we have $\mbox{Null}(\m{I}_n-\tilde{\m{W}})=\operatorname{span}(\m{1}_n)$. Using the mixing matrix $\tilde{\m{W}}$, we can rewrite problem \eqref{obj_fun2} in an equivalent compact form. Let us denote
	\begin{align*}
		\m{x} = [{\m{x}_1};{\m{x}_2};\ldots;{\m{x}_n}] \in {R^{np}}, \quad f(\m{x}) = \sum\limits_{i = 1}^n {{f_i}({\m{x}_i})}, \\
		\m{W} = \tilde{\m{W}} \otimes {\m{I}_p} \in {R^{np \times np}} \mbox{ and } \m{I}-\m{W} = {({\m{I}_n} - \tilde{\m{W}})} \otimes {\m{I}_p}.
	\end{align*}
	We can easily know that the equation $(\m{I}-\m{W})\m{x}=\m{0}$ holds if and only if ${\m{x}_1}={\m{x}_2}=\ldots={\m{x}_n}$. According to the Perron-Frobenius theorem \cite{pillai2005perron}, the eigenvalues of $\tilde{\m{W}}$ lie in $(-1,1]$ and 1 is the single eigenvalue. Since $\m{I}-\m{W}$ is symmetric positive semidefinite,  ${(\m{I} - \m{W})^{1/2}}$ has the same null space as $\m{I}-\m{W}$. 
	Hence, problem \eqref{obj_fun2} can be reformulated as
	\begin{align}\label{obj_fun3}
		{\m{x}^*} = &\arg \mathop {\min }\limits_{\m{x} \in {R^{np}}} f(\m{x}),\\
		&s.t.{\rm{ }}{(\m{I} - \m{W})^{1/2}}\m{x} = \m{0}.\notag
	\end{align}
	
	In this paper, we have the following assumptions on the objective function.
	\begin{assum}\label{as1}
		The local gradients $\left\{ \nabla f_i(\m{z}) \right\}_{i = 1}^n $ are Lipschitz continuous with constant $L>0$, i.e.,
		\begin{equation}\label{3.1}
			\left\| {\nabla {f_i}(\m{z}) - \nabla {f_i}(\tilde{\m{z}})} \right\| \le L \left\| {\m{z} - \tilde{\m{z}}} \right\|,
		\end{equation}
		$\forall \m{z},\tilde{\m{z}} \in {R^p},i = 1,\ldots,n$.
	\end{assum}
	
	\begin{assum}\label{as2}
		The local objective functions $\left\{ f_i(\m{z}) \right\}_{i = 1}^n $ are strongly convex with
		modulus $\mu>0$, i.e.,
		\begin{equation}\label{3.2}
			{f_i}(\tilde{\m{z}}) \ge {f_i}(\m{z}) + \nabla {f_i}{(\m{z})\tr}(\tilde{\m{z}}- \m{z}) + \frac{\mu }{2}{\left\| {\tilde{\m{z}} - \m{z}} \right\|^2},
		\end{equation}
		$\forall \m{z},\tilde{\m{z}} \in {R^p},i = 1,\ldots,n$.
	\end{assum}
	Combining \textbf{Assumption \ref{as1}} with \textbf{Assumption \ref{as2}}, we have
	\begin{equation}\label{3.3}
		\mu {\m{I}_p} \preceq {\nabla ^2}{f_i}(\m{z}) \preceq L {\m{I}_p}, \; \forall \m{z} \in {R^p},i = 1,\ldots,n.
	\end{equation}
	Since the Hessian $\nabla^2 f(\m{x})$ is a block diagonal matrix whose $i$-th diagonal block is $\nabla^2 f_i(\m{x}_i)$, the above bounds also hold for $\nabla^2 f(\m{x})$, that is
	\begin{equation}\label{3.4}
		\mu {\m{I}} \preceq {\nabla ^2}f(\m{x}) \preceq L {\m{I}}, \; \forall \m{x} \in {R^{np}}.
	\end{equation}
	
	As discussed previously, primal-dual methods are effective approaches to solving the problem \eqref{obj_fun3}.  
	However, most existing  methods only use first-order information to update the dual variables
	by employing a dual ascent step, such as
	\begin{equation}
		\bm{\lambda}^{t+1}=\bm{\lambda}^{t}+\alpha{{(\m{I} - \m{W})}^{1/2}}\m{x}^{t+1},
	\end{equation}
	where $\bm{\lambda}^{t}$ is the Lagrangian multiplier, also called dual variable. 
	In the following, we propose decentralized primal-dual methods whose primal and dual updates
	will use second-order information approximated by certain quasi-Newton techniques.

	\subsection{Algorithm Development}
	
	The augmented Lagrangian function of the problem \eqref{obj_fun3} is 
	\begin{equation}\label{AL2}
		\tilde{L}_{\alpha}(\m{x},\bm{\lambda}) = f(\m{x}) + \left\langle {\bm{\lambda},{{(\m{I} - \m{W})}^{1/2}}\m{x}} \right\rangle  + \frac{\alpha }{2}{\m{x}\tr}(\m{I} - \m{W})\m{x},
	\end{equation}
	where $\bm{\lambda} = [{\bm{\lambda}_1};{\bm{\lambda}_2};\ldots;{\bm{\lambda}_n}] \in R^{np}$, $\alpha >0 $ is the penalty parameter. \eqref{AL2} can be also viewed as the Lagrangian function of the following 
	penalized optimization
	\begin{align}\label{obj_fun4}
		{\m{x}^*} = &\arg \mathop {\min }\limits_{\m{x} \in {R^{np}}} f(\m{x})+\frac{\alpha}{2}{\m{x}\tr}(\m{I} - \m{W})\m{x},\\
		&s.t.{\rm{ }}{(\m{I} - \m{W})^{1/2}}\m{x} = \m{0}.\notag
	\end{align}
	Obviously, the problem \eqref{obj_fun4} is equivalent to the problem \eqref{obj_fun3}. 
	Let $\m{x}^*(\bm{\lambda})$ denote the minimizer of $\tilde{L}_{\alpha}(\cdot,\bm{\lambda})$, i.e.,
	\begin{equation}\label{pu0}
		\m{x}^*(\bm{\lambda})=\text{arg} \mathop {\min }\limits_{\m{x} \in R^{np}} \tilde{L}_{\alpha}(\m{x},\bm{\lambda}).
	\end{equation}
	The optimality condition of \eqref{pu0} gives
	\begin{equation}\label{le}
		\nabla f(\m{x}^*(\bm{\lambda}))+(\m{I}-\m{W})^{1/2}\bm{\lambda}+\alpha (\m{I}-\m{W})\m{x}^*(\bm{\lambda})=\m{0}.
	\end{equation}
	For any $\bm{\lambda}^{*} \in R^{np}$ that satisfies $(\m{I}-\m{W})^{1/2}\m{x}^*(\bm{\lambda}^*)=\m{0}$, $(\m{x}^*(\bm{\lambda}^*), \bm{\lambda}^*)$ 
	is an primal-dual solution of \eqref{obj_fun4}.
	We now apply Netwon's method to solve the feasibility system 
	\begin{equation}
		(\m{I}-\m{W})^{1/2}\m{x}^*(\bm{\lambda})=\m{0}.
	\end{equation}
	The $t$-th iteration of Netwon's method  gives
	\begin{align}\label{le1}
		(\m{I}-\m{W})^{1/2}\m{x}^*(\bm{\lambda}^t)+(\m{I}-\m{W})^{1/2} \left(\frac{\partial \m{x}^*(\bm{\lambda})}{\partial \bm{\lambda}} \right)\tr_{\bm{\lambda}=\bm{\lambda}^t} \left(\bm{\lambda}^{t+1}-\bm{\lambda}^t\right)=\m{0}.\notag
	\end{align}
	Differentiating \eqref{le} with respect to $\bm{\lambda}$, we have
	\[
	\left(\nabla^2 f(\m{x}^*(\bm{\lambda}))+\alpha(\m{I}-\m{W})\right)\left(\frac{\partial \m{x}^*(\bm{\lambda})}{\partial \bm{\lambda}} \right)\tr=-(\m{I}-\m{W})^{1/2},
	\]
	which gives
	\begin{equation}\label{le3}
		\left(\frac{\partial \m{x}^*(\bm{\lambda})}{\partial \bm{\lambda}} \right)\tr_{\bm{\lambda}=\bm{\lambda}^t}=-\left[\nabla^2_{\m{x}\m{x}}\tilde{L}_\alpha\left(\m{x}^*(\bm{\lambda}^t), \bm{\lambda}^t\right)\right]^{-1}(\m{I}-\m{W})^{1/2}.
	\end{equation}
	Substituting \eqref{le3} into \eqref{le1} yields
	\begin{align*}
		(\m{I}-\m{W})^{1/2}\m{x}^*(\bm{\lambda}^t)-(\m{I}-\m{W})^{1/2} \left[\nabla^2_{\m{x}\m{x}}\tilde{L}_\alpha\left(\m{x}^*(\bm{\lambda}^t), \bm{\lambda}^t\right)\right]^{-1}
		(\m{I}-\m{W})^{1/2} \left(\bm{\lambda}^{t+1}-\bm{\lambda}^t\right)=\m{0}.
	\end{align*}
	By letting $\m{x}^{t+1}=\m{x}^*(\bm{\lambda}^t)$, we have
	\begin{equation}\label{du1}
		\bm{\lambda}^{t+1}=\bm{\lambda}^t+ \left(\m{G}^{t}\right)^{\dagger} (\m{I}-\m{W})^{1 / 2}\m{x}^{t+1},
	\end{equation}
	where 
	\begin{equation*}
		{\m{G}^{t}} = {(\m{I} - \m{W})^{1/2}}\left[\nabla _{\m{x} \m{x}}^2{\tilde{L}_\alpha }({\m{x}^{t+1}},{\bm{\lambda}^t})\right]^{ - 1}{(\m{I} - \m{W})^{1/2}}.
	\end{equation*}
	Solving the primal problem \eqref{pu0} exactly is not only numerically expensive but also theoretically unnecessary. 
	One practical approach is to simply apply one Newton's iteration to the problem \eqref{pu0} and let 
	\begin{align}\label{pu1}
		\m{x}^{t+1}=\m{x}^t-\left[\nabla_{\m{x}\m{x}}^2{\tilde{L}_\alpha }({\m{x}^t},{\bm{\lambda}^t})\right]^{-1}\Big[\nabla f\left(\m{x}^t\right)
+(\m{I}-\m{W})^{1 / 2} \bm{\lambda}^{t}+\alpha(\m{I}-\m{W}) \m{x}^t\Big].
	\end{align}
	Unfortunately, the iterative schemes \eqref{du1} and \eqref{pu1} can not be directly applied in the decentralized setting.
	Note that ${\nabla ^2}f({\m{x}})$ is a block diagonal matrix whose $i-$th diagonal block is ${\nabla ^2}f_i({\m{x}_i})$ and 
	$\m{W}$ is a block sparse matrix related to the network structure. Thus, the Hessian $\nabla_{\m{x}\m{x}}^2{\tilde{L}_\alpha }({\m{x}},{\bm{\lambda}})$ is neighbor related in the calculation.
	However, the Hessian inverse as well as  $(\m{I}-\m{W})^{1/2}$ will destroy the neighbor relation.
	In addition, it is often very expensive to calculate the inverse of primal and dual Hessians in many practical computations. 
	In the following, we would like to modify  the iterative schemes \eqref{du1} and \eqref{pu1} into a decentralized setting that
	can be also computed in an efficient way. 
	
	\paragraph{Primal update}
	Note that \eqref{pu1} is equivalent to
	\begin{align*}
		\m{x}^{t+1}=\m{x}^t-\m{d}^t,
	\end{align*}
	where $\m{d}^t$ satisfies
	\begin{align}\label{fcz}
		\nabla_{\m{x}\m{x}}^2{\tilde{L}_\alpha }({\m{x}^t},{\bm{\lambda}^t}) \m{d}^t =\nabla_{\m{x}}{\tilde{L}_\alpha }({\m{x}^t},{\bm{\lambda}^t}).
	\end{align}
	We would solve the linear equations \eqref{fcz} inexactly by applying the Jacobi relaxation technique \cite{young2014iterative}, which is an effective technique for the iterative method. 
	By splitting the Hessian $\nabla_{\m{x}\m{x}}^2{\tilde{L}_\alpha }({\m{x}^t},{\bm{\lambda}^t})$ into two parts, $\nabla^2f(\m{x}^t)$ and $\alpha(\m{I}-\m{W})$, 
	one Jacobi relaxation iteration can be conducted by
	\begin{align*}
		\m{d}^{t,k+1}
		=\theta \left(-\alpha\nabla^2f(\m{x}^t)^{-1}(\m{I}-\m{W})\m{d}^{t,k}+\nabla^2f(\m{x}^t)^{-1} 
		\nabla_x\tilde{L}_\alpha (\m{x}^t,\bm{\lambda}^t) \right)
		+(1-\theta)\m{d}^{t,k}.
	\end{align*}
	where $\theta \in [0,1]$ is the relaxation parameter, introduced to accelerate the original Jacobi iteration. 
	For both computation and communication efficiency, our algorithm would only employ one Jacobi iteration as the following: 
	\begin{align*}
		\m{d}^{t,1}=\theta(-\alpha\nabla^2f(\m{x}^t)^{-1}(\m{I}-\m{W})\m{d}^{t,0}+\nabla^2f(\m{x}^t)^{-1}{\tilde{L}_\alpha }({\m{x}^t},{\bm{\lambda}^t}))
		+(1-\theta)\m{d}^{t,0},
	\end{align*}
	where 
	\begin{align*}
		\m{d}^{t,0}=\nabla^2f(\m{x}^t)^{-1}\nabla_{\m{x}}{\tilde{L}_\alpha }({\m{x}^t},{\bm{\lambda}^t}),
	\end{align*}
	and let $\m{d}^{t}=\m{d}^{t,1}$. Combining the above steps, we obtain
	\begin{align}\label{Ja}
		\m{d}^{t}=\left(\m{I}-\theta\alpha\nabla^2f(\m{x}^t)^{-1}(\m{I}-\m{W})\right)\nabla^2f(\m{x}^t)^{-1}\nabla_{\m{x}}{\tilde{L}_\alpha }({\m{x}^t},{\bm{\lambda}^t}).
	\end{align}
	Note that if $\theta=0$, \eqref{Ja} reduces to
	\begin{align}\label{dNT}
		\m{d}^{t}=\nabla^2f(\m{x}^t)^{-1}\nabla_{\m{x}}{\tilde{L}_\alpha }({\m{x}^t},{\bm{\lambda}^t}),
	\end{align}
	which is exactly primal update direction in NT\cite{zhang2021newton} with the parameter $\epsilon=0$.
	To avoid computing ${\nabla ^2}f({\m{x}}^t)$ and its inverse, we would approximate ${\nabla ^2}f({\m{x}}^t)$ by a block positive definite matrix $\m{B}^{t}$
	using BFGS quasi-Newton techniques. 
	Then, by  introducing a stepsize $\beta>0$ for global convergence, we have the following primal update,
	\begin{align}\label{pu11}
		\m{x}^{t+1}=\m{x}^{t}-\beta\left[\m{I}-\theta\alpha(\m{B}^t)^{-1}(\m{I}-\m{W})\right](\m{B}^t)^{-1}\nabla_{\m{x}}{\tilde{L}_\alpha }({\m{x}^t},{\bm{\lambda}^t}).
	\end{align}
	More specifically, we update $\m{B}^t$ by the following BFGS formula
	\begin{equation}\label{B}
		{\m{B}^{t+1}} = \left[ {\begin{array}{*{20}{c}}
				{{ \m{B}_1^{t + 1}}
				}&{}&{}\\
				{}&{...}&{}\\
				{}&{}&{\m{B}_n^{t+1}}
		\end{array}} \right],
	\end{equation}
	where we set $\m{B}_i^0=\m{I}_p$, 
	\begin{equation}\label{Bit}
		\m{B}_i^{t + 1} = \m{B}_i^t - \frac{{\m{B}_i^t\m{s}_i^t{{(\m{s}_i^t)}\tr}\m{B}_i^t}}{{{(\m{s}_i^t)}\tr}\m{B}_i^t\m{s}_i^t} + \frac{{\m{y}_i^t{{(\m{y}_i^t)}\tr}}}{{{(\m{s}_i^t)}\tr}\m{y}_i^t}, \quad t \ge 0,
	\end{equation}
	$\m{s}_i^t = \m{x}_i^{t + 1} - \m{x}_i^t$ and $\m{y}_i^t = \nabla {f_i}( {\m{x}_i^{t + 1}} ) - \nabla {f_i}( {\m{x}_i^t} )$.
	
	We have some comments about $\m{B}^t$. First, since $f_i$ is strongly convex, we have $ (\m{s}_i^t) \tr \m{y}_i^t > 0$,
	which together with the positive definite initialization of $\m{B}_i^0$ ensures $\m{B}^t \succ \m{0}$ for all $t \ge 0$.
	Second, as usual, in practical implementation we always directly update $\m{H}^t := ({\m{B}^t})^{ - 1}$. 
	In particular, using the BFGS inversion formula, we would have 
	$\m{H}_i^{0}=({\m{B}_i^0})^{ - 1}$ and
	\begin{align}\label{Hit}
		\m{H}_i^{t+1}=\m{H}_i^{t}-\frac{\m{H}_i^{t} {\m{y}}_i^t (\m{s}_i^t)\tr+\m{s}_i^t ({\m{y}}_i^t)\tr \m{H}_i^{t}}{(\m{s}_i^t)\tr {\m{y}}_i^t}
		+\left(1+\frac{({\m{y}}_i^t)\tr \m{H}_i^{t} {\m{y}}_i^t}{(\m{s}_i^t)\tr {\m{y}}_i^t}\right)\frac{ \m{s}_i^t(\m{s}_i^t)\tr}{(\m{s}_i^t)\tr {\m{y}}_i^t},\quad t \ge 0.
	\end{align}
	Note that $\m{H}^t$ is applied in algorithm implementation while $\m{B}^t$ is only 
	used in the paper for convenient theoretical analysis.
	
	Finally, we want to emphasize that it is improper to use BFGS techniques to directly approximate $\nabla_{\m{x}\m{x}}^2 \tilde{L}_{\alpha}(\m{x}^t,\bm{\lambda}^t) =
	\nabla^2 f(\m{x}^t) + \alpha (\m{I}-\m{W})$ in decentralized setting. Although $\m{x}\tr(\m{I}-\m{W})\m{x} \ge 0$ for all $\m{x}$, its restriction on each node, i.e.,    
	$(\m{x}_i)\tr\left( {\m{x}_i - \sum_{j \in {{\cal N}_i}} {{\tilde{W}_{ij}}} \m{x}_j} \right)$ may not be positive. This may destroy the positive definite property of BFGS matrices.


	\paragraph{Dual update}
	We first give the following lemma, which can be considered as an extension of
	\textbf{Lemma 1} in \cite{yuan2014analysis}.
	\begin{lemma}\label{lem3}
		Let $\m{A} \in R^{m \times n}$ and $\operatorname{rank}(\m{A})=r$ where $r \le m \le n$. Let $\m{M} \in R^{n \times n} $ be symmetric positive definite. Then
		\begin{equation}\label{lem2.2 eq1}
			\left[\m{A}(\m{M}+\m{A}\tr\m{A})^{-1}\m{A}\tr\right]^{\dagger}\m{A}=\left[\m{I}+(\m{A}\m{M}^{-1}\m{A}\tr)^{\dagger}\right]\m{A},
		\end{equation}
		which further implies
		\begin{equation}\label{lem2.2 eq3}
			\left[\m{A}(\m{M}+\m{A}\tr\m{N}\m{A})^{-1}\m{A}\tr\right]^{\dagger}\m{A}=\left[\m{N}+(\m{A}\m{M}^{-1}\m{A}\tr)^{\dagger}\right]\m{A},
		\end{equation}
		where $\m{N} \in R^{m \times m} $ is any symmetric positive definite matrix.
	\end{lemma}
	\begin{proof}
			To prove \textbf{Lemma \ref{lem3}}, we need the following result in \cite{yuan2014analysis}.
		\begin{lemma}\label{lem2}\cite{yuan2014analysis}
			Let the block matrix
			\begin{equation*}
				\m{M}=\left(\begin{array}{ll}
					\m{M}_{11} & \m{M}_{12} \\
					\m{M}_{21} & \m{M}_{22}
				\end{array}\right) \in {R}^{n \times n}
			\end{equation*}
			be nonsingular, where $\m{M}_{11}, \m{M}_{11}, \m{M}_{21}$ and $\m{M}_{22}$ are matrices in ${R}^{m \times m}, {R}^{m \times(n-m)}, {R}^{(n-m) \times m}$ and ${R}^{(n-m) \times(n-m)}$ respectively. If $\m{M}_{22}$ is also nonsingular, then the matrix $\m{M}_{11}-\m{M}_{12} \m{M}_{22}^{-1} \m{M}_{21}$ is also nonsingular, and moreover, the block representation of $\m{M}^{-1}$ can be written as
			\begin{equation*}
				\m{M}^{-1}=\left(\begin{array}{ll}
					\m{M}_{11} & \m{M}_{12} \\
					\m{M}_{21} & \m{M}_{22}
				\end{array}\right)^{-1}=\left(\begin{array}{ll}
					\m{Q}_{11} & \m{Q}_{12} \\
					\m{Q}_{21} & \m{Q}_{22}
				\end{array}\right)
			\end{equation*}
			such that
			\begin{equation*}
				\m{Q}_{11}=\left(\m{M}_{11}-\m{M}_{12} \m{M}_{22}^{-1} \m{M}_{21}\right)^{-1} .
			\end{equation*}
		\end{lemma}
		
		We now prove \textbf{Lemma \ref{lem3}}.
		By singular value decomposition(SVD) on \m{A}, \m{A} can be represented as $\m{U}\m{S}\m{V}\tr$ where $\m{U} \in R^{m \times m}$, $\m{V} \in R^{n \times n}$ are orthogonal matrices and $\m{S} \in R^{m \times n}$. The left-hand side of \eqref{lem2.2 eq1} is equal to
		\begin{align}\label{lem2.2 eq2}
			\left[\m{U}\m{S}\m{V}\tr(\m{M}+\m{V}\m{S}\tr\m{U}\tr\m{U}\m{S}\m{V}\tr)^{-1}\m{V}\m{S}\tr\m{U}\tr\right]^{\dagger}\m{A} 
			=\left[\m{U}\m{S}(\m{V}\tr\m{M}\m{V}+\m{S}\tr\m{S})^{-1}\m{S}\tr\m{U}\tr\right]^{\dagger}\m{A}.
		\end{align}
		Decompose 
		\begin{equation*}
			\m{U}=\left[\m{U}_1,\m{U}_2\right], \m{U}_1 \in R^{m \times r},\m{U}_2 \in R^{m \times (m-r)},
		\end{equation*}
		\begin{equation*}
			\m{V}=\left[\m{V}_1,\m{V}_2\right], \m{V}_1 \in R^{n \times r},\m{V}_2 \in R^{n \times (n-r)},
		\end{equation*}
		\begin{equation*}
			\m{S}=\left(\begin{array}{ll}
				\m{S}_{1} & \m{0} \\
				\m{0} & \m{0}
			\end{array}\right), \m{S}_{1} \in {R}^{r \times r} \text{ is a diagonal matrix.}
		\end{equation*}
		Then \eqref{lem2.2 eq2} can be written as
		\begin{align*}
			\Bigg\{ \left(\begin{array}{ll}
				\m{U}_1\m{S}_1 & \m{0} 
			\end{array}\right)\bigg[\left(\begin{array}{ll}
				\m{V}_1\tr\m{M}\m{V}_1 & \m{V}_1\tr\m{M}\m{V}_2 \\
				\m{V}_2\tr\m{M}\m{V}_1 & \m{V}_2\tr\m{M}\m{V}_2
			\end{array}\right)
			+\left(\begin{array}{ll}
				\m{S}_1\tr\m{S}_1 & \m{0} \\
				\m{0} & \m{0}
			\end{array}\right)\bigg]^{-1}\left(\begin{array}{ll}
				\m{S}_1\tr\m{U}_1\tr \\
				\m{0} 
			\end{array}\right)\Bigg\}^{\dagger}\m{A}.
		\end{align*}
		By applying \textbf{Lemma \ref{lem2}}, the above term can be simplified as
		\begin{align*}
			&\Big[\m{U}_1\m{S}_1\big(\m{S}_1\m{S}_1+\m{V}_1\tr\m{M}\m{V}_1-\m{V}_1\tr\m{M}\m{V}_2(\m{V}_2\tr\m{M}\m{V}_2)^{-1}\m{V}_2\tr\m{M}\m{V}_1\big)^{-1}\m{S}_1\tr\m{U}_1\tr\Big]^{\dagger}\m{A}\\
			=&\Big\{\m{U}_1\big[\m{I}+\m{S}_1^{-1}(\m{V}_1\tr\m{M}\m{V}_1-\m{V}_1\tr\m{M}\m{V}_2(\m{V}_2\tr\m{M}\m{V}_2)^{-1}\m{V}_2\tr\m{M}\m{V}_1)\m{S}_1^{-1}\big]^{-1}\m{U}_1\tr\Big\}^{\dagger}\m{A}\\
			=&\Big[\m{U}_1\m{U}_1\tr+\m{U}_1\m{S}_1^{-1}\big(\m{V}_1\tr\m{M}\m{V}_1-\m{V}_1\tr\m{M}\m{V}_2(\m{V}_2\tr\m{M}\m{V}_2)^{-1}\m{V}_2\tr\m{M}\m{V}_1\big)\m{S}_1^{-1}\m{U}_1\tr\Big]\m{A}\\
			=&\m{A}+\m{U}_1\m{S}_1^{-1}\big[\m{V}_1\tr\m{M}\m{V}_1 -\m{V}_1\tr\m{M}\m{V}_2(\m{V}_2\tr\m{M}\m{V}_2)^{-1}\m{V}_2\tr\m{M}\m{V}_1\big]\m{V}_1.
		\end{align*}
		The right-hand side of \eqref{lem2.2 eq1} can be rewritten as
		\begin{align*}
			\left[\m{I}+(\m{U}_1\m{S}_1\m{V}_1\tr\m{M}^{-1}\m{V}_1\m{S}_1\tr\m{U}_1\tr)^{\dagger}\right]\m{A}
			=\left[\m{I}+\m{U}_1\m{S}_1^{-1}(\m{V}_1\tr\m{M}^{-1}\m{V}_1)^{-1}\m{S}_1^{-1}\m{U}_1\tr\right]\m{A}.
		\end{align*}
		By applying \textbf{Lemma \ref{lem2}}, the above term can be simplified as
		\begin{align*}\notag
			&\left[\m{I}+\m{U}_1\m{S}_1^{-1}(\m{V}_1\tr\m{M}^{-1}\m{V}_1)^{-1}\m{S}_1^{-1}\m{U}_1\tr\right]\m{A}\\
			=&\Big[\m{I}+\m{U}_1\m{S}_1^{-1}\big(\m{V}_1\tr\m{M}\m{V}_1-\m{V}_1\tr\m{M}\m{V}_2(\m{V}_2\tr\m{M}\m{V}_2)^{-1}\m{V}_2\tr\m{M}\m{V}_1\big)\m{S}_1^{-1}\m{U}_1\tr\Big]\m{A}\\
			=&\m{A}+\m{U}_1\m{S}_1^{-1}\big[\m{V}_1\tr\m{M}\m{V}_1  -\m{V}_1\tr\m{M}\m{V}_2(\m{V}_2\tr\m{M}\m{V}_2)^{-1}\m{V}_2\tr\m{M}\m{V}_1\big]\m{V}_1,
		\end{align*}
		which shows \eqref{lem2.2 eq1} holds. Taking a preconditioning manner $\m{B}=\m{N}^{1/2}\m{A}$ and using the result \eqref{lem2.2 eq1}, we easily have \eqref{lem2.2 eq3} holds.
	\end{proof}

	\begin{lemma}\label{lemma2.3}
		Let $\m{A} \in R^{m \times n}$ and $\operatorname{rank}(\m{A})=r$ where $r \le m \le n$. Let $\m{M} \in R^{n \times n} $ and $\m{N} \in R^{m \times m} $ be any symmetric positive definite matrices. Then
		\begin{equation}\label{lemma2.3 eq1}
			\lim\limits_{\epsilon \to 0} \m{A}(\epsilon \m{M}+\m{A}\tr\m{N}\m{A})^{-1}\m{A}\tr=\m{A}(\m{A}\tr\m{N}\m{A})^{\dagger}\m{A}\tr.
		\end{equation}
	\end{lemma}
	\begin{proof}
		
		Employ SVD on \m{A} as the proof of \textbf{Lemma \ref{lem3}}. The left-hand side of \eqref{lemma2.3 eq1} can be rewritten as
		\begin{align*}
			&\lim\limits_{\epsilon \to 0}\m{U}_1\m{S}_1\left(\m{S}_1\tr\m{U}_1\tr\m{N}\m{U}_1\m{S}_1+\epsilon\m{V}_1\tr\m{M}\m{V}_1\right)^{-1}\m{S}_1\tr\m{U}_1\tr\\
			=& \lim\limits_{\epsilon \to 0}\m{U}_1\left(\m{U}_1\tr\m{N}\m{U}_1+\epsilon\m{S}_1^{-1}\m{V}_1\tr\m{M}\m{V}_1\m{S}_1^{-1}\right)^{-1}\m{U}_1\tr\notag\\
			=&\m{U}_1\m{U}_1\tr\m{N}^{-1}\m{U}_1\m{U}_1\tr.\notag
		\end{align*}
		The right-hand side of \eqref{lemma2.3 eq1} can be rewritten as
		\begin{align*}
			\m{U}_1\m{S}_1\m{V}_1\tr\left(\m{V}_1\m{S}_1\tr\m{U}_1\tr\m{N}\m{U}_1\m{S}_1\m{V}_1\tr\right)^{\dagger}\m{V}_1\m{S}_1\tr\m{U}_1\tr
			=\m{U}_1\m{U}_1\tr\m{N}^{-1}\m{U}_1\m{U}_1\tr,\notag
		\end{align*}
		which shows \eqref{lemma2.3 eq1} holds.
	\end{proof}
	Based on \textbf{Lemma \ref{lem3}}, the iteration
	\begin{equation*}
		\bm{\lambda}^{t+1}=\bm{\lambda}^t+ \left(\m{G}^{t}\right)^{\dagger} (\m{I}-\m{W})^{1 / 2}\m{x}^{t+1}
	\end{equation*}
	can be equivalently written as
	\begin{equation}\label{du3}
		\bm{\lambda}^{t+1}=\bm{\lambda}^t+ \alpha (\m{I} - \m{W})^{1 / 2}\m{x}^{t+1}
		+\left( \tilde{\m{G}}^{t}\right)^{\dagger}(\m{I} - \m{W})^{1 / 2}\m{x}^{t+1},
	\end{equation}
	where
	\begin{equation*}
		\tilde{\m{G}}^{t}={(\m{I} - \m{W})^{1/2}} (\nabla^2 f(\m{x}^{t + 1}))^{-1} {(\m{I} - \m{W})^{1/2}}.
	\end{equation*}
	Without the last term, \eqref{du3} is just the dual ascent step of the standard 
	Augmented Lagrange Method (ALM)
	for solving \eqref{obj_fun4}. Hence, the last term of \eqref{du3} can be simply viewed as a
	second-order correction term for the dual update.
	Meanwhile, we also notice that the computation of $(\m{I} - \m{W})^{1/2}$ is not only expensive but also
	destroys the network structure. To overcome this undesirable computation, multiplying
	\eqref{du3} by $(\m{I}-\m{W})^{1 / 2}$ and denoting $\m{v} = (\m{I}-\m{W})^{1 / 2} \bm{\lambda}$,
	we obtain 
	\begin{align}\label{dvv}
		\m{v}^{t+1} = & \m{v}^t+\alpha (\m{I} - \m{W})\m{x}^{t+1}  
		 +(\m{I}-\m{W})^{1/2} \left( \tilde{\m{G}}^{t}\right)^{\dagger} (\m{I}-\m{W})^{1/2}\m{x}^{t+1} \\
		= & \m{v}^t+\alpha (\m{I} - \m{W})\m{x}^{t+1} + (\m{I}-\m{W}) \m{D}^{t} (\m{I}-\m{W}) \m{x}^{t+1}, \nonumber
	\end{align}
	where 
	\[
	\m{D}^t = \left((\m{I}-\m{W})^{1/2} \right)^{\dagger}\left( \tilde{\m{G}}^{t}\right)^{\dagger}
	\left((\m{I}-\m{W})^{1/2} \right)^{\dagger}.
	\]
	However, in practice, for computation efficiency and numerical stability, 
	we would approximate $\m{D}^t$ by $\tilde{\m{P}}^t\tilde{\m{D}}$, where  $\tilde{\m{P}}^t$  is a block diagonal approximation 
	of  $(\tilde{\m{G}}^{t}+r^t\m{I})^{-1}$ and $\tilde{\m{D}}$ is a block diagonal matrix whose
	$i$-th diagonal block $\tilde{\m{D}}_i = \frac{1}{1-\tilde{W}_{ii}}\m{I}_p$.
	Here $r^t >0$ is a scalar. \textbf{Lemma~\ref{lemma2.3}} implies that 
	as $r^t$ goes to zero,
	$(\m{I}-\m{W})^{1/2} (\tilde{\m{G}}^{t}+r^t\m{I})^{-1} (\m{I}-\m{W})^{1/2}$
	approaches $ (\m{I}-\m{W})^{1/2} ( \tilde{\m{G}}^{t})^{\dagger} (\m{I}-\m{W})^{1/2}$.
	Finally, by introducing a dual stepsize $\gamma >0$ for ensuring global convergence,
	based on \eqref{dvv}, we propose to update $\{\bm{\lambda}^t\}$ and $\{\m{v}^t\}$ in our new algorithm  as 
	\begin{equation}\label{du11}
		\m{v}^{t+1}= \m{v}^t+ \gamma (\m{I} - \m{W})\bm{\nu}^t \mbox{ and } 
		\bm{\lambda}^{t+1}=  \bm{\lambda}^t+ \gamma (\m{I} - \m{W})^{1/2} \bm{\nu}^t, 
	\end{equation}
	where $\bm{\nu}^t = \alpha \m{x}^{t+1}+\tilde{\m{P}}^t\tilde{\m{D}}(\m{I}-\m{W})\m{x}^{t+1} $.
	%

	We now explain our construction of $\tilde{\m{P}}^t$ by using BB \cite{barzilai1988two}
	and dynamic average consensus techniques
	to capture certain curvature information of  $(\tilde{\m{G}}^{t}+r^t\m{I})^{-1}$.
	First, by using again $\m{B}^{t+1}$ to approximate $\nabla^2 f(\m{x}^{t + 1})$,
	noticing $\m{v}^{t} - \m{v}^{t-1} = (\m{I} - \m{W})^{1/2} (\bm{\lambda}^{t} - \bm{\lambda}^{t-1})$ and \eqref{du11},
	let us define
	\begin{align*}
		\bm{\zeta}^{t-1} 
		:=&{(\m{I} - \m{W})^{1/2}} (\m{B}^{t})^{-1} {(\m{I} - \m{W})^{1/2}} (\bm{\lambda}^{t} - \bm{\lambda}^{t-1})\\
		=&{(\m{I} - \m{W})^{1/2}} (\m{B}^{t})^{-1} (\m{v}^{t} - \m{v}^{t-1})\\
		=&{(\m{I} - \m{W})^{1/2}} \bm{\mu} ^{t-1},
	\end{align*} 
	where $\bm{\mu} ^{t-1}=(\m{B}^{t})^{-1} (\m{v}^{t} - \m{v}^{t-1})$, and  define
	\[
	\bm{\xi} ^{t-1} :=\bm{\lambda}^{t} - \bm{\lambda}^{t-1}= \gamma (\m{I} - \m{W})^{1/2} \bm{\nu}^{t-1}.
	\]
	Then, the standard centralized BB technique suggests to approximate $ (\tilde{\m{G}}^t +r^t\m{I})^{-1}$
	by a scalar matrix $p^{t}\m{I}$ , where
	\begin{align}\label{p}
		p^{t} =\left(\frac{(\bm{\zeta}^{t-1})\tr\bm{\xi}^{t-1}}{(\bm{\xi}^{t-1})\tr\bm{\xi}^{t-1}}+r^t \right)^{-1}
		=\left(\frac{(\m{v}^{t} - \m{v}^{t-1})\tr\bm{\mu} ^{t-1}}{\gamma (\m{v}^{t} - \m{v}^{t-1})\tr \bm{\nu}^{t-1}}+r^t \right)^{-1} 
		=\left(\frac{\sum_{i=1}^n\tilde{b}_i^{t}}{\sum_{i=1}^n\tilde{a}_i^{t}}+r^t \right)^{-1} 
	\end{align}
	with $\tilde{b}_i^{t} =(\m{v}_i^{t} - \m{v}_i^{t-1})\tr \bm{\mu}_i ^{t-1}$ and 
	$\tilde{a}_i^{t} = \gamma (\m{v}_i^{t} - \m{v}_i^{t-1})\tr \bm{\nu}_i^{t-1}
	= \gamma (\m{v}_i^{t} - \m{v}_i^{t-1})\tr\left(\alpha\m{x}_i^t+\tilde{\m{P}}^{t-1}_{i}\tilde{\m{D}}_i (\m{x}_i^{t} - \sum_{j \in \mathcal{N}_i} {{\tilde{W}_{ij}}\m{x}_j^{t}}) \right)$.
	Obviously, the calculation of $p^t$ in \eqref{p} needs global information from all nodes, which can not 
	be realized in the decentralized setting.
	On the other hand, we also notice that the scalar ${\sum_{i=1}^n\tilde{b}_i^{t}}/{\sum_{i=1}^n\tilde{a}_i^{t}}$
	in \eqref{p} is in fact the ratio of the average values of $\tilde{a}_i^{t}$ and $\tilde{b}_i^{t}$ over $n$ nodes. Motivated by the idea of dynamic average consensus \cite{zhu2010discrete}, we set $\tilde{\m{P}}^{t}$ 
	to have the following block scalar matrix format
	\begin{equation}\label{P1}
		\tilde{\m{P}}^{t} = \left[ {\begin{array}{*{20}{c}}
				{\tilde{p}_1^{t} \m{I}_p}&{}&{}\\
				{}&{...}&{}\\
				{}&{}&{\tilde{p}_n^{t} \m{I}_p}
		\end{array}} \right],
	\end{equation} 
	where
	\begin{equation}\label{pi}
		\tilde{p}_i^{t}=\left\{\begin{array}{cl}
			\left(\operatorname{Proj}_{[\underline{\omega},\overline{\omega}]}\left( \frac{b_i^{t}}{a_i^{t}} \right)+r^t \right)^{-1}, & \text {if } a_i^t \neq 0, \\
			\left(\overline{\omega}+r^t \right)^{-1}, & \text {if } a_i^t=0 \text{ and } b_i^t>0, \\
			\left(\underline{\omega}+r^t \right)^{-1}, & \text {if } a_i^t=0 \text{ and } b_i^t \leq 0,
		\end{array}\right.
	\end{equation}
	with $\underline{\omega} >\overline{\omega}> 0$ and the ratio $\frac{0}{0}$ being defined as zero.
	Here, $a_i^{t}$ and $b_i^{t}$ are certain estimations of $\frac{1}{n} \sum_{i=1}^n\tilde{a}_i^{t}$ 
	and $\frac{1}{n} \sum_{i=1}^n\tilde{b}_i^{t}$, respectively, and can be calculated only by the 
	local information from neighboring nodes. More specifically, in our algorithm $a_i^{t}$ and $b_i^{t}$ 
	are calculated as follows:
	\begin{equation}\label{a}
		{a}^{t}_i=\sum_{j \in \mathcal{N}_i} \tilde{W}_{i j}{a}^{t-1}_j+\tilde{a}_i^{t}-\tilde{a}_i^{t-1},\text{ } a^0_i=\tilde{a}^0_i=1,
	\end{equation}
	\begin{equation}\label{b}
		{b}^{t}_i=\sum_{j \in \mathcal{N}_i} \tilde{W}_{i j}{b}^{t-1}_j+\tilde{b}_i^{t}-\tilde{b}_i^{t-1},\text{ } b^0_i=\tilde{b}^0_i=1.
	\end{equation}	
	With the above calculation, we have the following lemma.
	\begin{lemma}\label{lem_a_b}
		$\frac{1}{n} \sum_{i=1}^n{a}_i^{t} = \frac{1}{n} \sum_{i=1}^n\tilde{a}_i^{t}$, $\frac{1}{n} \sum_{i=1}^n{b}_i^{t} = \frac{1}{n} \sum_{i=1}^n\tilde{b}_i^{t}$.
	\end{lemma}
	\begin{proof}
		
		Denote 
		\begin{align*}
			\m{a}^t=[a_1^t; \ldots ;a_n^t], \quad\m{b}^t=[b_1^t; \ldots ;b_n^t],\\
			\tilde{\m{a}}^{t}=[\tilde{a}_1^{t};\ldots; \tilde{a}_n^{t}],\quad \tilde{\m{b}}^{t}=[\tilde{b}_1^{t};\ldots; \tilde{b}_n^{t}].
		\end{align*}
		Then, \eqref{a} and \eqref{b} can be written in the following matrix format 
		\begin{equation}\label{a1}
			\m{a}^{t}=\tilde{\m{W}}\m{a}^{t-1} +\tilde{\m{a}}^{t} - \tilde{\m{a}}^{t-1},
		\end{equation}
		\begin{equation}\label{b1}
			\m{b}^{t}=\tilde{\m{W}}\m{b}^{t-1} +\tilde{\m{b}}^{t} - \tilde{\m{b}}^{t-1}.
		\end{equation}
		Since $\tilde{\m{W}}$ is doubly stochastic, we have $\m{1}_n\tr=\m{1}_n\tr\tilde{\m{W}}$. Therefore, we have
		\begin{align*}
			\frac{1}{n} \sum_{i=1}^n{a}_i^{t} =\frac{1}{n}\m{1}_n\tr\m{a}^{t}
			= \frac{1}{n}\m{1}_n\tr\tilde{\m{W}}\m{a}^{t-1}+\frac{1}{n}\m{1}_n\tr\tilde{\m{a}}^{t}-\frac{1}{n}\m{1}_n\tr\tilde{\m{a}}^{t-1}
			=\frac{1}{n} \sum_{i=1}^n{a}_i^{t-1}+\frac{1}{n} \sum_{i=1}^n\tilde{a}_i^{t}-\frac{1}{n} \sum_{i=1}^n\tilde{a}_i^{t-1}.
		\end{align*}
		The above recursive formula gives $\frac{1}{n} \sum_{i=1}^n{a}_i^{t}=\frac{1}{n} \sum_{i=1}^n{a}_i^{0}+\frac{1}{n} \sum_{i=1}^n\tilde{a}_i^{t}-\frac{1}{n} \sum_{i=1}^n\tilde{a}_i^{0}$.
		Since ${a}_i^{0}=\tilde{a}_i^{0},i=1,\ldots,n$, we obtain $\frac{1}{n} \sum_{i=1}^n{a}_i^{t} = \frac{1}{n} \sum_{i=1}^n\tilde{a}_i^{t}$. 
		The identity $\frac{1}{n} \sum_{i=1}^n{b}_i^{t} = \frac{1}{n} \sum_{i=1}^n\tilde{b}_i^{t}$ can be similarly proved. 
	\end{proof}
	
	
	Noticing $\m{v}^t = (\m{I}-\m{W})^{1/2} \bm{\lambda}^t$, the primal updates in \eqref{pu11}
	can be also written as
	\begin{align}\label{pu2}
		\m{x}^{t+1}=\m{x}^{t}-\beta\left[\m{I}-\theta\alpha(\m{B}^t)^{-1}(\m{I}-\m{W})\right](\m{B}^t)^{-1}\nabla_{\m{x}}{{L}_\alpha }({\m{x}^t},{\m{v}^t}),
	\end{align}
	where 
	\begin{equation}\label{AL3}
		L_{\alpha}(\m{x},\m{v})= f(\m{x}) + \left\langle {\m{v},\m{x}} \right\rangle  + \frac{\alpha }{2}{\m{x}\tr}(\m{I} - \m{W})\m{x}.
	\end{equation}
	Then, summarizing the above discussion, our decentralized primal-dual method (DPDM)
	is given in \textbf{Algorithm \ref{alg:Framwork1}}.
	
	\begin{figure*}[!t]
		\centering
		\begin{minipage}{\textwidth}
			\begin{algorithm}[H]
				\caption{DPDM}
				\label{alg:Framwork1}
				\begin{algorithmic}[1]
					\Require
					$\m{x}^0$,  MaxIter, $\alpha$, $\beta$, $\gamma$, $\theta$, $\underline{\omega}$, $\overline{\omega}$, $\m{W}$, 
					$\{r^t\}_{t \geq 0}$.
					\State Set $t=0$, $T=\operatorname{MaxIter}$, $\m{v}^0 = \m{0}$, $\tilde{\m{P}}^0=\frac{1}{1+r^0}\m{I}$, $\m{H}^0=\m{I}$, $a^0_i=\tilde{a}^0_i=1$, $b^0_i=\tilde{b}^0_i=1$, $i=1, \ldots, n$.
					\State If $t \geq T$, stop.
					\State $\m{x}^{t+1}=\m{x}^{t}-\beta\left[\m{I}-\theta\alpha\m{H}^t(\m{I}-\m{W})\right]\m{H}^t\nabla_{\m{x}}{{L}_\alpha }({\m{x}^t},{\m{v}^t})$.
					\State
					$
					{\m{H}^{t+1}} = \left[ {\begin{array}{*{20}{c}}
							{{ \m{H}_1^{t + 1}}
							}&{}&{}\\
							{}&{...}&{}\\
							{}&{}&{\m{H}_n^{t+1}}
					\end{array}} \right]$, where
					$	\m{H}_i^{t+1}=\m{H}_i^{t}-\frac{\m{H}_i^{t} {\m{y}}_i^t (\m{s}_i^t)\tr+\m{s}_i^t ({\m{y}}_i^t)\tr \m{H}_i^{t}}{(\m{s}_i^t)\tr {\m{y}}_i^t}
					+\left(1+\frac{({\m{y}}_i^t)\tr \m{H}_i^{t} {\m{y}}_i^t}{(\m{s}_i^t)\tr {\m{y}}_i^t}\right)\frac{\m{s}_i^t(\m{s}_i^t)\tr}{(\m{s}_i^t)\tr {\m{y}}_i^t}. $
					\State
					If $t \geq 1$, $\tilde{\m{P}}^{t} = \left[ {\begin{array}{*{20}{c}}
							{\tilde{p}_1^{t} \m{I}_p}&{}&{}\\
							{}&{...}&{}\\
							{}&{}&{\tilde{p}_n^{t} \m{I}_p}
					\end{array}} \right]$, 
				where
					\begin{align*}
						\tilde{p}_i^{t} & = \left\{\begin{array}{cl}
							\left(\operatorname{Proj}_{[\underline{\omega},\overline{\omega}]}\left( \frac{b_i^{t}}{a_i^{t}} \right)+r^t \right)^{-1}, & \text { if } a_i^t \neq 0; \\
							\left(\overline{\omega}+r^t \right)^{-1}, & \text { if } a_i^t=0 \text{ and } b_i^t>0; \\
							\left(\underline{\omega}+r^t\right)^{-1}, & \text { if } a_i^t=0 \text{ and } b_i^t \leq 0,
						\end{array}\right.\\
						{a}^{t}_i &=  \sum_{j \in \mathcal{N}_i} \tilde{W}_{i j}{a}^{t-1}_j+\tilde{a}_i^{t}-\tilde{a}_i^{t-1}, \quad
						{b}^{t}_i=\sum_{j \in \mathcal{N}_i} \tilde{W}_{i j}{b}^{t-1}_j+\tilde{b}_i^{t}-\tilde{b}_i^{t-1}, \\
						\tilde{a}_i^{t} 
						&= \gamma (\m{v}_i^{t} - \m{v}_i^{t-1})\tr  \left(\alpha\m{x}_i^t+\tilde{\m{P}}^{t-1}_{i}\tilde{\m{D}}_i \m{u}_i^t \right), \quad \tilde{b}_i^{t} =(\m{v}_i^{t} - \m{v}_i^{t-1})\tr \m{H}_i^{t} (\m{v}_i^{t} - \m{v}_i^{t-1}).
					\end{align*}	
					\State $	\m{v}^{t+1}=\m{v}^t+\gamma (\m{I} - \m{W}) 
					\left( \alpha \m{x}^{t+1}+\tilde{\m{P}}^t \tilde{\m{D}} \m{u}^{t+1}\right)$, where $\m{u}^{t+1} = (\m{I}-\m{W})\m{x}^{t+1}$.
					\State Set $t=t+1$ and go to Step 2.
					\Ensure
					$\m{x}^T$.
				\end{algorithmic}
			\end{algorithm}
		\end{minipage}
	\end{figure*}

	\begin{remark}
		\text{~}
		\begin{itemize}
			\item \textbf{Algorithm \ref{alg:Framwork1}} can be run parallelly on each node, which only uses 
			the local information from neighboring nodes.
			The updates of $\m{B}^t$ and $\tilde{\m{P}}^t$ exploit second-order information
			and only involve matrix-vector products, which would be computationally efficient. 
			\item 
			At most $2+ I_{R\setminus \{0\}}(\theta)$  rounds of vector communications are taken in both the primal and dual updates of DPDM,
			which is comparable to or mildly higher than some other primal-dual quasi-Newton methods (see Table~\ref{comparison} for details).
			However, our numerical experiments show the reduction in the total number of iterations and the 
			CPU time by applying quasi-Newton techniques in DPDM would offset its slightly higher communication cost per iteration.     
			\item \textbf{Algorithm \ref{alg:Framwork1}} has the following relationships with
			some existing well-known decentralized algorithms. 
			Set the parameters $\beta=1$, $\gamma=1$ and $\theta=0$ in \textbf{Algorithm \ref{alg:Framwork1}}.
			If $\m{B}^{t}=\nabla^2 f(\m{x}^t)+\epsilon\m{I}$, $\tilde{\m{P}}^t=\m{0}$, 
			we obtain the algorithm NT \cite{zhang2021newton},
			while the algorithm PD-QN \cite{eisen2019primal} takes $\m{B}^{t}$ and $\tilde{\m{P}}^t$ as 
			some other special approximate matrices.
			On the other hand, if $\m{B}^{t}= 2 \alpha \m{I}$ and $\tilde{\m{P}}^t=\m{0}$, 
			\textbf{Algorithm 1} can be reduced as
			\begin{align*}
				{\m{x}^{t + 1}} = (\m{I} + \m{W}){\m{x}^t} - \frac{{\m{I} + \m{W}}}{2}{\m{x}^{t - 1}} - \frac{1}{{2\alpha }} \nabla f({\m{x}^t}) 
				+ \frac{1}{{2\alpha }} \nabla f({\m{x}^{t - 1}})
			\end{align*}
			which is equivalent to EXTRA \cite{shi2015extra} with the stepsize ${1/ {(2\alpha )}}$.
			\item A single quasi-Newton primal steps is performed in \textbf{Algorithm \ref{alg:Framwork1}}.
			Allowing  multiple primal updates per iteration will yield a more accurate approximate solution 
			$\m{x}^{t+1}$ of minimizing $L_{\alpha}({\m{x}},{\m{v}^t})$, which would also benefit the dual update. 
			By this consideration, a generalized decentralized primal-dual method(GDPDM) is given as \textbf{Algorithm \ref{alg:Framwork2}}, where $S \ge 1$ is the number of primal steps per iteration.
			
			\begin{algorithm}[H]
				\caption{GDPDM}
				\label{alg:Framwork2}
				\begin{algorithmic}[1]
					\Require
					$\m{x}^0$, MaxIter1, MaxIter2, $\alpha$, $\beta$, $\gamma$, $\theta$, $\underline{\omega}$, $\overline{\omega}$, $\m{W}$,
					$\{r^t\}_{t \ge 0}$.
					\State Set $t=0$, $T=$MaxIter1, $S=$MaxIter2,  $\m{v}^0 = \m{0}$, $\tilde{\m{P}}^0=\frac{1}{1+r^0}\m{I}$, $\m{H}^0=\m{I}$, $a^0_i=\tilde{a}^0_i=1$, $b^0_i=\tilde{b}^0_i=1$, $i=1, \ldots, n$.
					\State If $t \geq T$, stop.
					\State $\m{x}^{t,0}=\m{x}^{t}$, $\m{H}^{t,0}=\m{H}^{t}$.
					\State For $s=0,1,\ldots,S-1$
					\State
					\quad $\m{x}^{t,s+1}=\m{x}^{t,s}-\beta\left[\m{I}-\theta\alpha\m{H}^{t,s}(\m{I}-\m{W})\right]\m{H}^{t,s}\nabla_{\m{x}}{{L}_\alpha }({\m{x}^{t,s}},{\m{v}^t})$.
					\State \quad
					$
					{\m{H}^{t,s+1}} = \left[ {\begin{array}{*{20}{c}}
							{{\m{H}_1^{t,s+1}}}&{}&{}\\
							{}&{...}&{}\\
							{}&{}&{{\m{H}_n^{t,s+1}}}
					\end{array}} \right]$, where\\
					\begin{align*}
						\m{H}_i^{t+1,s} &= \m{H}_i^{t,s}-\frac{\m{H}_i^{t,s} {\m{y}}_i^{t,s} (\m{s}_i^{t,s})\tr+\m{s}_i^{t,s}({\m{y}}_i^{t,s})\tr \m{H}_i^{t,s}}{(\m{s}_i^{t,s})\tr {\m{y}}_i^{t,s}}
						+\left(1+\frac{({\m{y}}_i^{t,s})\tr \m{H}_i^{t,s} {\m{y}}_i^{t,s}}{(\m{s}_i^{t,s})\tr {\m{y}}_i^{t,s}}\right)\frac{ \m{s}_i^{t,s}(\m{s}_i^{t,s})\tr}{(\m{s}_i^{t,s})\tr {\m{y}}_i^{t,s}}, \\
						\m{s}_i^{t,s} & = \m{x}_i^{t,s+1} - \m{x}_i^{t,s} \mbox{ and } 
						\m{y}_i^{t,s} = \nabla {f_i}( \m{x}_i^{t, s+1} ) - \nabla {f_i}( \m{x}_i^{t,s} ).
					\end{align*}	
					\State End for.
					\State $\m{x}^{t+1}=\m{x}^{t, S}, \m{H}^{t+1}=\m{H}^{t, S}$.
					\State
					If $t \geq 1$, $\tilde{\m{P}}^{t} = \left[ {\begin{array}{*{20}{c}}
							{\tilde{p}_1^{t} \m{I}_p}&{}&{}\\
							{}&{...}&{}\\
							{}&{}&{\tilde{p}_n^{t} \m{I}_p}
					\end{array}} \right]$, where
					$	\tilde{p}_i^{t}  = \left\{\begin{array}{cl}
						\left(\operatorname{Proj}_{[\underline{\omega},\overline{\omega}]}\left( \frac{b_i^{t}}{a_i^{t}} \right)+r^t \right)^{-1}, & \text { if } a_i^t \neq 0; \\
						\left(\overline{\omega}+r^t \right)^{-1}, & \text { if } a_i^t=0 \text{ and } b_i^t>0; \\
						\left(\underline{\omega}+r^t\right)^{-1}, & \text { if } a_i^t=0 \text{ and } b_i^t \leq 0,
					\end{array}\right. $
					\begin{align*}
						{a}^{t}_i & =  \sum_{j \in \mathcal{N}_i} \tilde{W}_{i j}{a}^{t-1}_j+\tilde{a}_i^{t}-\tilde{a}_i^{t-1}, \quad
						{b}^{t}_i=\sum_{j \in \mathcal{N}_i} \tilde{W}_{i j}{b}^{t-1}_j+\tilde{b}_i^{t}-\tilde{b}_i^{t-1}, \\
						\tilde{a}_i^{t} &
						= \gamma (\m{v}_i^{t} - \m{v}_i^{t-1})\tr  \left(\alpha\m{x}_i^t+\tilde{\m{P}}^{t-1}_{i}\tilde{\m{D}}_i \m{u}_i^t \right), \quad \tilde{b}_i^{t} =(\m{v}_i^{t} - \m{v}_i^{t-1})\tr \m{H}_i^{t} (\m{v}_i^{t} - \m{v}_i^{t-1}).
					\end{align*}		
						\State  $	\m{v}^{t+1}=\m{v}^t+\gamma (\m{I} - \m{W}) 
						\left( \alpha \m{x}^{t+1}+\tilde{\m{P}}^t \tilde{\m{D}} \m{u}^{t+1} \right)$, where $ \m{u}^{t+1} = (\m{I}-\m{W})\m{x}^{t+1}$.
						\State Set $t=t+1$ and go to Step 2.
						\Ensure
						$\m{x}^T$.
					\end{algorithmic}
				\end{algorithm}
		\end{itemize}
	\end{remark}

	\paragraph{Quasi-Newton tracking}
	As Newton Tracking \cite{zhang2021newton}, we show that DPDM also
	maintains a quasi-Newton tracking property, implying that the updating direction of
	each $\m{x}_i^t$ is a local quasi-Newton direction and will be a global quasi-Newton direction
	when $\left\{ \m{x}_i^t \right\}_{i = 1}^n $ are getting consensus. 
	
	First, by \eqref{pu2}, we have
	\begin{align*}
		&{\m{B}^t}{\m{x}^{t + 1}} - {\m{B}^t}{\m{x}^t} \\
		= & - \beta\left[ {\nabla f\left( {{\m{x}^t}} \right) + {\m{v}^t} + \alpha (\m{I} - \m{W}){\m{x}^t}} \right]
		+\beta\theta\alpha(\m{I}-\m{W})(\m{B}^t)^{-1}\left[ {\nabla f\left( {{\m{x}^t}} \right) + {\m{v}^t} + \alpha (\m{I} - \m{W}){\m{x}^t}} \right]\\
		=&- \beta\left[ {\nabla f\left( {{\m{x}^t}} \right) + {\m{v}^t} + \alpha (\m{I} - \m{W}){\m{x}^t}} \right]+\beta(\m{I} - \m{W})\m{h}^t,
	\end{align*}
	where 
	\begin{equation*}
		\m{h}^t=\theta\alpha(\m{B}^t)^{-1}\left[ {\nabla f\left( {{\m{x}^t}} \right) + {\m{v}^t} + \alpha (\m{I} - \m{W}){\m{x}^t}} \right].
	\end{equation*}
	Hence, we have 
	\begin{equation}\label{2.25}
		{\m{x}^{t + 1}} = {\m{x}^t} - \beta \m{d}^t,
	\end{equation}
	where $ \m{d}^t = ({\m{B}^t})^{-1} \m{g}^t$ and 
	\begin{equation}\label{2.26}
		\m{g}^t = \nabla f\left( {{\m{x}^t}} \right) + {\m{v}^t} + (\m{I} - \m{W}) ( \alpha \m{x}^t + \m{h}^t).
	\end{equation}

	We now show the quasi-Newton tracking by \eqref{2.25}, \eqref{2.26}, and \eqref{B}. Firstly, $\m{B}_i^{t+1}$ given by \eqref{Bit} is actually the solution of the following optimization problem
	\begin{align*}
		\m{B}_i^{t + 1} = \arg &\mathop {\min}\limits_\m{Z} {\mathop{\rm tr}\nolimits} \left[ {{{\left( {\m{B}_i^t} \right)}^{ - 1}}\m{Z}} \right] - \log \det \left[ {{{\left( {\m{B}_i^t} \right)}^{ - 1}}\m{Z}} \right],\\
		&s.t.\quad \m{Z}\m{s}_i^t = \m{y}_i^t,\quad \m{Z} \succeq 0,
	\end{align*}
	where $\m{s}_i^t = \m{x}_i^{t + 1} - \m{x}_i^t$, $\m{y}_i^t = \nabla {f_i}( {\m{x}_i^{t + 1}} ) - \nabla {f_i}( {\m{x}_i^t} )$.
	Hence, we have
	\begin{equation}\label{2.28}
		\m{B}_i^{t + 1}\m{s}_i^t = \m{y}_i^t.
	\end{equation}
	Summing \eqref{2.28} over $i=1, \ldots , n$, we obtain
	\begin{equation}\label{2.29}
		\sum\limits_{i = 1}^n {\m{B}_i^{t + 1}(\m{x}_i^{t + 1} - \m{x}_i^t)}  = \sum\limits_{i = 1}^n {\left( {\nabla {f_i}( {\m{x}_i^{t + 1}} ) - \nabla {f_i}( {\m{x}_i^t} )} \right)}.
	\end{equation}
	When $\left\{ \m{x}_i^t \right\}_{i = 1}^n $ are consensus for large $t$, 
	i.e., $\m{x}_i^t = \bar{\m{x}}^t$ for all $i=1,\ldots, n$ and large $t$,
	\eqref{2.29} implies $\sum_{i=1}^n\m{B}_i^{t+1}$ will satisfy
	the global quasi-Newton secant equation of the original problem \eqref{obj_fun1},
	i.e.,
	\[
	\sum_{i=1}^n\m{B}_i^{t+1}(\bar{\m{x}}^{t+1} - \bar{\m{x}}^t) =
	\nabla F(\bar{\m{x}}^{t+1}) - \nabla F(\bar{\m{x}}^t),
	\]
	where $F(\m{z}) = \sum_{i=1}^n f_i(\m{z})$.
	On the other hand, by \eqref{2.26} and \eqref{du11}, we have 
	\begin{align}\label{2.31}
		& \m{B}^{t + 1} \m{d}^{t + 1} - \m{B}^t \m{d}^t = \m{g}^{t+1} - \m{g}^t \\
		= &  \nabla f \left( \m{x}^{t+1} \right) - \nabla f\left( \m{x}^t \right) + \m{v}^{t+1} - \m{v}^t  
		 + (\m{I} - \m{W})\left[ \alpha (\m{x}^{t + 1} - \m{x}^t) +\m{h}^{t+1}-\m{h}^t\right]\notag\\
		=&  \nabla f\left(\m{x}^{t + 1}\right) - \nabla f \left(\m{x}^t \right)+(\m{I} - \m{W}) \tilde{\m{g}}^t,\notag
	\end{align}
	where $\tilde{\m{g}}^t= \alpha (\m{x}^{t + 1} -\m{x}^t) + \gamma \bm{\nu}^t +\m{h}^{t+1}-\m{h}^t$. The $i$-th block of \eqref{2.31} can be written as
	\begin{align}\label{2.32}
		&{\m{B}_i^{t + 1}}  \m{d}_i^{t + 1} -  {\m{B}_i^{t}}  \m{d}_i^{t} 
		= \nabla {f_i}\left( {\m{x}_i^{t + 1}} \right) - \nabla {f_i}\left( {\m{x}_i^t} \right) 
		+\left( {\tilde{\m{g}}_i^t - \sum\limits_{j \in {\cal N}_i} {\tilde{W}_{ij}} \tilde{\m{g}}_j^t} \right).
	\end{align}
	Summing \eqref{2.32} over $i=1, \ldots , n$ and using $ \tilde{\m{W}}\tr \m{1} = \m{1} $ yields
	\begin{equation}\label{2.33}
		\sum\limits_{i = 1}^n {\m{B}_i^{t + 1} \m{d}_i^{t + 1}}  = \sum\limits_{i = 1}^n {\m{B}_i^{t} \m{d}_i^t}  
		+ \sum\limits_{i = 1}^n {\left( {\nabla {f_i}( {\m{x}_i^{t + 1}} ) - \nabla {f_i}( {\m{x}_i^t} )} \right)} .
	\end{equation}
	Since $\m{v}^0 = \m{0}$ in DPDM, we have from \eqref{2.26} that 
	\[
	\sum_{i=1}^n {\m{B}_i^{0}\m{d}_i^0} = \sum_{i=1}^n \m{g}_i^0 = 
	\sum_{i=1}^n \nabla f_i \left( \m{x}_i^0 \right).
	\]
	So, summing \eqref{2.33} over $t$, we have
	\begin{equation}\label{2.34}
		\sum\limits_{i = 1}^n {\m{B}_i^{t} \m{d}_i^t}  = \sum\limits_{i = 1}^n {\nabla {f_i}\left( {\m{x}_i^t} \right)}.
	\end{equation}
	Then, when $\left\{ \m{x}_i^t \right\}_{i = 1}^n $ are getting consensus for large $t$, 
	we have from $\m{d}_i^t = (\m{x}_i^t -\m{x}_i^{t+1})/\beta$ that
	$\left\{ \m{d}_i^t \right\}_{i = 1}^n $ are also getting consensus for large $t$,
	which implies $\frac{1}{n}\sum_{i=1}^n\m{d}_i^t =: \bar{\m{d}}^t = \m{d}_i^t$
	for all $i=1, \ldots, n$ and large $t$.
	Hence, by \eqref{2.34}, we have
	\begin{equation}\label{2.35}
		\sum\limits_{i = 1}^n {\m{B}_i^{t}}{\bar{\m{d}}^t} = \sum\limits_{i = 1}^n {\nabla {f_i}\left( {{{\bar{ \m{x}}}^t}} \right)},
	\end{equation}
	which together with \eqref{2.29} shows $\bar{\m{d}}^t$ is a a global quasi-Newton direction
	of original problem \eqref{obj_fun1}.
	
	\section{CONVERGENCE ANALYSIS}
	In this section, we will analyze the global convergence of \textbf{Algorithm \ref{alg:Framwork2}}.
	Note that \textbf{Algorithm \ref{alg:Framwork1}} is a special case of \textbf{Algorithm \ref{alg:Framwork2}}
	with $S=1$. For convenience, let us denote $\m{Z}=\m{I}-\m{W}$, $\rho=\bm{\lambda}_{\max}(\m{Z})$,
	$\sigma=\bm{\lambda}_{2}(\m{Z})$ and
	\begin{align*}
		\tilde{\m{B}}^{t,s}&=\left[\m{I}-\theta\alpha(\m{B}^{t,s})^{-1}\m{Z}\right] \left(\m{B}^{t,s}\right)^{-1}.
	\end{align*}
	Then, \textbf{Algorithm \ref{alg:Framwork2}} gives
	\begin{equation}\label{x-update-2}
		\m{x}^{t,s+1}=\m{x}^{t,s}-\beta \tilde{\m{B}}^{t,s} \nabla_{\m{x}}{{L}_\alpha }({\m{x}^{t,s}},{\m{v}^t}).
	\end{equation}
	By \textbf{Assumption \ref{as1}}, \textbf{\ref{as2}}, we can easily get 
	\begin{equation}\label{L_bound}
		\mu \m{I} \preceq \nabla_{\m{x} \m{x}}L_{\alpha}(\m{x},\m{v}) 
		= \nabla_{\m{x} \m{x}} \tilde{L}_{\alpha}(\m{x},\bm{\lambda})  \preceq \C{L} \m{I},
		\; \forall \m{x} \in {R^{np}},
	\end{equation}
	where $\C{L}=L+\rho\alpha$. 
	We also make the following assumption.
	\begin{assum}\label{as3}
		The approximate matrices $\{\m{B}^{t,s}\}$ satisfy
		\begin{equation*}
			\psi \m{I} \preceq \left(\m{B}^{t,s}\right)^{-1} \preceq \bar{\psi} \m{I},
		\end{equation*}
		for any $t$ and $s$, where $\bar{\psi} > \psi >0$.
		And the parameter $\theta$ is chosen such that $0 \leq \theta < \min \{1, \psi/(\alpha \rho) \}$ 
	\end{assum}
	From \textbf{Assumption \ref{as3}}, we can easily derive
	\begin{equation*}
		\psi\m{I}
		\preceq	\tilde{\m{B}}^{t,s} \preceq \Psi \m{I}, \mbox{ where } \Psi=
		\frac{\bar{\psi}\psi}{\psi-\theta\alpha\rho}.
	\end{equation*} 
	Let $\m{P}^t=\alpha \m{I}+(\m{I}-\m{W})^{1/2}\tilde{\m{P}}^t \tilde{\m{D}}(\m{I}-\m{W})^{1/2}$.
	From \eqref{pi}, we can also derive
	\begin{equation*}
		\alpha \m{I}
		\preceq	\m{P}^t \preceq \bar{\alpha}  \m{I}, \mbox{ where } \bar{\alpha} =
		\alpha+\frac{\rho}{\underline{\omega}(1-\max_i\{\tilde{W}_{ii}\})}.
	\end{equation*}
	In addition, by \eqref{du11}, we have
	\begin{equation}\label{du5}
		\m{v}^{t+1}=\m{v}^t+\gamma (\m{I} - \m{W})^{1/2}\m{P}^t(\m{I} - \m{W})^{1/2}\m{x}^{t+1},
	\end{equation}
	and
	\begin{equation}\label{du6}
		\bm{\lambda}^{t+1}=\bm{\lambda}^t+\gamma\m{P}^t(\m{I} - \m{W})^{1/2}\m{x}^{t+1}.
	\end{equation}
	In the following, we define $\beta_1$ and $\beta_2$ be lower and upper bounds of eigenvalues of 
	$\beta \tilde{\m{B}}^{t,s}$:
	\begin{align}\label{bound1}
		\beta_1 = \beta \psi \quad \mbox{and} \quad \beta_2 = \beta \Psi,
	\end{align}
	and let $\gamma_1$ and $\gamma_2$ the lower and upper bounds to eigenvalues of $\gamma\m{P}^t$:
	\begin{align}\label{bound2}
		\gamma_1 = \gamma \alpha \quad \mbox{and} \quad
		\gamma_2 =  \gamma \bar{\alpha}.
	\end{align}
	
	By \textbf{Assumption \ref{as2}} and Slater's condition, which holds due to linear constraints, 
	strong duality holds for the problem \eqref{obj_fun4}. Hence, a dual equivalent problem
	to the problem \eqref{obj_fun4} as well as the problem \eqref{obj_fun3} is 
	\begin{align}\label{obj_fun5}
		\mathop {\max }\limits_{\bm{\lambda} \in {R^{np}}} g(\bm{\lambda}),
	\end{align}
	where $g(\bm{\lambda})={\min}_{\m{x} \in {R^{np}}}\tilde{L}_{\alpha}(\m{x},\bm{\lambda})$ is the dual function.
	In addition, denoting  $ \tilde{f}(\cdot) := f(\cdot)+(\alpha/2)\|\cdot\|_{\m{Z}}$, we have
	\begin{align}\label{conjugate-f}
		g(\bm{\lambda}) = \min_{\m{x}}\{f(\m{x})+\bm{\lambda}\tr\m{Z}^{1/2}\m{x}+(\alpha/2)\m{x}\tr\m{Z}\m{x}\}
		=  -\tilde{f}^*(-\m{Z}^{1/2}\bm{\lambda}),
	\end{align}
	where $\tilde{f}^*$ is the conjugate function of $\tilde{f}$.
	Note that $\tilde{f}^*$ is $1/\C{L}$-strongly convex.
	Let $-\m{v}^*$ be the unique minimizer of $\tilde{f}^*$ and
	$\Lambda^*$ be the dual optimal solution set of \eqref{obj_fun5}.
	Then, it follows from \eqref{conjugate-f} 
	that 
	\[
	\Lambda^*=\{\bm{\lambda}^*:\m{Z}^{1/2}\bm{\lambda}=\m{v}^*\}
	\]
	and for any $\bm{\lambda}^* \in \Lambda^*$, we have 
	\[
	g(\bm{\lambda}^*)=-\tilde{f}^*(-\m{Z}^{1/2}\bm{\lambda}^*)=-\tilde{f}^*(-\m{v}^*).
	\]
	Furthermore, by \eqref{obj_fun5} and the definition of $\m{x}^*(\bm{\lambda})$ in \eqref{pu0}, we have $g(\bm{\lambda})=\tilde{L}_{\alpha}(\m{x}^*(\bm{\lambda}),\bm{\lambda})$, 
	\begin{equation}\label{grad_g}
		\nabla g(\bm{\lambda})=\m{Z}^{1/2}\m{x}^*(\bm{\lambda}),
	\end{equation}
	and  $	\m{x}^*(\bm{\lambda}^*)=\m{x}^*$ by the strong duality.
	The following lemma reveals some important properties of the function $g(\cdot)$.
	\begin{lemma}\label{lem3.1}
		Under \textbf{Assumptions \ref{as1}}, \textbf{\ref{as2}} and \textbf{\ref{as3}}, the dual function 
		$g(\cdot)$ is $L_g$-Lipschitz smooth and for any  ${\bm{\lambda}}^* \in \Lambda^*$,
		the PL inequality holds
		\begin{equation*}
			g\left(\bm{\lambda}^*\right)-g(\bm{\lambda}^t) \leq \frac{1}{2 \mu_g}\|\nabla g(\bm{\lambda}^t)\|^2
		\end{equation*}
		for all $t \ge0$, where $\mu_g=\sigma/\C{L}$ and $L_g=\rho/\mu$.
	\end{lemma}
	\begin{proof}
		By taking derivatives with respect to $\bm{\lambda}$ on both sides of $\nabla_{\m{x}}\tilde{L}_{\alpha}(\m{x}^*(\bm{\lambda}),\bm{\lambda})=\m{0}$, we obtain 
		\begin{equation*}
			\nabla_{\m{x}\m{x}}\tilde{L}_{\alpha}(\m{x}^*(\bm{\lambda}),\bm{\lambda})\nabla\m{x}^*(\bm{\lambda})+\nabla_{\bm{\lambda} \m{x}}\tilde{L}_{\alpha}(\m{x}^*(\bm{\lambda}),\bm{\lambda})=\m{0}.
		\end{equation*}
		Note that $\nabla_{\bm{\lambda} \m{x}}\tilde{L}_{\alpha}(\m{x}^*(\bm{\lambda}),\bm{\lambda})=\m{Z}^{1/2}$. Then we have 
		\begin{equation*}
			\nabla\m{x}^*(\bm{\lambda})=-\left[\nabla_{\m{x}\m{x}}\tilde{L}_{\alpha}(\m{x}^*(\bm{\lambda}),\bm{\lambda})\right]^{-1}\m{Z}^{1/2}.
		\end{equation*}
		So, for any $\bm{\lambda}_1, \bm{\lambda}_2 \in R^{np}$, by mean value theorem, we have 
		\begin{align}\label{bound-xxxx}
			\m{x}^*(\bm{\lambda}_1)-\m{x}^*(\bm{\lambda}_2)
			=\nabla\m{x}^*(\tilde{\bm{\lambda}})(\bm{\lambda}_1-\bm{\lambda}_2) 
			=-\left[\nabla_{\m{x}\m{x}}\tilde{L}_{\alpha}(\m{x}^*(\tilde{\bm{\lambda}}),\tilde{\bm{\lambda}})\right]^{-1}\m{Z}^{1/2}(\bm{\lambda}_1-\bm{\lambda}_2),
		\end{align}	
		where  $\tilde{\bm{\lambda}}=(1-\eta)\bm{\lambda}_1+\eta\bm{\lambda}_2$	for some $\eta \in [0,1]$.
		Since $ \nabla g(\bm{\lambda})=\m{Z}^{1/2}\m{x}^*(\bm{\lambda})$ by \eqref{grad_g},
		we have
		\begin{align}\label{diff-grad-g}
			\left\|\nabla g(\bm{\lambda}_1)-\nabla g(\bm{\lambda}_2)\right\|
			=\left\|\m{Z}^{1/2}\left[\nabla_{\m{x}\m{x}}\tilde{L}_{\alpha}(\m{x}^*(\tilde{\bm{\lambda}}),\tilde{\bm{\lambda}})\right]^{-1}\m{Z}^{1/2}(\bm{\lambda}_1-\bm{\lambda}_2)\right\| .
		\end{align}
		Then, by $\mu$-strongly convexity of $\tilde{L}_{\alpha}(\cdot,\tilde{\bm{\lambda}})$,
		we have 
		\[
		\left\|\nabla g(\bm{\lambda}_1)-\nabla g(\bm{\lambda}_2)\right\| \le 
		L_g\left\| \bm{\lambda}_1-\bm{\lambda}_2\right\|,
		\]
		where $L_g=\rho/\mu$ and $\rho = \bm{\lambda}_{\max}(\m{Z})$.
		Hence, $g$ is $L_g$-Lipschitz smooth.
		In the following, consider any $\bm{\lambda}_1, \bm{\lambda}_2 \in R^{np}$
		such that $\bm{\lambda}_1 - \bm{\lambda}_2$ belongs to the column space of  $\m{Z}^{1/2}$,
		denoted as $\C{S}$.
		Then, by \eqref{diff-grad-g}, we have 
		\begin{equation}\label{zzz-convex}
			\left\|\nabla g(\bm{\lambda}_1)-\nabla g(\bm{\lambda}_2)\right\| \ge 
			\mu_g\left\| \bm{\lambda}_1-\bm{\lambda}_2\right\|,
		\end{equation}
		where $\mu_g=\sigma/\C{L}=\sigma/(L+\rho\alpha)$ and $\sigma = \bm{\lambda}_2(\m{Z})$,
		which implies that for any $\bm{\lambda} \in R^{np}$
		such that $\bm{\lambda} - \bm{\lambda}_2$ belongs to  $\C{S}$, we have
		\begin{equation}\label{ggg-convex}
			g(\bm{\lambda}) \ge g(\bm{\lambda}_2) + (\nabla g(\bm{\lambda}_2)) \tr (\bm{\lambda} - \bm{\lambda}_2)
			+ \frac{\mu_g}{2} \|\bm{\lambda} - \bm{\lambda}_2\|^2.
		\end{equation}
		Now let $\phi(\bm{\lambda}) =  
		g(\bm{\lambda}) -  (\nabla g(\bm{\lambda}_1)) \tr (\bm{\lambda} - \bm{\lambda}_1)$.
		Then, we have from $\nabla \phi(\bm{\lambda}_1) = \m{0}$ and \eqref{ggg-convex} that 
		\begin{align}\label{1234} 
			&\phi(\bm{\lambda}_1)  =  \min_{\bm{\lambda} \in \C{S}} \phi(\bm{\lambda})  
			\ge  \min_{\bm{\lambda} \in \C{S}} \left\{ \phi(\bm{\lambda}_2) + (\nabla \phi(\bm{\lambda}_2)) \tr 
			(\bm{\lambda} - \bm{\lambda}_2) + \frac{\mu_g}{2} \|\bm{\lambda} - \bm{\lambda}_2\|^2 \right) \\
			=&  \phi(\bm{\lambda}_2) - \frac{1}{2 \mu_g} \|\nabla g(\bm{\lambda}_2)\|^2,\notag
		\end{align}
		where the last equality follows by \eqref{grad_g} and 
		\[
		\bm{\lambda} = \bm{\lambda}_2
		- \frac{1}{\mu_g} \nabla g(\bm{\lambda}_2) = \bm{\lambda}_2 - 
		\frac{1}{\mu_g} \m{Z}^{1/2} \m{x}^*(\bm{\lambda}_2) \in \C{S}
		\]
		being the minimizer of the optimization in  \eqref{1234}.
		Simplifying \eqref{1234} gives
		\begin{align}\label{5678}
			g(\bm{\lambda}_2) -  g(\bm{\lambda}_1)  
			\le (\nabla g(\bm{\lambda}_1)) \tr (\bm{\lambda}_2 - \bm{\lambda}_1)
			+ \frac{1}{2\mu_g} \|\nabla g(\bm{\lambda}_1) - \nabla g(\bm{\lambda}_2)\|^2 
		\end{align}
		By \eqref{du11} with setting $\bm{\lambda}^0 = \m{0}$, we have $\bm{\lambda}^t$ belongs 
		$\C{S}$ for all $t \ge 0$. In addition, it is easy to see there exists an optimal
		$\bm{\lambda}^* \in \Lambda^*$ that belongs to $\C{S}$.
		So, we have $\bm{\lambda}^t - \bm{\lambda}^*$ belongs to $\C{S}$.
		Then, it follows from \eqref{5678} with $\bm{\lambda}_1 = \bm{\lambda}^t$ and 
		$\bm{\lambda}_2 = \bm{\lambda}^*$, and  $\nabla g(\bm{\lambda}^*) = \m{0}$  that
		\begin{equation*}
			g(\bm{\lambda}^*)-g(\bm{\lambda}^t) \leq \frac{1}{2 \mu_g}\|\nabla g(\bm{\lambda}^t)\|^2
		\end{equation*}
		for all $t \ge 0$. Since $g(\bm{\lambda}^*)$ is a constant for any  $\bm{\lambda}^* \in \Lambda^*$,
		the above inequality holds for any  $\bm{\lambda}^* \in \Lambda^*$.
	\end{proof}
	To establish a key recursive relation in our convergence analysis,
	we define the following quantities:   
	\begin{align*}
		& \Delta_{\bm{\lambda}}^t=g\left({\bm{\lambda}}^*\right)-g\left({\bm{\lambda}}^t\right), \\
		& \Delta_{\m{x}}^t=\tilde{L}_{\alpha}\left(\m{x}^t, {\bm{\lambda}}^t\right)-\tilde{L}_{\alpha}\left(\m{x}^*\left({\bm{\lambda}}^t\right), {\bm{\lambda}}^t\right) \mbox{ and }\\
		& \Delta_{\m{x}}^{t,s}=\tilde{L}_{\alpha}\left(\m{x}^{t,s}, {\bm{\lambda}}^{t}\right)-\tilde{L}_{\alpha}\left(\m{x}^*(\bm{\lambda}^{t}), {\bm{\lambda}}^{t}\right),\\
		& \mbox{ for } s=0,1,\ldots,S-1,
	\end{align*}
	where ${\bm{\lambda}}^*$ can be an optimal dual solution in  $\Lambda^*$ and $\Delta_{\bm{\lambda}}^t$ is the dual optimality gap.
	We define a potential function by combing the performance metrics $\Delta_{\bm{\lambda}}^t$ and $\Delta_{\m{x}}^t$ as
	\begin{equation}\label{merit}
		\Delta^t=7\Delta_{\bm{\lambda}}^t+\Delta_{\m{x}}^t.
	\end{equation}
	Then, the main idea for establishing global convergence is to bound $\Delta_{\bm{\lambda}}^t$ and $\Delta_{\m{x}}^t$ by coupled inequalities,
	and then integrate these inequalities to show the potential function $\Delta^t$ decays at a linear rate. 
	We first provide two necessary lemmas for deriving the bounds of $\Delta_{\bm{\lambda}}^t$ and $\Delta_{\m{x}}^t$.
	\begin{lemma}\label{lem3.2}
		Under \textbf{Assumption \ref{as2}}, the iterates generated from \textbf{Algorithm \ref{alg:Framwork2}} satisfy
		\begin{equation}\label{lem3.2 eq1}
			\left\|\m{Z}^{1/2} (\m{x}^{t+1}- \m{x}^*(\bm{\lambda}^t))\right\| \leq \frac{\rho^{1/2}}{\mu}\left\|\nabla_{\m{x}} \tilde{L}_{\alpha}\left(\m{x}^{t+1}, \bm{\lambda}^t\right)\right\| .
		\end{equation}
	\end{lemma}
	\begin{proof}
		
		By the definition of $\rho = \lambda_{\max}(\m{Z})$, we have
		\begin{equation}\label{lem3.2 eq2}
			\left\|\m{Z}^{1/2} (\m{x}^{t+1}- \m{x}^*(\bm{\lambda}^t))\right\| \leq \rho^{1/2} \left\|\m{x}^{t+1}- \m{x}^*(\bm{\lambda}^t)\right\|.
		\end{equation}
		Using the $\mu$-strong convexity of $\tilde{L}_{\alpha}(\cdot,\bm{\lambda}^t)$, we have
		\begin{align}\label{lem3.2 eq3}
			\left\|\nabla_{\m{x}} \tilde{L}_{\alpha}\left(\m{x}^{t+1}, \bm{\lambda}^t\right)\right\| 
			=\left\|\nabla_{\m{x}} \tilde{L}_{\alpha}\left(\m{x}^{t+1}, \bm{\lambda}^t\right)-\nabla_{\m{x}} \tilde{L}_{\alpha}\left(\m{x}^*(\bm{\lambda}^t), \bm{\lambda}^t\right)\right\| 
			\geq \mu\left\|\m{x}^{t+1}-\m{x}^*(\bm{\lambda}^t)\right\| .
		\end{align}
		Thus, combing \eqref{lem3.2 eq2} and \eqref{lem3.2 eq3}, we obtain \eqref{lem3.2 eq1}.
	\end{proof}
	\begin{lemma}\label{lem3.3}
		Under \textbf{Assumptions \ref{as1}}, \textbf{\ref{as2}} and \textbf{\ref{as3}}, the iterates generated from \textbf{Algorithm \ref{alg:Framwork2}} satisfy
		\begin{align}\label{lem3.3 eq1}
			\left\|\m{Z}^{1/2} \m{x}^{t+1}\right\|_{\gamma \m{P}^t}^2 
			\leq \frac{2 \rho \gamma_2}{\mu^2}\left\|\nabla_{\m{x}} \tilde{L}_{\alpha}\left(\m{x}^{t+1}, \bm{\lambda}^t\right)\right\|^2+2\left\|\nabla g(\bm{\lambda}^t)\right\|_{\gamma \m{P}^t}^2
		\end{align}
		and
		\begin{align}\label{lem3.3 eq2}
			\left\|\bm{\lambda}^{t+1}-\bm{\lambda}^t\right\|^2 
			\leq \frac{2 \rho (\gamma_2)^2}{\mu^2}\left\|\nabla_{\m{x}} \tilde{L}_{\alpha}\left(\m{x}^{t+1}, \bm{\lambda}^t\right)\right\|^2+2\gamma_2\left\|\nabla g(\bm{\lambda}^t)\right\|_{\gamma \m{P}^t}^2.
		\end{align}
	\end{lemma}
	\begin{proof}
		Using the triangle inequality, we have
	\begin{align*}
		\left\|\m{Z}^{1/2} \m{x}^{t+1}\right\|_{\gamma \m{P}^t}^2 
		\leq &2\left\|\m{Z}^{1/2} (\m{x}^{t+1}- \m{x}^*(\bm{\lambda}^t))\right\|_{\gamma \m{P}^t}^2+2\left\|\m{Z}^{1/2} \m{x}^*(\bm{\lambda}^t)\right\|_{\gamma \m{P}^t}^2 \\
		=&2\left\|\m{Z}^{1/2} (\m{x}^{t+1}- \m{x}^*(\bm{\lambda}^t))\right\|_{\gamma \m{P}^t}^2+2\left\|\nabla g(\bm{\lambda}^t)\right\|_{\gamma \m{P}^t}^2 \\
		\leq &\frac{2 \rho \gamma_2}{\mu^2}\left\|\nabla_{\m{x}} \tilde{L}_{\alpha}\left(\m{x}^{t+1}, \bm{\lambda}^t\right)\right\|^2+2\left\|\nabla g(\bm{\lambda}^t)\right\|_{\gamma \m{P}^t}^2,
	\end{align*}
	where the equality is due to \eqref{grad_g}, and the last inequality is obtained by \textbf{Lemma \ref{lem3.2}}.
	Since 	$\rho=\bm{\lambda}_{\max}(\m{Z})$, and $\gamma \m{P}^t \preceq \gamma_2 \m{I}$, we have from \eqref{du6} and \eqref{lem3.3 eq1} that 
	\begin{align*}
		\left\|\bm{\lambda}^{t+1}-\bm{\lambda}^{t}\right\|^2
		= &\left\|\gamma \m{P}^t \m{Z}^{1/2} \m{x}^{t+1}\right\|^2 
		\leq  \gamma_2\left\|\m{Z}^{1/2} \m{x}^{t+1}\right\|_{\gamma \m{P}^t}^2 \\
		\leq &\frac{2 \rho (\gamma_2)^2}{\mu^2}\left\|\nabla_{\m{x}} \tilde{L}_{\alpha}\left(\m{x}^{t+1}, \bm{\lambda}^t\right)\right\|^2+2\gamma_2\left\|\nabla g(\bm{\lambda}^t)\right\|_{\gamma \m{P}^t}^2.
	\end{align*}
	\end{proof}
	
	We now have the reduction on the dual optimality gap $\Delta_{\bm{\lambda}}^t$.
	\begin{lemma}\label{lem3.4}
		Under \textbf{Assumptions \ref{as1}}, \textbf{\ref{as2}} and \textbf{\ref{as3}}, the dual optimality gap $\Delta_{\bm{\lambda}}^{t}$ satisfies
		\begin{align}\label{lem3.4 eq1}
			\Delta_{\bm{\lambda}}^{t+1} \leq  \Delta_{\bm{\lambda}}^{t}-\left(\frac{1}{2}-L_g\gamma_2\right)\left\|\nabla g(\bm{\lambda}^{t})\right\|_{\gamma \m{P}^t}^2 
			 +\frac{2L_g\rho (\gamma_2)^2+ \rho \gamma_2}{2\mu^2}\left\|\nabla_{\m{x}} \tilde{L}_{\alpha}\left(\m{x}^{t+1}, \bm{\lambda}^t\right)\right\|^2.
		\end{align}
	\end{lemma}
	\begin{proof}
	Considering the $L_g$-Lipschitz continuity of $\nabla g(\cdot)$ given by \textbf{Lemma \ref{lem3.1}}, we have from \eqref{du6} that
	\begin{align}\label{lem3.4 eq2}
		g(\bm{\lambda}^{t+1}) 
		\geq &g(\bm{\lambda}^{t})+\left\langle\nabla g(\bm{\lambda}^{t}), \bm{\lambda}^{t+1}-\bm{\lambda}^{t}\right\rangle-\frac{L_g}{2}\left\|\bm{\lambda}^{t+1}-\bm{\lambda}^{t}\right\|^2 \\
		= &g(\bm{\lambda}^{t})+\left\langle\nabla g(\bm{\lambda}^{t}), \gamma \m{P}^t \m{Z}^{1/2} \m{x}^{t+1}\right\rangle-\frac{L_g}{2}\left\|\bm{\lambda}^{t+1}-\bm{\lambda}^{t}\right\|^2.\notag
	\end{align}
	In addition, by the inequality that $2\langle\m{a},\m{b}\rangle \geq -\|\m{a}\|^2- \|\m{b}\|^2$ for any $\m{a}, \m{b} \in R^{np}$, we have
	\begin{align}\label{lem3.4 eq3}
		\left\langle\nabla g(\bm{\lambda}^{t}), \gamma \m{P}^t \m{Z}^{1/2} \m{x}^{t+1}\right\rangle 
		=&\left\langle\nabla g(\bm{\lambda}^{t}), \gamma \m{P}^t \m{Z}^{1/2} \m{x}^{t+1}-\gamma\m{P}^t\nabla g(\bm{\lambda}^{t})\right\rangle +\left\langle\nabla g(\bm{\lambda}^{t}), \gamma \m{P}^t \nabla g(\bm{\lambda}^{t})\right\rangle \\
		\geq&-\frac{1}{2}\left\|\nabla g(\bm{\lambda}^{t})\right\|_{\gamma\m{P}^t}^2-\frac{1}{2}\left\|\m{Z}^{1/2} \m{x}^{t+1}-\nabla g(\bm{\lambda}^{t})\right\|_{\gamma \m{P}^t}^2 +\left\|\nabla g(\bm{\lambda}^{t})\right\|_{ \gamma\m{P}^t}^2 \notag\\
		=&\frac{1}{2}\left\|\nabla g(\bm{\lambda}^{t})\right\|_{\gamma \m{P}^t}^2-\frac{1}{2}\left\|\m{Z}^{1/2} \m{x}^{t+1}-\m{Z}^{1/2}\m{x}^*(\bm{\lambda}^t))\right\|_{\gamma \m{P}^t}^2 \notag \\
		\geq & \frac{1}{2}\left\|\nabla g(\bm{\lambda}^{t})\right\|_{\gamma \m{P}^t}^2-\frac{ \rho \gamma_2}{2\mu^2}\left\|\nabla_{\m{x}} \tilde{L}_{\alpha}\left(\m{x}^{t+1}, \bm{\lambda}^{t}\right)\right\|^2,\notag
	\end{align}
	where we use \eqref{grad_g} in the second equality and \textbf{Lemma \ref{lem3.2}} in the second inequality . By substituting \eqref{lem3.3 eq2} and \eqref{lem3.4 eq3} into \eqref{lem3.4 eq2}, we have 
	\begin{align*}
		g(\bm{\lambda}^{t+1}) \geq  g(\bm{\lambda}^{t})+\left(\frac{1}{2}-L_g\gamma_2\right)\left\|\nabla g(\bm{\lambda}^{t})\right\|_{\gamma \m{P}^t}^2
		-\frac{2L_g\rho (\gamma_2)^2+ \rho \gamma_2}{2\mu^2}\left\|\nabla_{\m{x}} \tilde{L}_{\alpha}\left(\m{x}^{t+1}, \bm{\lambda}^t\right)\right\|^2.
	\end{align*}
	Then, the lemma follows from subtracting $g(\bm{\lambda}^*)$ in the above inequality. 
	\end{proof}
	We bound $\Delta_{\m{x}}^{t+1}$ in two steps. We firstly give $\Delta_{\m{x}}^{t+1}$ a bound related to $\Delta_{\m{x}}^{t,S-1}$.
	\begin{lemma}\label{lem3.5}
		Under \textbf{Assumptions \ref{as1}}, \textbf{\ref{as2}} and \textbf{\ref{as3}}, 
		\begin{itemize}
			\item [(1) ] for $S=1$, we have
			\begin{align}\label{lem3.5 eq0}
				\Delta_{\m{x}}^{t+1} \leq  \Delta_{\m{x}}^{t}+2\left\|\nabla g(\bm{\lambda}^t)\right\|_{\gamma \m{P}^t}^2  +\Delta_{{\bm{\lambda}}}^{t}-\Delta_{{\bm{\lambda}}}^{t+1}
				-\left\|\nabla_{\m{x}} \tilde{L}_{\alpha}\left(\m{x}^{t}, \bm{\lambda}^{t}\right)\right\|^2_{\beta \tilde{\m{B}}^{t}-\frac{\C{L}}{2}(\beta \tilde{\m{B}}^{t} )^2 -\frac{2 \rho \gamma_2}{\mu^2}(1+\C{L}\beta_2)\m{I}};
			\end{align}
			\item [(2) ] for $S>1$, we have
			\begin{align*}
				\Delta_{\m{x}}^{t+1} \leq & \Delta_{\m{x}}^{t,S-1}+\frac{2 \rho \gamma_2}{\mu^2}\left\|\nabla_{\m{x}} \tilde{L}_{\alpha}\left(\m{x}^{t+1}, \bm{\lambda}^t\right)\right\|^2
				+2\left\|\nabla g(\bm{\lambda}^t)\right\|_{\gamma \m{P}^t}^2 +\Delta_{{\bm{\lambda}}}^{t}-\Delta_{{\bm{\lambda}}}^{t+1} \\
				& -\left\|\nabla_{\m{x}} \tilde{L}_{\alpha}\left(\m{x}^{t,S-1}, \bm{\lambda}^{t}\right)\right\|^2_{\beta \tilde{\m{B}}^{t,S-1}-\frac{\C{L}}{2}(\beta \tilde{\m{B}}^{t,S-1} )^2 }.
			\end{align*}
		\end{itemize}
	\end{lemma}
	\begin{proof}
		At first, we split $\Delta_{\m{x}}^{t+1}$ into three terms as follows
	\begin{align*}
		\Delta_{\m{x}}^{t+1}= & \tilde{L}_{\alpha}\left(\m{x}^{t+1}, \bm{\lambda}^{t+1}\right)-\tilde{L}_{\alpha}\left(\m{x}^*(\bm{\lambda}^{t+1}), \bm{\lambda}^{t+1}\right) \\
		= & \underbrace{\tilde{L}_{\alpha}\left(\m{x}^{t+1}, \bm{\lambda}^{t+1}\right)-\tilde{L}_{\alpha}\left(\m{x}^{t+1}, \bm{\lambda}^{t}\right)}_{\text {term (A)}} \\
		& +\underbrace{\tilde{L}_{\alpha}\left(\m{x}^{t+1}, \bm{\lambda}^{t}\right)-\tilde{L}_{\alpha}\left(\m{x}^*(\bm{\lambda}^{t}), \bm{\lambda}^{t}\right)}_{\text {term }(\m{B})} \\
		& +\underbrace{\tilde{L}_{\alpha}\left(\m{x}^*(\bm{\lambda}^{t}), \bm{\lambda}^{t}\right)-\tilde{L}_{\alpha}\left(\m{x}^*(\bm{\lambda}^{t+1}), \bm{\lambda}^{t+1}\right)}_{\text {term }(\m{C})}.\\
	\end{align*}
	In the following, We give upper bounds on terms (A)-(C), respectively. For term (A), we have from \eqref{du6} that
	\begin{align}\label{lem3.5 eq1}
		\tilde{L}_{\alpha}\left(\m{x}^{t+1}, \bm{\lambda}^{t+1}\right)-\tilde{L}_{\alpha}\left(\m{x}^{t+1}, \bm{\lambda}^{t}\right)
		=&\left(\bm{\lambda}^{t+1}-\bm{\lambda}^{t}\right)\tr \m{Z}^{1/2} \m{x}^{t+1} \\
		=&\left(\gamma \m{P}^t \m{Z}^{1/2} \m{x}^{t+1}\right)\tr\m{Z}^{1/2}  \m{x}^{t+1}
		= \left\|\m{Z}^{1/2} \m{x}^{t+1}\right\|_{\gamma \m{P}^t}^2\notag\\
		\leq&\frac{2 \rho \gamma_2}{\mu^2}\left\|\nabla_{\m{x}} \tilde{L}_{\alpha}\left(\m{x}^{t+1}, \bm{\lambda}^t\right)\right\|^2+2\left\|\nabla g(\bm{\lambda}^t)\right\|_{\gamma \m{P}^t}^2,\notag
	\end{align}
	where the inequality follows from \eqref{lem3.3 eq1}.
	For term (B), using the $\C{L}$-Lipschitz continuity of $\tilde{L}_{\alpha}(\cdot,\bm{\lambda}^{t})$ and \eqref{x-update-2}, we have
	\begin{align}\label{lem3.5 eq2}
		&\tilde{L}_{\alpha}\left(\m{x}^{t+1}, \bm{\lambda}^{t}\right)-\tilde{L}_{\alpha}\left(\m{x}^*(\bm{\lambda}^{t}), \bm{\lambda}^{t}\right)\\
		\leq & \Delta_{\m{x}}^{t,S-1}+\nabla_{\m{x}} \tilde{L}_{\alpha}\left(\m{x}^{t,S-1}, \bm{\lambda}^{t}\right)\tr\left(\m{x}^{t+1}-\m{x}^{t,S-1}\right)  +\frac{\C{L}}{2}\left\|\m{x}^{t+1}-\m{x}^{t,S-1}\right\|^2 \notag\\
		= & \Delta_{\m{x}}^{t,S-1} -\left\| \nabla_{\m{x}} \tilde{L}_{\alpha}\left(\m{x}^{t,S-1}, \bm{\lambda}^{t}\right)\right\|^2_{\beta \tilde{\m{B}}^{t,S-1}}  +\frac{\C{L}}{2}\left\|\beta \tilde{\m{B}}^{t,S-1} \nabla_{\m{x}} \tilde{L}_{\alpha}\left(\m{x}^{t,S-1}, \bm{\lambda}^{t}\right)\right\|^2.\notag
	\end{align}
	For term (C), since $\tilde{L}_{\alpha}\left(\m{x}^*(\bm{\lambda}), {\bm{\lambda}}\right)=g(\bm{\lambda})$ for any $\bm{\lambda} \in R^{np}$, we have
	\begin{align}\label{lem3.5 eq3}
		\tilde{L}_{\alpha}\left(\m{x}^*(\bm{\lambda}^{t}), {\bm{\lambda}}^{t}\right)-\tilde{L}_{\alpha}\left(\m{x}^*(\bm{\lambda}^{t+1}), {\bm{\lambda}}^{t+1}\right)
		=g(\bm{\lambda}^{t})-g(\bm{\lambda}^{t+1}) 
		=\Delta_{\bm{\lambda}}^{t}-\Delta_{\bm{\lambda}}^{t+1} .
	\end{align}
	Adding \eqref{lem3.5 eq1}, \eqref{lem3.5 eq2} and \eqref{lem3.5 eq3} together, we obtain
	\begin{align}\label{lem3.5 eq4}
		\Delta_{\m{x}}^{t+1} \leq & \Delta_{\m{x}}^{t,S-1}+\frac{2 \rho \gamma_2}{\mu^2}\left\|\nabla_{\m{x}} \tilde{L}_{\alpha}\left(\m{x}^{t+1}, \bm{\lambda}^t\right)\right\|^2 +2\left\|\nabla g(\bm{\lambda}^t)\right\|_{\gamma \m{P}^t}^2  +\Delta_{{\bm{\lambda}}}^{t}-\Delta_{{\bm{\lambda}}}^{t+1}\\
		& -\left\|\nabla_{\m{x}} \tilde{L}_{\alpha}\left(\m{x}^{t,S-1}, \bm{\lambda}^{t}\right)\right\|^2_{\beta \tilde{\m{B}}^{t,S-1}-\frac{\C{L}}{2}(\beta \tilde{\m{B}}^{t,S-1} )^2 }.\notag
	\end{align}
	When $S=1$, by \eqref{L_bound} and \eqref{bound1}, we have 
	\begin{equation}\label{lem3.5 eq5}
		\left\|\nabla_{\m{x}} \tilde{L}_{\alpha}\left(\m{x}^{t+1}, \bm{\lambda}^t\right)\right\|^2 \leq (1+\C{L}\beta_2)\left\|\nabla_{\m{x}} \tilde{L}_{\alpha}\left(\m{x}^{t}, \bm{\lambda}^t\right)\right\|^2.
	\end{equation}
	Substituting the above inequality to \eqref{lem3.5 eq4} and setting $S=1$, we obtain
	\begin{align*}
		\Delta_{\m{x}}^{t+1} \leq  \Delta_{\m{x}}^{t}+2\left\|\nabla g(\bm{\lambda}^t)\right\|_{\gamma \m{P}^t}^2  +\Delta_{{\bm{\lambda}}}^{t}-\Delta_{{\bm{\lambda}}}^{t+1}-\left\|\nabla_{\m{x}} \tilde{L}_{\alpha}\left(\m{x}^{t}, \bm{\lambda}^{t}\right)\right\|^2_{\beta \tilde{\m{B}}^{t}-\frac{\C{L}}{2}(\beta \tilde{\m{B}}^{t} )^2 -\frac{2 \rho \gamma_2}{\mu^2}(1+\C{L}\beta_2)\m{I}}\notag
	\end{align*}
	\end{proof}
	We now show that $\Delta_{\m{x}}^{t,s}$ decays after each iteration in the inner loop of \textbf{Algorithm \ref{alg:Framwork2}}.
	\begin{lemma}\label{lem3.6}
		Under \textbf{Assumptions \ref{as1}}, \textbf{\ref{as2}} and \textbf{\ref{as3}}, if $\beta_2\leq1/\C{L}$, we have
		\begin{equation*}
			\Delta_{\m{x}}^{t,S-1} \leq  (1-\beta_1\mu)^{S-1}\Delta_{\m{x}}^{t},
		\end{equation*}
		where $\beta_1$ and $\beta_2$ are defined in \eqref{bound1}. 
	\end{lemma}
	
	\begin{proof}
		Using the $\C{L}$-Lipschitz continuity of $\tilde{L}_{\alpha}(\cdot,\bm{\lambda}^{t})$ and \eqref{x-update-2}, we have
	\begin{align}\label{lem3.6 eq1}
		\Delta_{\m{x}}^{t,S-1}
		=&\tilde{L}_{\alpha}\left(\m{x}^{t,S-1}, \bm{\lambda}^{t}\right)-\tilde{L}_{\alpha}\left(\m{x}^*(\bm{\lambda}^{t}), {\bm{\lambda}}^{t}\right)\\
		\leq & \Delta_{\m{x}}^{t,S-2}+\nabla_{\m{x}} \tilde{L}_{\alpha}\left(\m{x}^{t,S-2}, \bm{\lambda}^{t}\right)\tr\left(\m{x}^{t,S-1}-\m{x}^{t,S-2}\right)  +\frac{\C{L}}{2}\left\|\m{x}^{t,S-1}-\m{x}^{t,S-2}\right\|^2 \notag\\
		=&\Delta_{\m{x}}^{t,S-2} -\left\|\nabla_{\m{x}} \tilde{L}_{\alpha}\left(\m{x}^{t,S-2}, \bm{\lambda}^{t}\right)\right\|^2_{\beta \tilde{\m{B}}^{t,S-2}-\frac{\C{L}}{2}(\beta \tilde{\m{B}}^{t,S-2} )^2 }.\notag
	\end{align}
	By the definitions of $\beta_1$ and $\beta_2$ defined in \eqref{bound1}, if $\beta_2\leq1/\C{L}$, we have 
	\begin{align}\label{lem3.6 eq2}
		\beta \tilde{\m{B}}^{t,S-2}-\frac{\C{L}}{2}(\beta \tilde{\m{B}}^{t,S-2} )^2
		=\beta \tilde{\m{B}}^{t,S-2}\left(\m{I}-\frac{\C{L}}{2}\beta \tilde{\m{B}}^{t,S-2} \right)
		\succeq  \frac{1}{2}\beta \tilde{\m{B}}^{t,S-2}\succeq  \frac{\beta_1}{2}\m{I}. 
	\end{align}
	By the $\mu$-strong convexity of $\tilde{L}_{\alpha}(\cdot,\bm{\lambda}^{t})$, we have
	\begin{align}\label{lem3.6 eq3}
		\left\|\nabla_{\m{x}} \tilde{L}_{\alpha}\left(\m{x}^{t,S-2}, \bm{\lambda}^{t}\right)\right\|^2
		\geq  2\mu\left(\tilde{L}_{\alpha}\left(\m{x}^{t,S-2}, {\bm{\lambda}}^{t}\right)-\tilde{L}_{\alpha}\left(\m{x}^*(\bm{\lambda}^{t}), {\bm{\lambda}}^{t}\right)\right)
		= 2\mu\Delta_{\m{x}}^{t,S-2}.
	\end{align}
	Substituting \eqref{lem3.6 eq2} and \eqref{lem3.6 eq3} into \eqref{lem3.6 eq1}, we obtain
	\begin{align*}
		\Delta_{\m{x}}^{t,S-1}\leq (1-\beta_1\mu)\Delta_{\m{x}}^{t,S-2}
		\leq(1-\beta_1\mu)^{S-1}\Delta_{\m{x}}^{t}.
	\end{align*}
	\end{proof}
	Combining \textbf{Lemma \ref{lem3.4}}, \textbf{Lemma \ref{lem3.5}} and \textbf{Lemma \ref{lem3.6}}, we establish the linear decay rate of $\Delta^t$.
	\begin{theorem}\label{the3.7}
		Under \textbf{Assumptions \ref{as1}}, \textbf{\ref{as2}} and \textbf{\ref{as3}}, 
		if the primal and dual step sizes, namely $\beta$ and $\gamma$, satisfy
		$\beta \leq \frac{1}{\Psi{\C{L}}}$ and 
		\begin{equation*}
			\gamma < \min \left\{\frac{1}{6\bar{\alpha}L_g},\left(\frac{1}{(1-\beta\psi\mu)^{S}}-1\right)\frac{\mu^2}{11\bar{\alpha} \C{L}\rho}\right\},
		\end{equation*}
		there exists a constant $\kappa \in (0,1)$ such that 
		\begin{equation*}
			\Delta^{t+1} \leq \kappa\Delta^t,
		\end{equation*}
		where
		\begin{align}\label{def-kappa}
			\kappa=\max\Bigg\{1-\frac{2}{7}\left(1-6L_g\gamma\bar{\alpha}\right) \mu_g\gamma \alpha,
			\left(1+\frac{(12L_g\gamma \bar{\alpha}+10)\C{L}\rho \gamma \bar{\alpha}}{\mu^2}\right)(1-\beta\psi\mu)^{S}\Bigg\}.
		\end{align}
		Especially, when $S=1$, if $\beta\leq \frac{1}{2\Psi{\C{L}}}$ and $	\gamma \leq \min\left\{\frac{1}{12\bar{\alpha}L_g},\frac{4\mu^2}{99\bar{\alpha}\C{L}\rho }\right\}$, 
		there exists a constant $q \in (0,1)$ such that 
		\begin{equation*}
			\Delta^{t+1} \leq (1-q)\Delta^t,
		\end{equation*}
		where $q= \min \left\{\frac{\gamma_1\mu_g}{7},\frac{\beta\psi\mu}{2}\right\}$.
	\end{theorem}
	\begin{proof}
			Using the derivation similar to \textbf{Lemma \ref{lem3.6}}, if $\beta_2\leq1/\C{L}$, we have
		\begin{align}\label{the3.7 eq2}
			\Delta_{\m{x}}^{t+1} \leq  (1-\beta_1\mu)^{S}\Delta_{\m{x}}^{t}+\frac{2 \rho \gamma_2}{\mu^2}\left\|\nabla_{\m{x}} \tilde{L}_{\alpha}\left(\m{x}^{t+1}, \bm{\lambda}^t\right)\right\|^2
			+2\left\|\nabla g(\bm{\lambda}^t)\right\|_{\gamma \m{P}^t}^2 +\Delta_{{\bm{\lambda}}}^{t}-\Delta_{{\bm{\lambda}}}^{t+1},
		\end{align}
		and
		\begin{align}\label{the3.7 eq3}
			\left\|\nabla_{\m{x}} \tilde{L}_{\alpha}\left(\m{x}^{t+1}, \bm{\lambda}^t\right)\right\|^2
			\leq2\C{L} \left(\tilde{L}_{\alpha}\left(\m{x}^{t+1}, \bm{\lambda}^{t}\right)-\tilde{L}_{\alpha}\left(\m{x}^*(\bm{\lambda}^{t}), \bm{\lambda}^{t}\right)\right)
			\leq 2\C{L}(1-\beta_1\mu)^{S}\Delta_{\m{x}}^{t},
		\end{align}
		where the first inequality is by the $\C{L}$-Lipschitz continuity of $\tilde{L}_{\alpha}(\cdot,\bm{\lambda}^{t})$.
		Substituting \eqref{the3.7 eq3} into \eqref{the3.7 eq2} and \eqref{lem3.4 eq1}, respectively, we obtain
		\begin{align}\label{the3.7 eq4}
			\Delta_{\m{x}}^{t+1} \leq  (1-\beta_1\mu)^{S}\Delta_{\m{x}}^{t}+\frac{4\C{L} \rho \gamma_2(1-\beta_1\mu)^{S}}{\mu^2}\Delta_{\m{x}}^{t}
			+2\left\|\nabla g(\bm{\lambda}^t)\right\|_{\gamma \m{P}^t}^2+\Delta_{{\bm{\lambda}}}^{t}-\Delta_{{\bm{\lambda}}}^{t+1},
		\end{align}
		\begin{align}\label{the3.7 eq5}
			\Delta_{\bm{\lambda}}^{t+1} \leq  \Delta_{\bm{\lambda}}^{t}-\left(\frac{1}{2}-L_g\gamma_2\right)\left\|\nabla g(\bm{\lambda}^{t})\right\|_{\gamma \m{P}^t}^2 
			+\frac{2L_g\rho (\gamma_2)^2\C{L}+ \rho \gamma_2\C{L}}{\mu^2}(1-\beta_1\mu)^{S}\Delta_{\m{x}}^{t},
		\end{align}
		Adding \eqref{the3.7 eq5} multiplied by $6$ with \eqref{the3.7 eq4}, we have
		\begin{align*}
			7\Delta_{\bm{\lambda}}^{t+1}+\Delta_{\m{x}}^{t+1}
			\leq  7\Delta_{\bm{\lambda}}^{t}-\left(1-6L_g\gamma_2\right)\left\|\nabla g(\bm{\lambda}^{t})\right\|_{\gamma \m{P}^t}^2
			+\left(1+\frac{(12L_g\gamma_2+10)\C{L}\rho \gamma_2}{\mu^2}\right)(1-\beta_1\mu)^{S}\Delta_{\m{x}}^{t}
		\end{align*}
		It follows from \textbf{Lemma \ref{lem3.1}} that 
		\begin{equation*}
			\|\nabla g({\bm{\lambda}}^t)\|^2 \geq 2 \mu_g\left(g\left({\bm{\lambda}}^*\right)-g({\bm{\lambda}^t})\right)=2 \mu_g \Delta_{{\bm{\lambda}}}^t.
		\end{equation*}
		Then, we have from \eqref{bound2} that
		\begin{align*}
			&7\Delta_{\bm{\lambda}}^{t+1}+\Delta_{\m{x}}^{t+1}\\
			\leq &7\Delta_{\bm{\lambda}}^{t}-2\left(1-6L_g\gamma_2\right) \mu_g\gamma_1\Delta_{\bm{\lambda}}^{t}+\left(1+\frac{(12L_g\gamma_2+10)\C{L}\rho \gamma_2}{\mu^2}\right)(1-\beta_1\mu)^{S}\Delta_{\m{x}}^{t}\\
			=&\left(1-\frac{2}{7}\left(1-6L_g\gamma_2\right) \mu_g\gamma_1\right)7\Delta_{\bm{\lambda}}^{t}+\left(1+\frac{(12L_g\gamma_2+10)\C{L}\rho \gamma_2}{\mu^2}\right)(1-\beta_1\mu)^{S}\Delta_{\m{x}}^{t}\\
			\leq&\kappa\left(7\Delta_{\bm{\lambda}}^{t}+\Delta_{\m{x}}^{t}\right),
		\end{align*}
		where 
		\begin{align*}
			\kappa=\max\Bigg\{1-\frac{2}{7}\left(1-6L_g\gamma_2\right) \mu_g\gamma_1,
			\left(1+\frac{(12L_g\gamma_2+10)\C{L}\rho \gamma_2}{\mu^2}\right)(1-\beta_1\mu)^{S}\Bigg\}.
		\end{align*}
		The condition that
		\begin{equation*}
			\gamma_2 < \min \left\{\frac{1}{6L_g},\left(\frac{1}{(1-\beta_1\mu)^{S}}-1\right)\frac{\mu^2}{11 \C{L}\rho}\right\}
		\end{equation*}
		ensures $\kappa \in (0,1)$, which implies
		$\Delta^t = 7\Delta_{\bm{\lambda}}^{t}+\Delta_{\m{x}}^{t}$ decays at a linear rate.
		
		For $S=1$, we use a similar proof framework but with different details. By \eqref{lem3.4 eq1} and \eqref{lem3.5 eq5}, we have
		\begin{align}\label{the3.7 eq6}
			\Delta_{\bm{\lambda}}^{t+1} \leq  \Delta_{\bm{\lambda}}^{t}-\left(\frac{1}{2}-L_g\gamma_2\right)\left\|\nabla g(\bm{\lambda}^{t})\right\|_{\gamma \m{P}^t}^2 
			+\frac{2L_g\rho (\gamma_2)^2+ \rho \gamma_2}{2\mu^2}(1+\C{L}\beta_2)^2\left\|\nabla_{\m{x}} \tilde{L}_{\alpha}\left(\m{x}^{t}, \bm{\lambda}^t\right)\right\|^2,
		\end{align}
		Adding \eqref{the3.7 eq6} multiplied by $6$ with \eqref{lem3.5 eq0} and denoting
		$$
		\hat{\m{B}}^t=\beta \tilde{\m{B}}^{t}-\frac{\C{L}}{2}(\beta \tilde{\m{B}}^{t} )^2 -\frac{12L_g\rho (\gamma_2)^2+ 10\rho \gamma_2}{2\mu^2}(1+\C{L}\beta_2)^2\m{I},
		$$ 
		we have
		\begin{align}\label{the3.7 eq7}
			&7\Delta_{\bm{\lambda}}^{t+1}+\Delta_{\m{x}}^{t+1}\\
			\leq & 7\Delta_{\bm{\lambda}}^{t}+\Delta_{\m{x}}^{t}-\left(1-6L_g\gamma_2\right)\left\|\nabla g(\bm{\lambda}^{t})\right\|_{\gamma \m{P}^t}^2-\nabla_{\m{x}} \tilde{L}_{\alpha}\left(\m{x}^{t}, \bm{\lambda}^{t}\right) \tr \hat{\m{B}}^t\nabla_{\m{x}} \tilde{L}_{\alpha}\left(\m{x}^{t}, \bm{\lambda}^{t}\right)\notag\\
			\leq &7\Delta_{\bm{\lambda}}^{t}+\Delta_{\m{x}}^{t}-\left(1-6L_g\gamma_2\right)\left\|\nabla g(\bm{\lambda}^{t})\right\|_{\gamma \m{P}^t}^2-\nabla_{\m{x}} \tilde{L}_{\alpha}\left(\m{x}^{t}, \bm{\lambda}^{t}\right) \tr \left(\frac{1}{2}\beta \tilde{\m{B}}^{t}-\frac{\C{L}}{2}(\beta \tilde{\m{B}}^{t} )^2 \right)\nabla_{\m{x}} \tilde{L}_{\alpha}\left(\m{x}^{t}, \bm{\lambda}^{t}\right)\notag
		\end{align}
		where the second inequality is due to $\frac{1}{2}\beta \tilde{\m{B}}^{t} -\frac{(12L_g\gamma_2+10)\rho \gamma_2}{2\mu^2}(1+\C{L}\beta_2)^2\m{I} \succeq \m{0}$ which is from 
		$$
		\gamma_2=\bar{\alpha}\gamma \leq \min\left\{\frac{1}{12L_g},\frac{4\mu^2\beta_1}{99\rho }\right\} \text{ and } \beta_2 = \Psi \beta \leq \frac{1}{2\C{L}}.
		$$
		Further from the relation $\beta_2 \leq \frac{1}{2\C{L}}$, we obtain
		\begin{align*}
			\frac{1}{2}\beta \tilde{\m{B}}^{t}-\frac{\C{L}}{2}(\beta \tilde{\m{B}}^{t} )^2 \succeq \frac{\lambda_{\min}(\beta\tilde{\m{B}}^t)}{4}\m{I}=\frac{\beta_1}{4}\m{I}.
		\end{align*}
		Then using the inequality $\left\|\nabla_{\m{x}} \tilde{L}_{\alpha}\left(\m{x}^{t}, \bm{\lambda}^{t}\right)\right\|^2 \geq 2 \mu \Delta_{\m{x}}^{t}$ from the $\mu$-strong convexity of $\tilde{L}_{\alpha}(\cdot,\bm{\lambda}^{t})$, we have
		\begin{align}\label{the3.7 eq8}
			\left\|\nabla_{\m{x}} \tilde{L}_{\alpha}\left(\m{x}^{t}, \bm{\lambda}^{t}\right)\right\|^2_{\frac{1}{2}\beta \tilde{\m{B}}^{t}-\frac{\C{L}}{2}(\beta \tilde{\m{B}}^{t} )^2} \geq\frac{\beta_1\mu}{2}\Delta_{\m{x}}^{t}.
		\end{align}
		It follows from \textbf{Lemma \ref{lem3.1}} that 
		\begin{equation*}
			\|\nabla g({\bm{\lambda}}^t)\|^2 \geq 2 \mu_g\left(g\left({\bm{\lambda}}^*\right)-g({\bm{\lambda}^t})\right)=2 \mu_g \Delta_{{\bm{\lambda}}}^t.
		\end{equation*}
		Then we have by $\gamma_2 \leq \frac{1}{12L_g}$ that
		\begin{equation}\label{the3.7 eq9}
			\left(1-6L_g\gamma_2\right)\left\|\nabla g(\bm{\lambda}^{t})\right\|_{\gamma \m{P}^t}^2 \geq \gamma_1\mu_g \Delta_{{\bm{\lambda}}}^t.
		\end{equation}
		By substituting \eqref{the3.7 eq8} and \eqref{the3.7 eq9} into \eqref{the3.7 eq7}, we obtain
		\begin{align*}
			7\Delta_{\bm{\lambda}}^{t+1}+\Delta_{\m{x}}^{t+1}
			\leq &7\Delta_{\bm{\lambda}}^{t}-\gamma_1\mu_g \Delta_{{\bm{\lambda}}}^t+\Delta_{\m{x}}^{t}-\frac{\beta_1\mu}{2}\Delta_{\m{x}}^{t}\\
			= &7\left(1-\frac{\gamma_1\mu_g}{7}\right)\Delta_{{\bm{\lambda}}}^t+\left(1-\frac{\beta_1\mu}{2}\right)\Delta_{\m{x}}^{t}\\
			\leq & (1-q)\left(	7\Delta_{\bm{\lambda}}^{t}+\Delta_{\m{x}}^{t}\right)
		\end{align*}
		where $q= \min \left\{\frac{\gamma_1\mu_g}{7},\frac{\beta_1\mu}{2}\right\}$. Using the definitions of $\beta_1$, $\beta_2$ defined in \eqref{bound1} and   $\gamma_1$, $\gamma_2$ defined in \eqref{bound2} completes the proof.
	\end{proof}
	By \textbf{Theorem~\ref{the3.7}}, we can directly establish that  $\m{x}^t$ converges to the 
	unique minimizer $\m{x}^*=\m{x}^*(\bm{\lambda}^*)$ of the original problem \eqref{obj_fun4} at a linear rate.
	\begin{corollary}\label{cor3.8}
		Under \textbf{Assumptions \ref{as1}}, \textbf{\ref{as2}} and \textbf{\ref{as3}}, 
		if the primal and dual step sizes, namely $\beta$ and $\gamma$, satisfy
		$\beta \leq  \frac{1}{\Psi{\C{L}}}$ and 
		\begin{equation*}
			\gamma < \min \left\{\frac{1}{6\bar{\alpha}L_g},\left(\frac{1}{(1-\beta\psi\mu)^{S}}-1\right)\frac{\mu^2}{11\bar{\alpha} \C{L}\rho}\right\}, 
		\end{equation*}
		the iterates $\{\m{x}^t \}$
		generated by \textbf{Algorithm \ref{alg:Framwork2}} converge to $\m{x}^*(\bm{\lambda}^*)$
		linearly, more specifically, there exists a constant $\kappa \in (0,1)$ such that
		\begin{equation*}
			\left\|\m{x}^t-\m{x}^*(\bm{\lambda}^*) \right\|^2\leq c\kappa^t,
		\end{equation*}
		where $c={4\Delta^0}/\left({\mu\min\left\{\frac{7\mu}{\C{L}},1\right\}}\right)$ and
		$0< \kappa < 1$ is defined in \eqref{def-kappa}. Especially, when $S=1$, if $\beta\leq \frac{1}{2\Psi{\C{L}}}$ and $	\gamma \leq \min\left\{\frac{1}{12\bar{\alpha}L_g},\frac{4\mu^2}{99\bar{\alpha}\C{L}\rho }\right\}$, we have $\left\|\m{x}^t-\m{x}^*(\bm{\lambda}^*) \right\|^2\leq c(1-q)^t$, where $q= \min \left\{\frac{\gamma_1\mu_g}{7},\frac{\beta \psi\mu}{2}\right\}$.
	\end{corollary}
	\begin{proof}
	
	By \eqref{bound-xxxx}, for any $\bm{\lambda}_1, \bm{\lambda}_2 \in R^{np}$, we have
	\begin{equation}\label{cor3.8 eq1}
		\left\|\m{x}^*(\bm{\lambda}_1)-\m{x}^*(\bm{\lambda}_2) \right\| \leq \frac{1}{\mu}\left\|\m{Z}^{1/2}\bm{\lambda}_1-\m{Z}^{1/2}\bm{\lambda}_2 \right\|.
	\end{equation}
	We now give lower bounds on $\Delta_{\bm{\lambda}}^t$ and $\Delta_{\m{x}}^{t}$, respectively.
	Using $1/\C{L}$-strong convexity of $\tilde{f}^*(\cdot)$, we have
	\begin{align}\label{cor3.8 eq2}
		&\Delta_{\bm{\lambda}}^t=g\left({\bm{\lambda}}^*\right)-g\left({\bm{\lambda}}^t\right)
		=\tilde{f}^*(-\m{Z}^{1/2}\bm{\lambda}^t)-\tilde{f}^*(-\m{Z}^{1/2}\bm{\lambda}^*)\\
		\geq&\frac{1}{2\C{L}}\left\|\m{Z}^{1/2}\bm{\lambda}^t-\m{Z}^{1/2}\bm{\lambda}^* \right\|^2
		\geq \frac{\mu^2}{2\C{L}}\left\|\m{x}^*(\bm{\lambda}^t)-\m{x}^*(\bm{\lambda}^*) \right\|^2,\notag
	\end{align}
	where the last inequality follows from \eqref{cor3.8 eq1}. 
	On the other hand, by the $\mu$-strong convexity of $\tilde{L}_{\alpha}(\cdot,\bm{\lambda}^t)$, we obtain
	\begin{align}\label{cor3.8 eq3}
		\Delta_{\m{x}}^t=\tilde{L}_{\alpha}\left(\m{x}^t, {\bm{\lambda}}^t\right)-\tilde{L}_{\alpha}\left(\m{x}^*\left({\bm{\lambda}}^t\right), {\bm{\lambda}}^t\right)
		\geq  \frac{\mu}{2}\left\|\m{x}^t-\m{x}^*(\bm{\lambda}^t) \right\|^2.
	\end{align}
	Substituting \eqref{cor3.8 eq2} and \eqref{cor3.8 eq3} into the expression of $\Delta^t$ yields
	\begin{align}\label{cor3.8 eq4}
		\Delta^t=&7\Delta_{\bm{\lambda}}^t+\Delta_{\m{x}}^t
		\geq\frac{7\mu^2}{2\C{L}}\left\|\m{x}^*(\bm{\lambda}^t)-\m{x}^*(\bm{\lambda}^*) \right\|^2+\frac{\mu}{2}\left\|\m{x}^t-\m{x}^*(\bm{\lambda}^t) \right\|^2\\
		\geq &\frac{\mu}{2}\min\left\{\frac{7\mu}{\C{L}},1\right\}\Big[\left\|\m{x}^*(\bm{\lambda}^t)-\m{x}^*(\bm{\lambda}^*) \right\|^2+\left\|\m{x}^t-\m{x}^*(\bm{\lambda}^t) \right\|^2\Big]\notag\\
		\geq &\frac{\mu}{4}\min\left\{\frac{7\mu}{\C{L}},1\right\}\left\|\m{x}^t-\m{x}^*(\bm{\lambda}^*) \right\|^2,\notag
	\end{align}
	where the last inequality is due to the inequality that $\|\m{a}\|^2+\|\m{b}\|^2 \geq (1/2)\|\m{a}+\m{b}\|^2$ for any $\m{a},\m{b} \in R^{np}$. Thus, by using Theorem \ref{the3.7}, we have
	\begin{equation*}
		\frac{\mu}{4}\min\left\{\frac{7\mu}{\C{L}},1\right\}\left\|\m{x}^t-\m{x}^*(\bm{\lambda}^*) \right\|^2\leq\kappa^t\Delta^0.
	\end{equation*}
	The proof for the case $S=1$ follows the above steps and using the definitions of $\beta_1$, $\beta_2$ defined in \eqref{bound1} and   $\gamma_1$, $\gamma_2$ defined in \eqref{bound2} completes the proof.
	\end{proof}
	\begin{remark}
		The condition numbers of the objective function and augmented Lagrangian function are defined as
		\begin{equation*}
			\kappa_f:=\frac{L}{\mu} \; \text{ and } \; \kappa_l:=\frac{\C{L}}{\mu}.
		\end{equation*}
		The condition number of the network can be defined as
		\begin{equation*}
			\kappa_g:=\frac{\lambda_{\max}(\m{I}-\tilde{\m{W}})}{\lambda_2(\m{I}-\tilde{\m{W}})}=\frac{\rho}{\sigma},
		\end{equation*}
		which measures the network topology and is an important factor affecting the performance 
		of decentralized methods. In general, a smaller condition number means greater connectivity of the network.
		From \textbf{Corollary \ref{cor3.8}}, we can derive the iteration complexity of our DPDM as
		$$
		\mathcal{O}\left( \kappa_l^2\kappa_g log\left( \frac{1}{\epsilon}\right) \right). $$ 
		When $S$ is sufficiently large, the iteration complexity of GDPDM can reduce to $\mathcal{O}\left( \kappa_l \kappa_g log\left( \frac{1}{\epsilon}\right) \right) $.
		
		Our methods do not have a faster-than-linear rate since the dual Hessian is approximated by the BB technique. 
		It is well-known that the BB method only has a linear convergence rate for minimizing strongly convex problems, which is the same as the steepest
		descent (SD) method. However, this does not limit the practical much superior performance
		of the BB method over the SD method. In our experiments, our methods are directly compared with the primal-dual methods using other dual ascent step sizes, such as NT and ESOM.
		
		Note that existing decentralized second-order and quasi-Newton methods exhibit different performance across several aspects, including primal-dual updates, convergence rates, computational costs, 
		etc.. Here we compare our method with other major decentralized second-order and quasi-Newton methods. The comparisons are summarized in Table~\ref{comparison}.

		\begin{table}[H]
			\tiny
			\caption{Comparisons of decentralized second-order and quasi-Newton methods}\label{comparison}
			\centering
			\begin{threeparttable}
				\begin{tabular}{c|c|c|c|c|c|c|c}
					\hline
					&\multirow{2}{*}{Primal order}  &\multirow{2}{*}{Dual order}  &\multirow{2}{*}{\shortstack{Convergence\\ rate}}
					&\multirow{2}{*}{Iteration complexity}
					&\multirow{2}{*}{Assumption\tnote{3}}
					&\multirow{2}{*}{\shortstack{Calculate\\ Hessian or not}} &\multirow{2}{*}{\shortstack{Communication\\ overhead\tnote{4}}}   \\
					&  &  &  & &  &  &  \\
					\hline
					\multirow{3}{*}{NT\tnote{5}\cite{zhang2021newton}} &\multirow{3}{*}{Second}  &\multirow{3}{*}{First}  &\multirow{3}{*}{Linear} &\multirow{3}{*}{\shortstack{$O( \max \{\kappa_f \sqrt{\kappa_g}+ \kappa_f^2,$ \\ $\kappa_g^{3/2}/\kappa_f+\kappa_f \sqrt{\kappa_g}\} log( \frac{1}{\epsilon}) ) $}} &\multirow{3}{*}{SC+LG+LH} &\multirow{3}{*}{Yes}  &\multirow{3}{*}{$p$}  \\
					&  &  &  & &  &  &  \\
					&  &  &  & &  &  &  \\
					ESOM-$K$\cite{mokhtari2016decentralized}&Second  &First  &Linear &$O\left( \kappa_f^2/\sigma log\left( \frac{1}{\epsilon}\right) \right) $ &SC+LG+LH &Yes  &$(1+K)p$  \\
					\hline
					DBFGS\cite{eisen2017decentralized}&Quasi second\tnote{1} &$\backslash$  &Inexactly linear\tnote{2} &$\backslash$ &SC+LG &No  &$3p$  \\
					
					DR-LM-DFP\cite{zhang2023variance}&Quasi second  &$\backslash$  &Linear &$O\left( \kappa_f^2\kappa_g^2 log\left( \frac{1}{\epsilon}\right) \right) $ &SC+LG &No  &$2p$  \\
					
					D-LM-BFGS\cite{zhang2023variance}&Quasi second  &$\backslash$  &Linear &$O\left( \kappa_f^2\kappa_g^2 log\left( \frac{1}{\epsilon}\right) \right) $ &SC+LG &No  &$2p$  \\
					
					PD-QN-$K$\cite{eisen2019primal}&Quasi second  &Quasi second  &$\backslash$ &$\backslash$ &SC+LG &No  &$(4+K)p$  \\
					
					DPDM&Quasi second  &Quasi second  &Linear &$O\left( \kappa_l^2\kappa_g log\left( \frac{1}{\epsilon}\right) \right) $ &SC+LG &No  &$(2+ I_{R \setminus \{0\}}(\theta))p+2$\tnote{$\star$}  \\
					\hline
				\end{tabular}
				\begin{tablenotes}
					\item[1] ``Quasi second" means the second-order information is captured by Hessian approximations using gradient information. 
					\item[2] ``Inexactly linear'' means that the method only converges linearly to a small neighborhood of the solution.
					\item [3] ``SC'', ``LG'', ``LH'' respectively mean ``Strong convexity'', ``Lipschitz continuous gradient'', ``Lipschitz continuous Hessian''.
					\item[4] For the fixed underlying network, communication overhead can be defined as rounds of communication per iteration $\times$ dimensions of transmitted scalars, vectors, or matrices. 
					\item[5] The iteration complexity of NT is computed when $\m{x}^t$ is close to $\m{x}^*$. So the global iteration complexity is larger than the value given above. 
					\item[$\star$] Here, the addition of $2$ is because of the additional $2$ scalar communications of $a_j^t$ and $b_j^t$ in each iteration of DPDM. 
				\end{tablenotes}
			\end{threeparttable}
		\end{table}
		
	\end{remark}
	
	\section{NUMERICAL EXPERIMENTS}
	In this section, we would like to test and compare our developed algorithms with 
	some well-developed first-order and second-order algorithms 
	on solving the linear regression problems and the 
	logistic regression problems over a connected undirected network with edge density $d \in (0,1]$.
	An additional experiment about the effect of network topology is also presented here.
	For the generated network, we choose the Metropolis constant edge weight matrix \cite{xiao2007distributed} 
	as the mixing matrix, that is
	\begin{equation*}
		\tilde{W}_{i j}=\left\{\begin{array}{cl}
			\frac{1}{\max \{\operatorname{deg}(i), \operatorname{deg}(j)\}+1}, & \text { if }(i, j) \in \mathcal{E}, \\
			0, & \text { if }(i, j) \notin \mathcal{E} \text { and } i \neq j, \\
			1-\sum_{k \in \mathcal{N}_i \setminus \{i\}} \tilde{W}_{i k}, & \text { if } i=j,
		\end{array}\right.
	\end{equation*}
	where $(i, j) \in \mathcal{E}$ indicates there is an edge between 
	node $i$ and node $j$, and $\operatorname{deg}(i)$ means the degree of node $i$. 
	
	Based on our analysis, we also propose an algorithm $\text{GDPDM}^{+}$, which applies an adaptive stop 
	criteria for the inner loop as stated in \textbf{Algorithm \ref{alg:Framwork3}}, while the remaining steps are the same as those of GDPDM.
	In fact, it is not difficult to see the global convergence of $\text{GDPDM}^{+}$ can be also obtained
	within the convergence analysis framework of $\text{GDPDM}$.
	Our comparison algorithms in this section include GT\cite{qu2017harnessing}, EXTRA\cite{shi2015extra}, ESOM-1\cite{mokhtari2016decentralized}(called ESOM below), NT\cite{zhang2021newton}, DBFGS\cite{eisen2017decentralized}, Damped regularized limited-memory DFP\cite{zhang2023variance}(called DR-LM-DFP below), and Damped limited-memory BFGS\cite{zhang2023variance}(called D-LM-BFGS below). 
	\begin{algorithm}[H]
		\caption{Inner loop of $\text{GDPDM}^{+}$ with respect to node $i$}
		\label{alg:Framwork3}
		\begin{algorithmic}[1]
			\State $\m{x}_i^{t,0}=\m{x}_i^{t}$, $\m{H}_i^{t,0}=\m{H}_i^{t}$, block$(i)=1$.
			\State For $s =0, \ldots, S-1$ 
			\State
			\quad If block$(i)==1$
			\State\quad \quad Update $\m{x}_i^{t,s+1}$ by \eqref{pu2}.
			\State
			\quad \quad
			Update $\m{H}_i^{t,s+1}$ by \eqref{Hit}.
			\State
			\quad Else
			\State
			\quad \quad $\m{x}_i^{t,s+1}=\m{x}_i^{t,s}$, and $\m{H}_i^{t,s+1}=\m{H}_i^{t,s}$.
			\State \quad End if.
			\State \quad If $\|\m{x}_i^{t,s+1}-\m{x}_i^{t}\| \le c\|\m{v}_i^{t}-\m{v}_i^{t-1}\| (0<c<1)$
			\State \quad \quad block$(i)=0$.
			\State \quad End if.
			\State End for.
		\end{algorithmic}
	\end{algorithm}
	
	We use GDPDM(K) and $\text{GDPDM(K)}^{+}$to denote GDPDM and $\text{GDPDM}^{+}$ with $S=K$, respectively.
	In addition, we update $r^t$  in \eqref{pi} as 
	\begin{equation}\label{rt}
		r^t=c_r (\eta_r)^t,
	\end{equation}
	where $c_r>0$ and $0<\eta_r<1$, and set $\overline{\omega} >> \underline{\omega} >0$ to be
	a large and a small positive constant, respectively. 
	Recall that $\tilde{p}_i^{t}$ is an approximation of $p^t$, which can be rewritten as
	\begin{equation*}
		p^{t} =\left(\frac{\|\bm{\lambda}^{t} - \bm{\lambda}^{t-1}\|_{\hat{\m{B}}^{t+1}}^2}{\|\bm{\lambda}^{t} - \bm{\lambda}^{t-1}\|^2}+r^t \right)^{-1},
	\end{equation*}
	where $\hat{\m{B}}^{t+1}={(\m{I} - \m{W})^{1/2}} (\m{B}^{t+1})^{-1} {(\m{I} - \m{W})^{1/2}}$. By \textbf{Assumptions} \textbf{\ref{as3}}, we have $\frac{1}{\rho \bar{\psi}+r^t}\leq p^t \leq\frac{1}{\sigma \psi +r^t}$. If $\tilde{p}_i^{t}$ tends to $p^{t}$, we would have $\tilde{p}_i^{t} \in [\frac{1}{\rho \bar{\psi}+r^t},\frac{1}{\sigma \psi +r^t}] \subset [\frac{1}{\overline{\omega}+r^t},\frac{1}{\underline{\omega}+r^t}]$.
	The success of each algorithm is measured by
	\begin{equation*}
		\mbox{Relative error} :=	\frac{1}{n}\sum_{i=1}^n\frac{\|\m{x}^t_i-\m{z}^*\|}{\|\m{z}^*\|+1},
	\end{equation*}
	where the true solution $\m{z}^*$ is explicitly obtained for the linear regression problem
	and is pre-computed by a centralized algorithm for the logistic regression problem.
	In our experiments, we would analyze the impact of networks with different condition numbers. 
	So we further introduce the communication volume which can be calculated as follows: 
	\begin{align*}\notag
		&\text{Communication volume}\\
		=&\text{~number of iterations}\\
		&\times \text{number of communication rounds per iteration}\\
		&\times \text{number of edges, i.e., }dn(n-1)/2\\
		&\times \text{dimension of transmitted vectors on each edge}.
	\end{align*}
	In all experiments, we set the number of nodes $n=10$. For all comparison algorithms, we initialize $\m{x}^0=\m{0}$ and set $\m{v}^0=\m{0}$ for our algorithms. 
	All experiments are coded in MATLAB R2017b and run on a laptop with Intel Core i5-9300H CPU, 16GB RAM, and Windows 10 operating system.
	
	\subsection{Linear Regression Problem}
	In this subsection, we investigate the impacts of the condition number of the objective function, denoted as $\kappa_f$, by comparing our algorithms with first-order algorithms, EXTRA\cite{shi2015extra} and GT\cite{qu2017harnessing}. We consider the following optimization problem
	\begin{equation}\label{linear_problem}
		\mathop {\min }\limits_{\m{z} \in {R^p}} \sum_{i=1}^n \frac{1}{2} \m{z}\tr \m{A}_i \m{z} + \m{b}_i\tr \m{z},
	\end{equation}
	where $\m{A}_i \in R^{p \times p}$ and $\m{b}_i \in R^p$ are private data available to node $i$. To control the condition number of problem \eqref{linear_problem}, we construct $\m{A}_i = \m{Q}\tr \operatorname {diag} \{ a_1,...,a_p \} \m{Q}$, where $\m{Q}$ is a random orthogonal matrix. We set $a_1=1$ and $a_p$ as an arbitrarily large number, 
	and generate $a_j \sim \operatorname{U}(1,2)$ for $j=2, \ldots, p-1$, where $\operatorname{U}(1,2)$ represents the uniform distribution from 1 to 2. So $\kappa_f=a_p/a_1=a_p$.
	
	We set $p=1000$ and edge density $d=0.36$ for the network, where $\kappa_g = 9.8$. We set $a_p$ = 10, 100, 1000, 10000. All the algorithm parameters are set for their better performance and are listed in Table \ref{table0}  where parameter notations follow the source papers.
		\begin{table}[H]
		\caption{Parameter settings for linear regression}\label{table0}
		\centering
		\begin{tabular}{c|c|c|c}
			\hline
			&\multirow{2}{*}{GT}  &\multirow{2}{*}{EXTRA }   &\multirow{2}{*}{\shortstack{GDPDM(1)(GDPDM(2);GDPDM(3);\\GDPDM(4);$\text{GDPDM(4)}^{+}$)}}  \\
			& & &\\
			\hline
			$\kappa_f=10$&$\backslash$   &$\backslash $  & \multirow{4}{*}{\shortstack{$\beta=0.49(0.31;0.22;0.17;0.57)$, $\omega=0.32$, \\ $\gamma=1$, $\alpha=2.8$, $r^t=0.95^t$, $c=0.4$ }} \\
			\cline{1-3}
			$\kappa_f=100$&$\eta=5\times 10^{-3}$   &$\alpha= 10^{-2}$  & \\
			\cline{1-3}
			$\kappa_f=1000$&$\eta=5\times 10^{-4}$    &$\alpha= 10^{-3}$  &  \\
			\cline{1-3}
			$\kappa_f=10000$& $\eta=5\times 10^{-5}$  &$\alpha= 10^{-4}$  &    \\
			\hline
		\end{tabular}
	\end{table}

	From figure \ref{GDPDM1}, We find that for such simple quadratic problems, GDPDM(1), namely DPDM is most efficient since the outer iteration number has no significant change 
	but the communication cost and CPU time increase as the number of inner iterations increases. 
	In this case, one BFGS iteration is sufficient to minimize the quadratic primal problem.
	
	
	
	Figure \ref{condition_number} shows that our algorithm is more robust to the problem condition number than EXTRA and GT. The convergence rate of first-order algorithms becomes obviously slow when the condition number increases.

	
	\begin{figure}[H]
		\centering
		\subfloat[]{\includegraphics[width=3.5in]{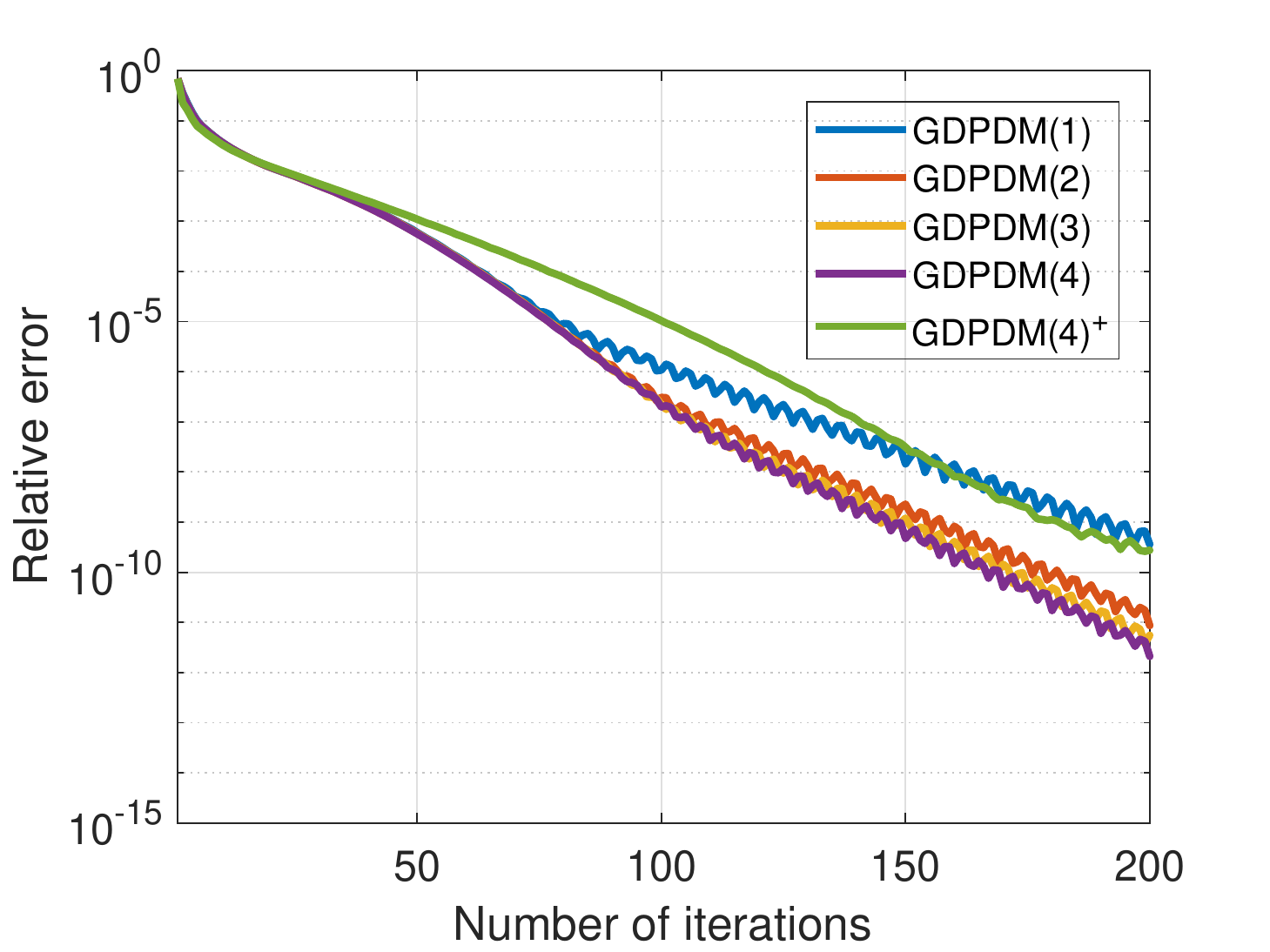}%
			\label{fig_1_a}}
		\hfil
		\subfloat[]{\includegraphics[width=3.5in]{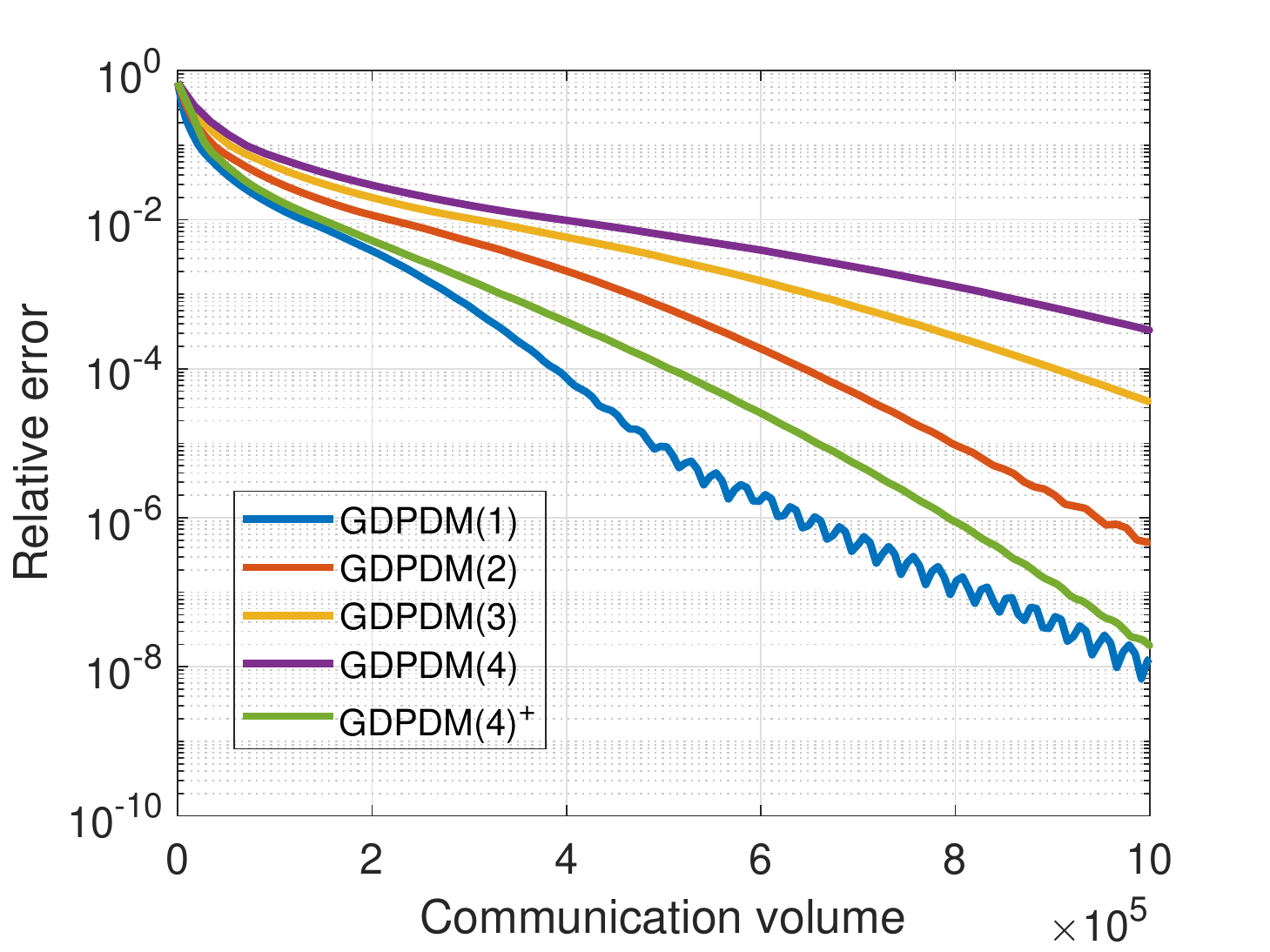}%
			\label{fig_1_b}}
		\caption{Comparisons among GDPDMs in terms of the iteration number and communication volume for $\kappa_f=10$.}
		\label{GDPDM1}
	\end{figure}
	
	\begin{figure}[H]
		\centering
		\subfloat[]{\includegraphics[width=3.5in]{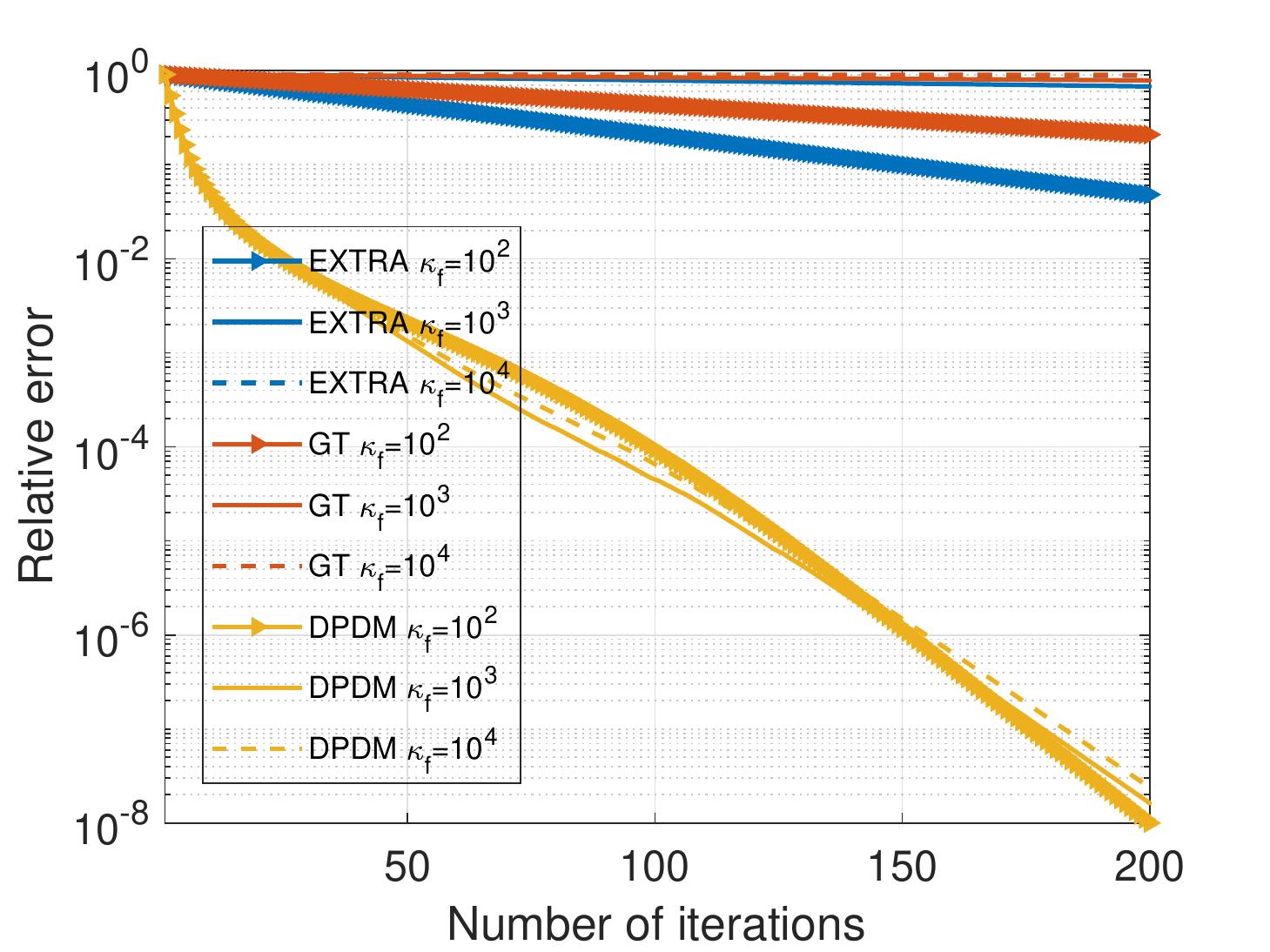}%
			\label{fig_2_a}}
		\hfil
		\subfloat[]{\includegraphics[width=3.5in]{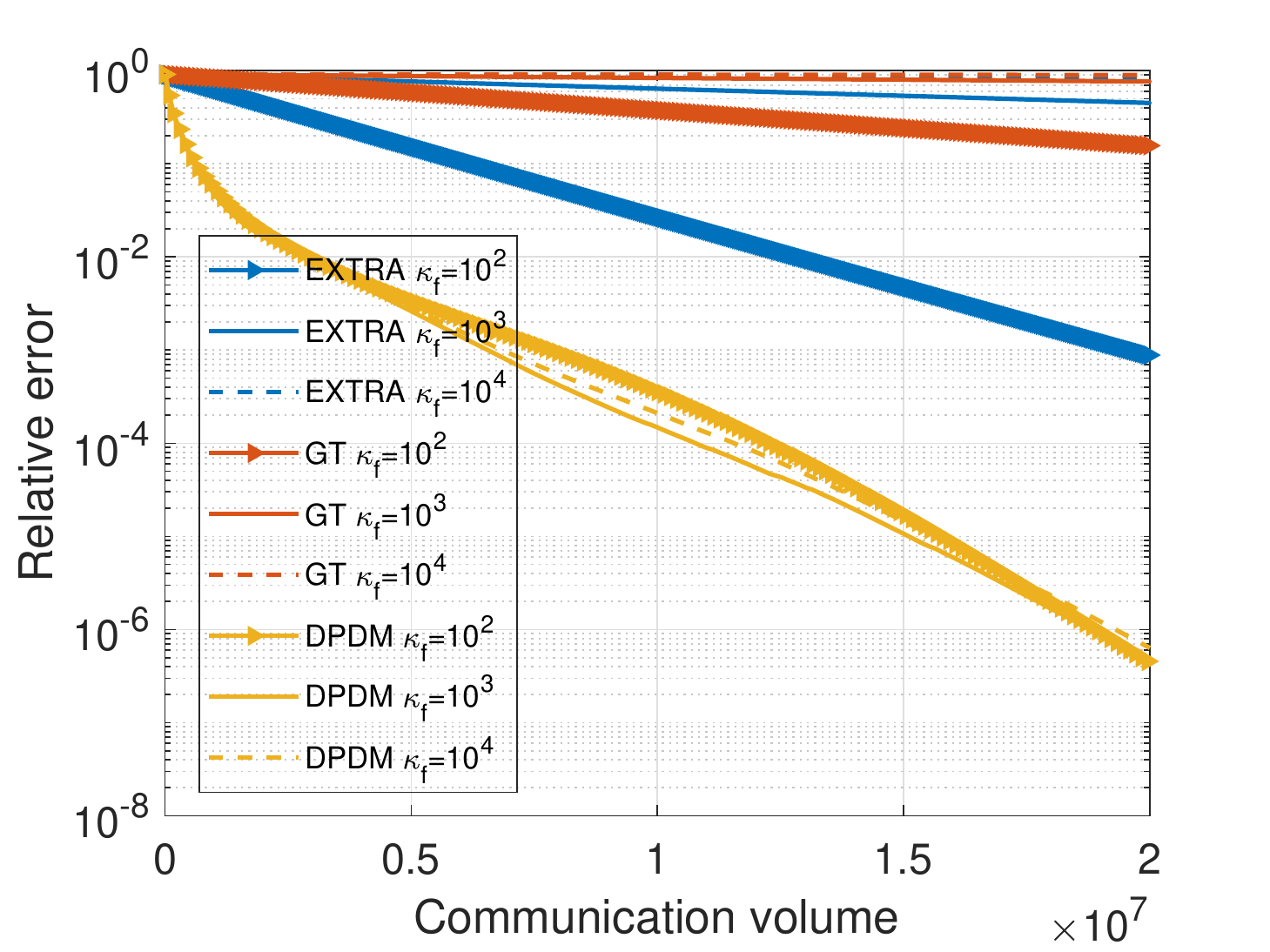}%
			\label{fig_2_b}}
		\hfil
		\subfloat[]{\includegraphics[width=3.5in]{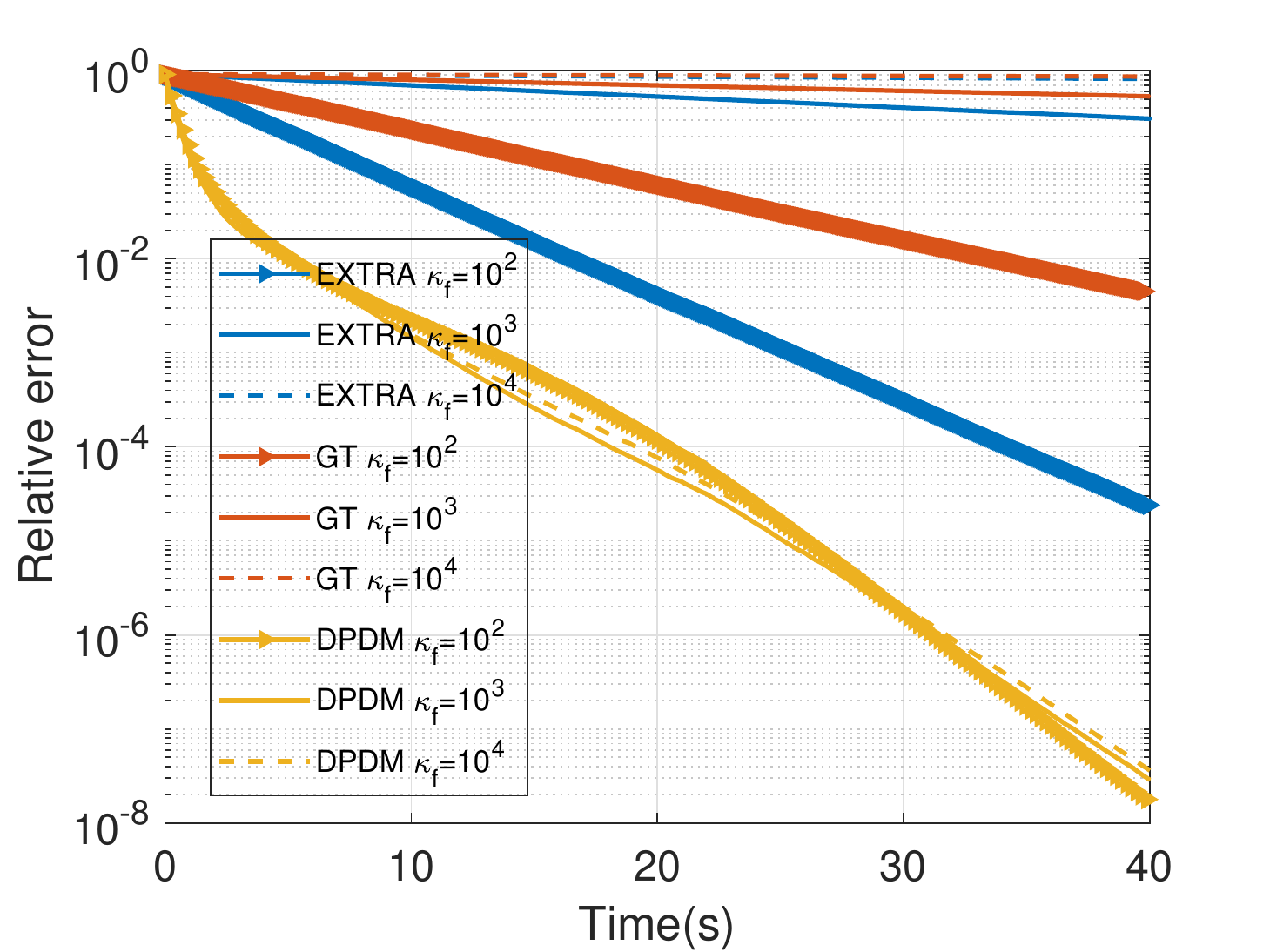}%
			\label{fig_2_c}}
		\caption{Comparisons with decentralized first-order algorithms
			in terms of the iteration number, communication volume, and CPU time (in seconds).}
		\label{condition_number}
	\end{figure}

	\subsection{Logistic Regression Problem}
	In this subsection, we firstly compare our algorithms with  GT \cite{qu2017harnessing} and two second-order algorithms: ESOM\cite{mokhtari2016decentralized} and NT\cite{zhang2021newton}, which use Hessian in the primal domain with dual ascent in the dual domain. We consider the logistic regression 
	\begin{equation}\label{logistic_problem}
		\mathop {\min }\limits_{\m{z} \in {R^p}} \sum_{i=1}^n  \sum_{j=1}^{n_i} \log \left(1+\exp (-b_{ij} \m{a}_{ij}\tr \m{z} ) \right)+\frac{ \hat{\lambda}}{2}\|\m{z}\|^2,
	\end{equation}
	where $\m{a}_{ij} \in R^p$ are the feature vectors and $ b_{ij} \in \{-1,+1\}$ are the labels. The experiments are conducted on four datasets
	from the LIBSVM library: \textbf{mushroom}, \textbf{ijcnn1}, \textbf{w8a} and \textbf{a9a}.
	The edge density $d=0.36$ and the regularization parameter $\hat{\lambda}=1$. 
	All the algorithm parameters are set for their better performance and are listed in Table \ref{table2} where parameter notations follow the source papers.
	
	
	
		\begin{table}[H]
		\scriptsize
		\caption{Parameter settings for logistic regression with decentralized first-order and second-order algorithms and our GDPDMS}\label{table2}
		\centering
		\begin{tabular}{c|c|c|c|c}
			\hline
			&GT  &ESOM  &NT  &GDPDM(1)(GDPDM(2);GDPDM(4);$\text{GDPDM(4)}^{+}$)  \\
			\hline
			\multirow{4}{*}{\shortstack{\textbf{mushroom}\\ \# of samples ($\sum_{i=1}^n n_i=8120$)\\\# of features ($p=112$)}}&\multirow{4}{*}{$\eta=0.08$}  &\multirow{4}{*}{$\epsilon=0.3$, $\alpha=3.9$}  &\multirow{4}{*}{$\epsilon=3.1$, $\alpha=3.9$}  &\multirow{4}{*}{\shortstack{$\beta=0.48(0.3;0.17;0.51)$, $\omega=0.18$, $\gamma=1$,\\ $\alpha=3.6$, $r^t=0.95^t$, $c=0.6$}}   \\
			&  &  &  &  \\
			&  &  &  &  \\
			&  &  &  &  \\
			\hline
			\multirow{4}{*}{\shortstack{\textbf{ijcnn1}\\ \# of samples ($\sum_{i=1}^n n_i=49990$)\\\# of features ($p=22$)}}&\multirow{4}{*}{$\eta=0.06$ } &\multirow{4}{*}{$\epsilon=0.1$, $\alpha=3.9$}  &\multirow{4}{*}{$\epsilon=3.5$, $\alpha=3.9$}  &\multirow{4}{*}{\shortstack{$\beta=0.41(0.33;0.22;0.44)$, $\omega=0.15$, $\gamma=1$,\\ $\alpha=4$, $r^t=0.95^t$, $c=0.3$}}  \\
			&  &  &  &  \\
			&  &  &  &  \\
			&  &  &  &  \\
			\hline
			\multirow{4}{*}{\shortstack{\textbf{w8a}\\ \# of samples ($\sum_{i=1}^n n_i=49740$)\\\# of features ($p=300$)}}&\multirow{4}{*}{$\eta=0.06$ } &\multirow{4}{*}{$\epsilon=0.2$, $\alpha=4.1$  } &\multirow{4}{*}{$\epsilon=3.4$, $\alpha=4.0$ } &\multirow{4}{*}{\shortstack{$\beta=0.47(0.38;0.22;0.45)$, $\omega=0.17(0.17;0.17;0.16)$,\\ $\gamma=1$, $\alpha=3.6$, $r^t=0.95^t$, $c=0.6$ }}   \\
			&  &  &  & \\
			&  &  &  &  \\
			&  &  &  &  \\
			\hline
			\multirow{4}{*}{\shortstack{\textbf{a9a}\\ \# of samples ($\sum_{i=1}^n n_i=32560$)\\\# of features ($p=123$)}}& \multirow{4}{*}{$\eta=0.06$ }& \multirow{4}{*}{$\epsilon=0.1$, $\alpha=3.9$ } &\multirow{4}{*}{$\epsilon=3.3$, $\alpha=3.9$ } & \multirow{4}{*}{\shortstack{$\beta=0.38(0.34;0.19;0.42)$, $\omega=0.15$, $\gamma=1$, \\ $\alpha=4.0$, $r^t=0.95^t$, $c=0.45$ }}    \\
			&  &  &  &  \\
			&  &  &  &  \\
			&  &  &  &  \\
			\hline
		\end{tabular}
	\end{table}
	
	Figures \ref{mush}, \ref{ijcnn}, \ref{w8a} and \ref{a9a} show that our algorithms converge significantly faster than  GT, ESOM, and NT in terms of both the iteration number and CPU time.
	As more primal updates are allowed in each inner iteration of GDPDMs, the (outer) iteration number is reduced while the CPU time is often increased.
	However, it can be seen that $\text{GDPDM(4)}^{+}$, which adaptively controls the number of inner iterations for solving the more complex and nonlinear logistic regression problem, 
	performs generally best and keeps a good balance between the iteration number and CPU time. We can see that the performance of
	$\text{GDPDM(4)}^{+}$ is close to GDPDM(4) in terms of the number of iterations and has about the same low communication and time cost as GDPDM(1).
	
	On the other hand, since multiple primal updates in each iteration would lead to more communication cost, GDPDMs do not outperform the second-order method NT in terms of communication volume. For saving communication in NT, the topology-dependent term $\alpha(\m{I}-\m{W})$ is removed when involving matrix inverse calculation as shown in \eqref{dNT}. This renders NT needing a big regularization parameter $\epsilon$ to compensate for the loss of network topology information. And $\epsilon$ is required to be bounded below for convergence of NT, which could also affect the fast local convergence.  Most GDPDMs are more communication-efficient than the second-order method ESOM, while all GDPDMs are superior to the first-order method GT. 
	Moreover, ESOM and NT obviously perform significantly worse in terms of CPU time, since exact Hessian and matrix inversions are calculated in their iterations. Figures \ref{fig_5_d}, \ref{fig_6_d}, \ref{fig_7_d} and \ref{fig_8_d} show the needed time and communication volume for reaching a $10^{-10}$-accuracy solution. 
	We think it may not be an ideal way to trade much time efficiency for a slight improvement in communication. Our GDPDMs, such as GDPDM(1), $\text{GDPDM(4)}^{+}$ are located near the bottom left corner, implying they 
	are quite efficient in both communication cost and CPU time.

	\begin{figure}[H]
		\centering
		\subfloat[]{\includegraphics[width=3.5in]{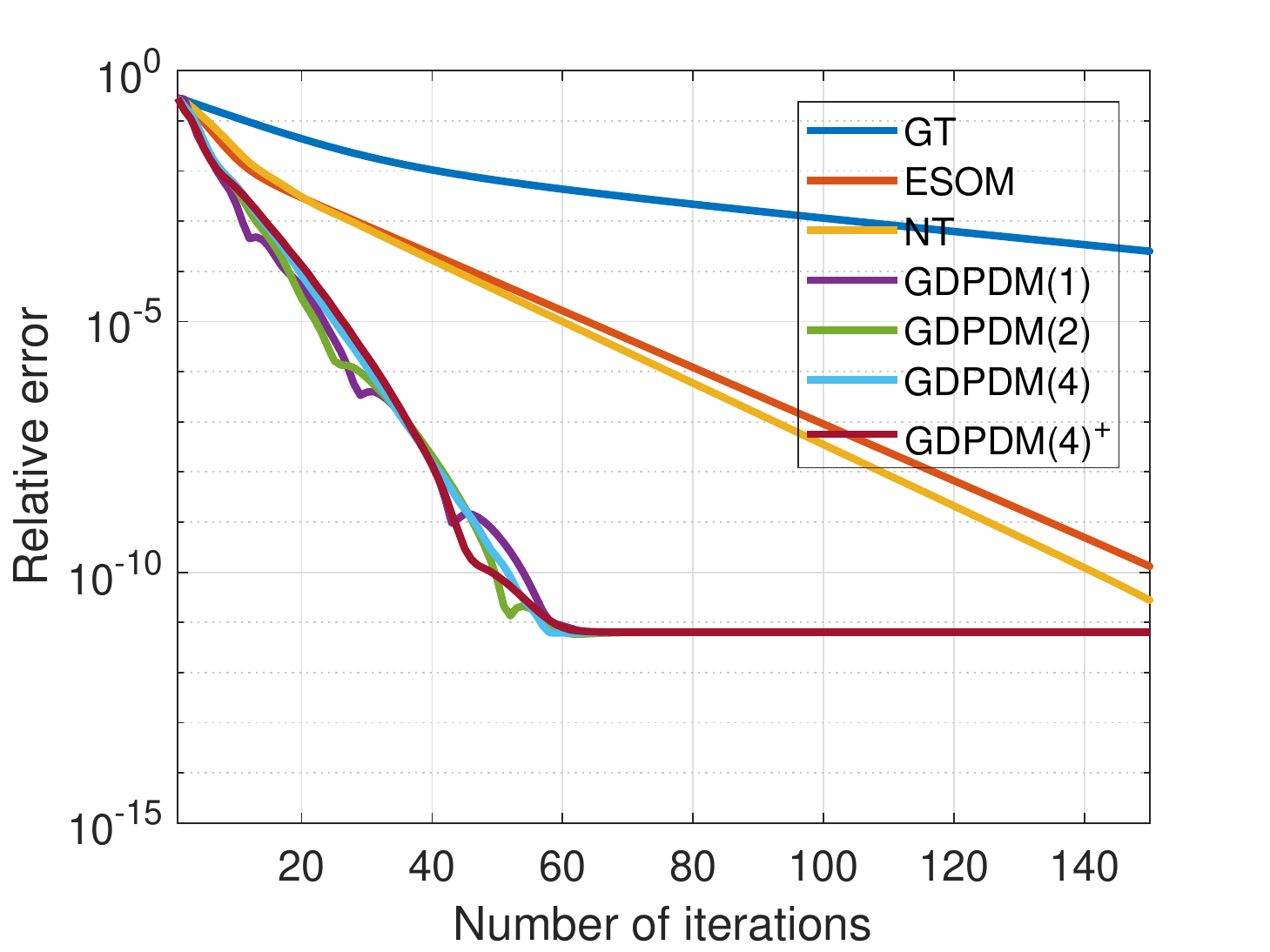}%
			\label{fig_5_a}}
		\hfil
		\subfloat[]{\includegraphics[width=3.5in]{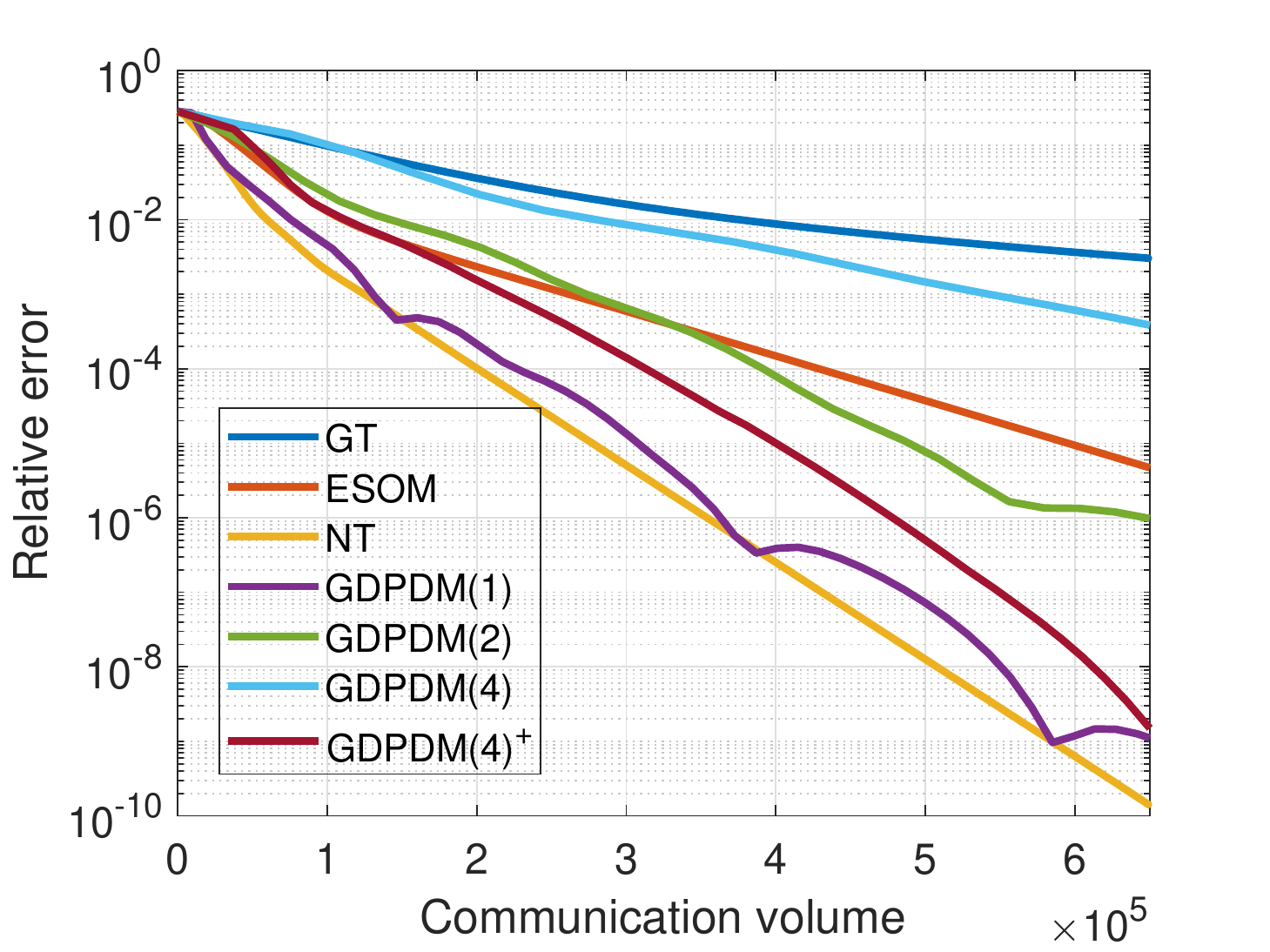}%
			\label{fig_5_b}}
		\hfil
		\subfloat[]{\includegraphics[width=3.5in]{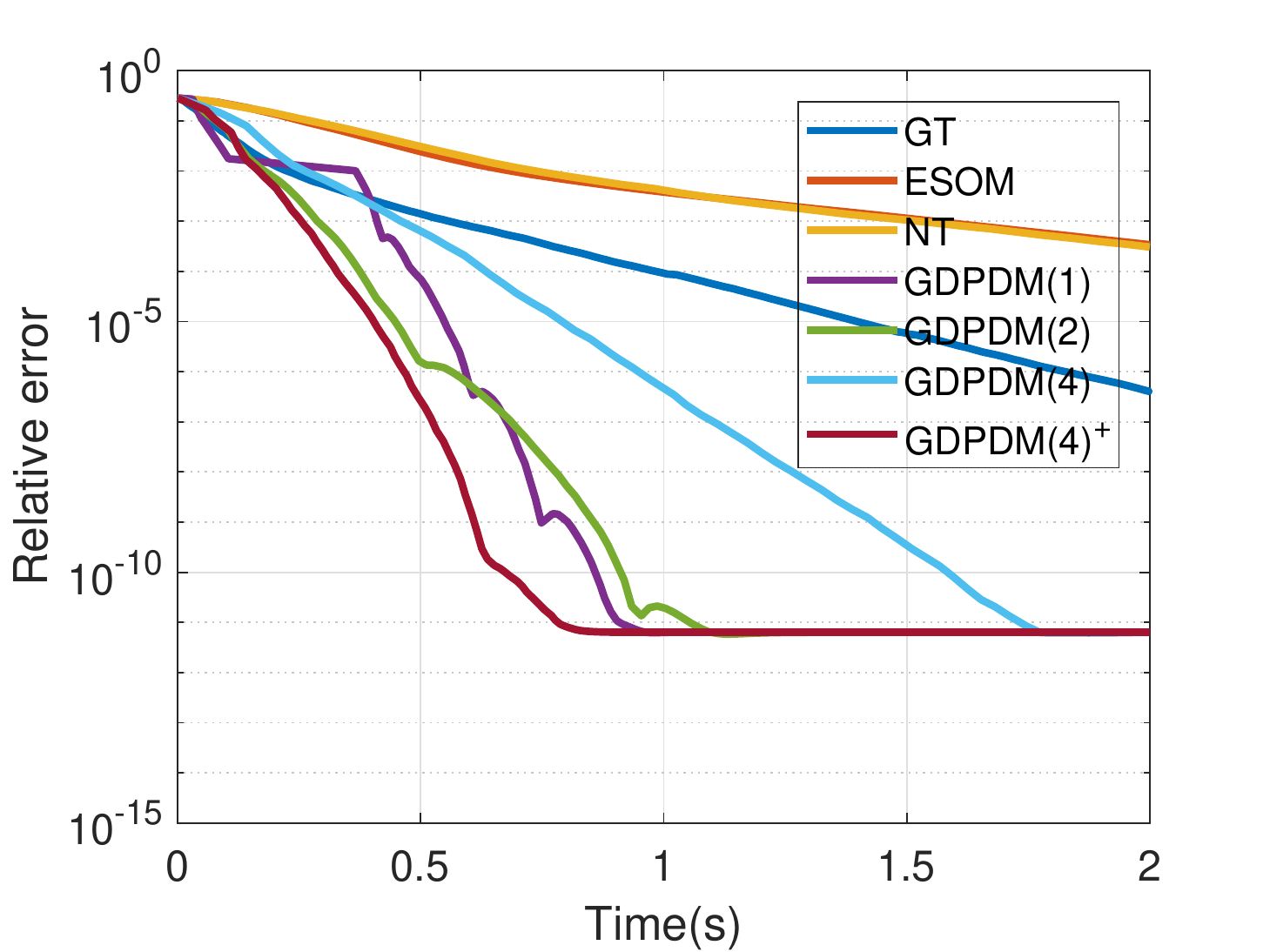}%
			\label{fig_5_c}}
		\hfil
		\subfloat[]{\includegraphics[width=3.5in]{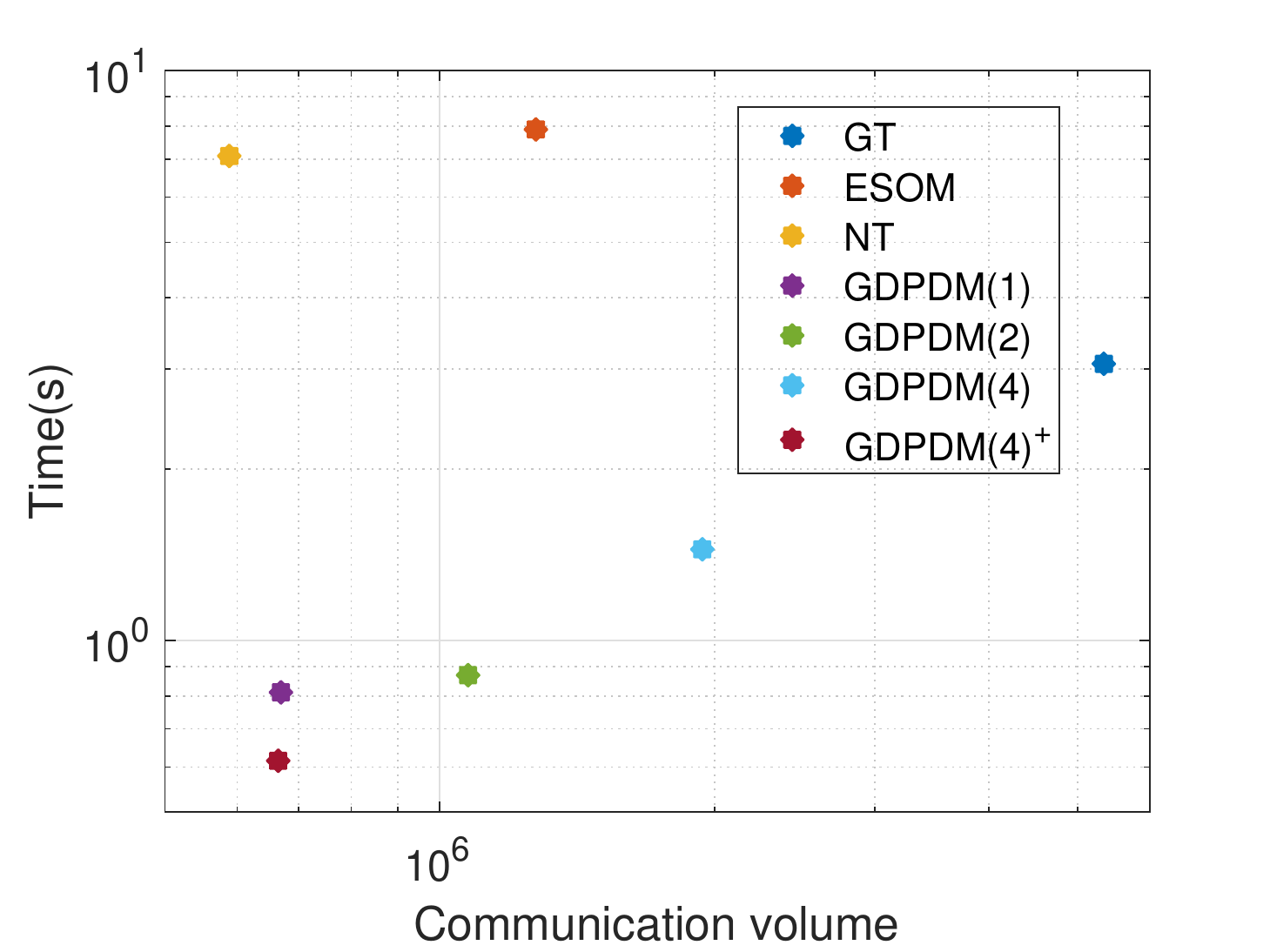}%
			\label{fig_5_d}}
		\caption{(a-c) Comparisons with decentralized second-order algorithms 
			in terms of the iteration number, communication volume, and CPU time (in seconds) using \textbf{mushroom} dataset. (d) Balance between time and communication volume.}
		\label{mush}
	\end{figure}
	
	\begin{figure}[H]
		\centering
		\subfloat[]{\includegraphics[width=3.5in]{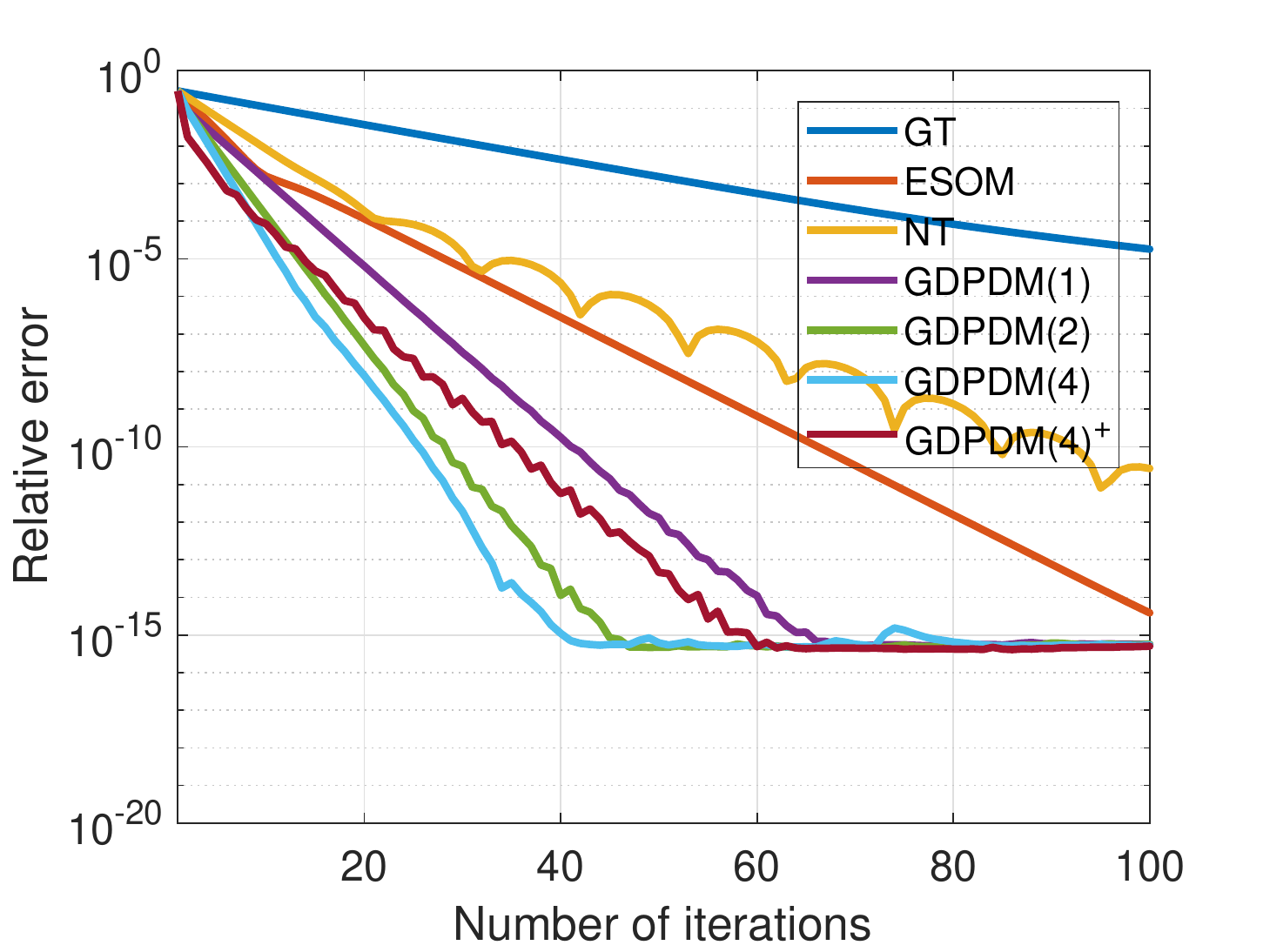}%
			\label{fig_6_a}}
		\hfil
		\subfloat[]{\includegraphics[width=3.5in]{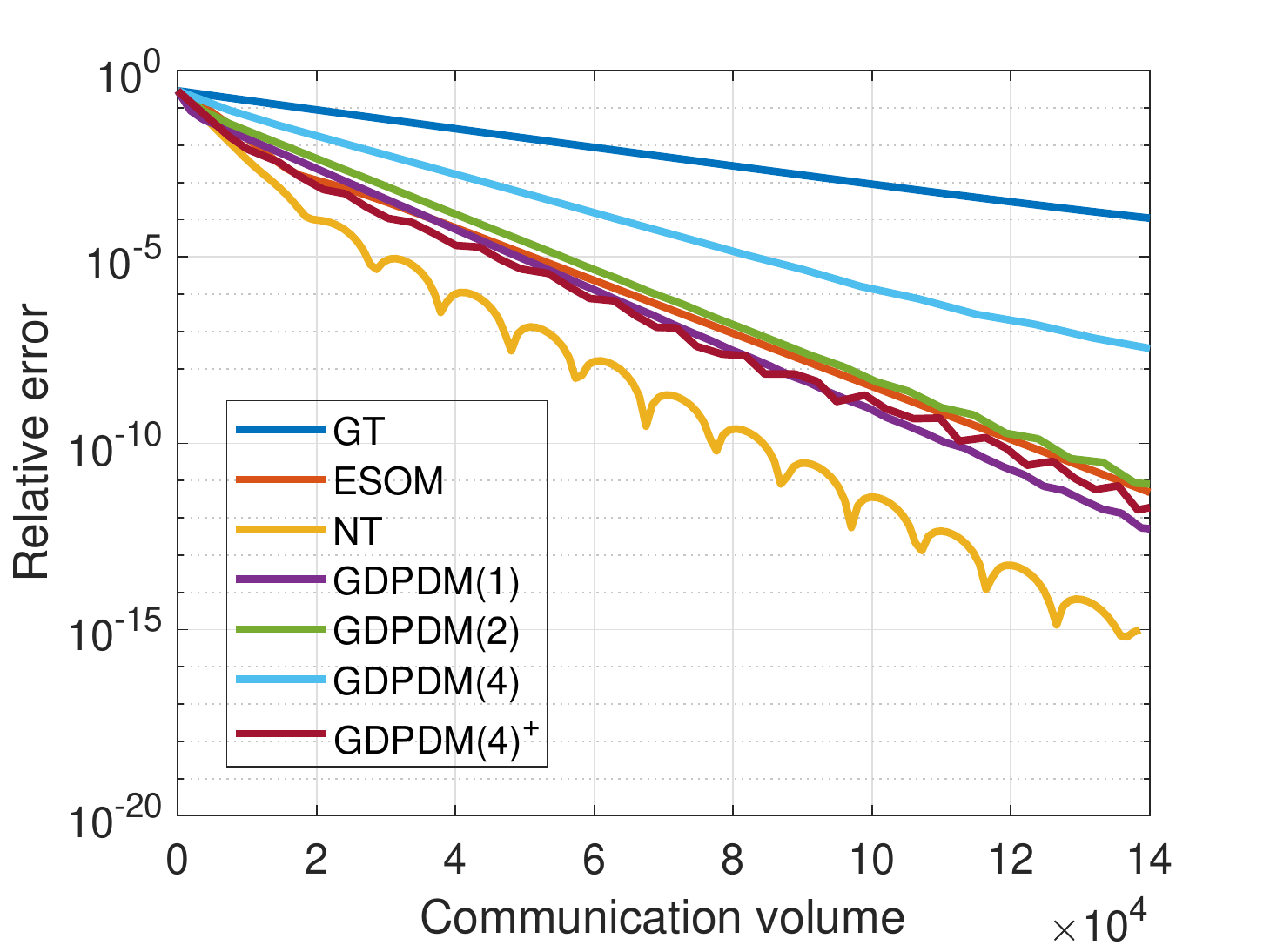}%
			\label{fig_6_b}}
		\hfil
		\subfloat[]{\includegraphics[width=3.55in]{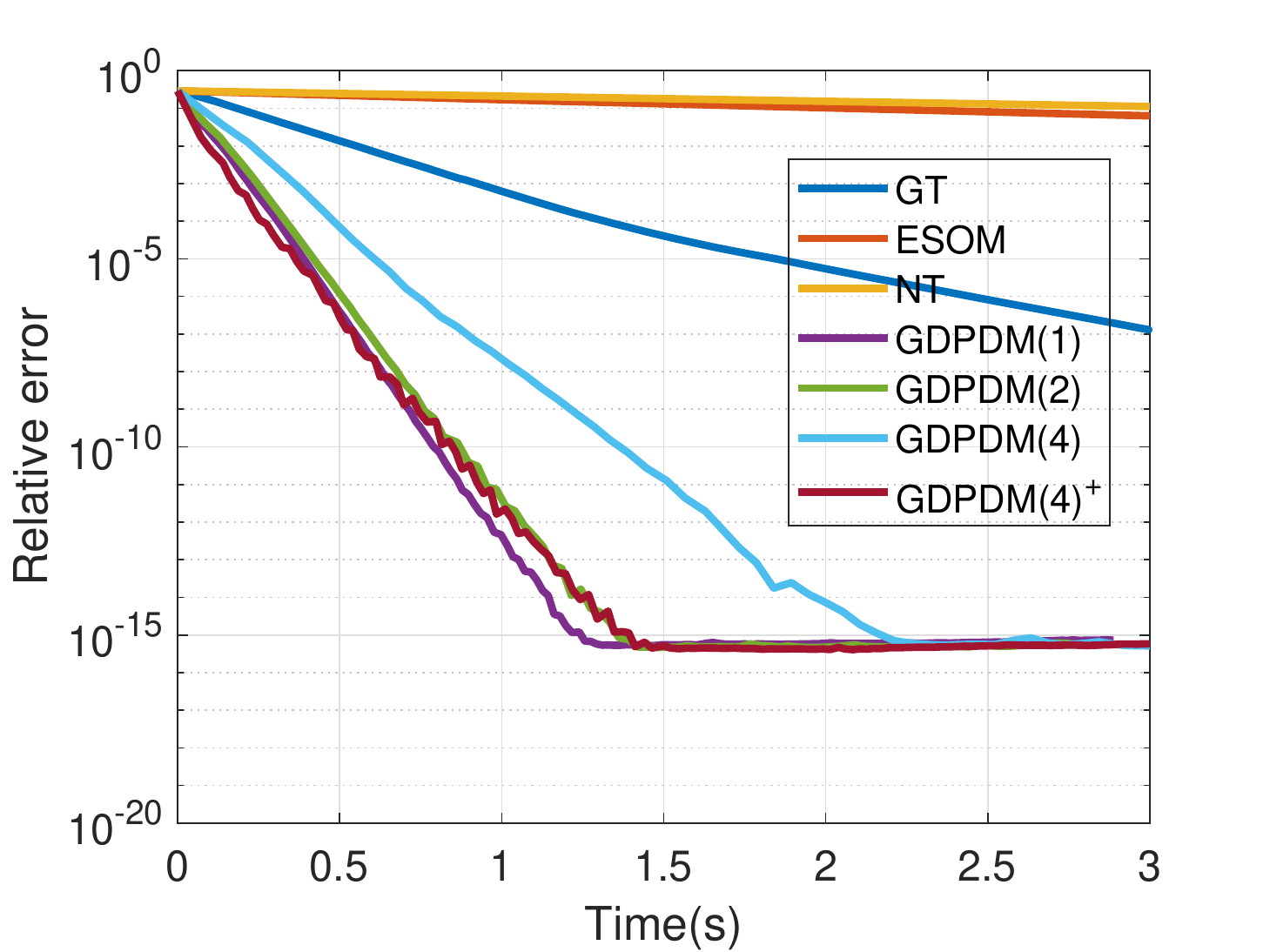}%
			\label{fig_6_c}}
		\hfil
		\subfloat[]{\includegraphics[width=3.5in]{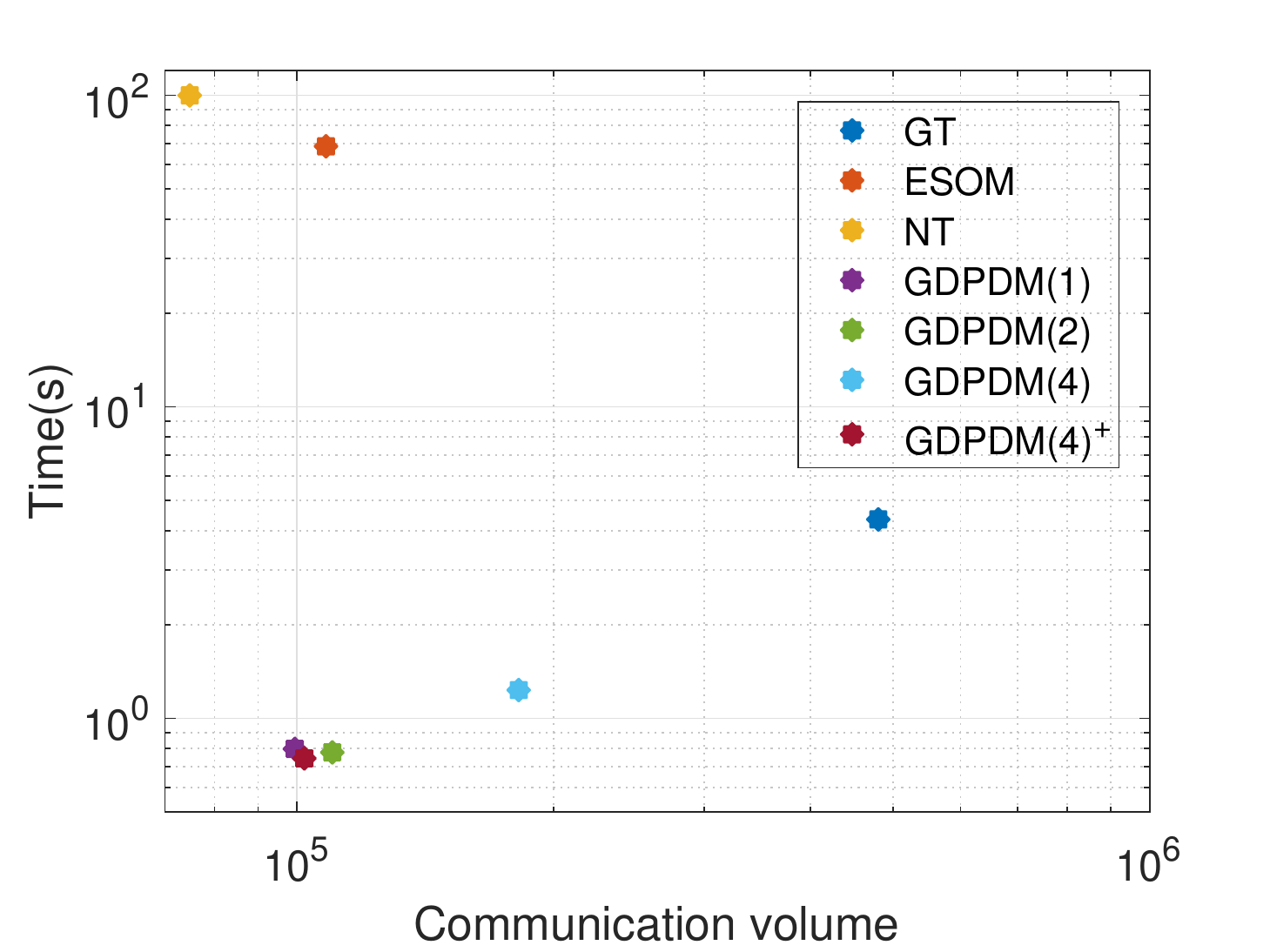}%
			\label{fig_6_d}}
		\caption{(a-c) Comparisons with decentralized second-order algorithms in terms of the iteration number, communication volume, and CPU time (in seconds) using \textbf{ijcnn1} dataset. (d) Balance between time and communication volume.}
		\label{ijcnn}
	\end{figure}

	\begin{figure}[H]
		\centering
		\subfloat[]{\includegraphics[width=3.5in]{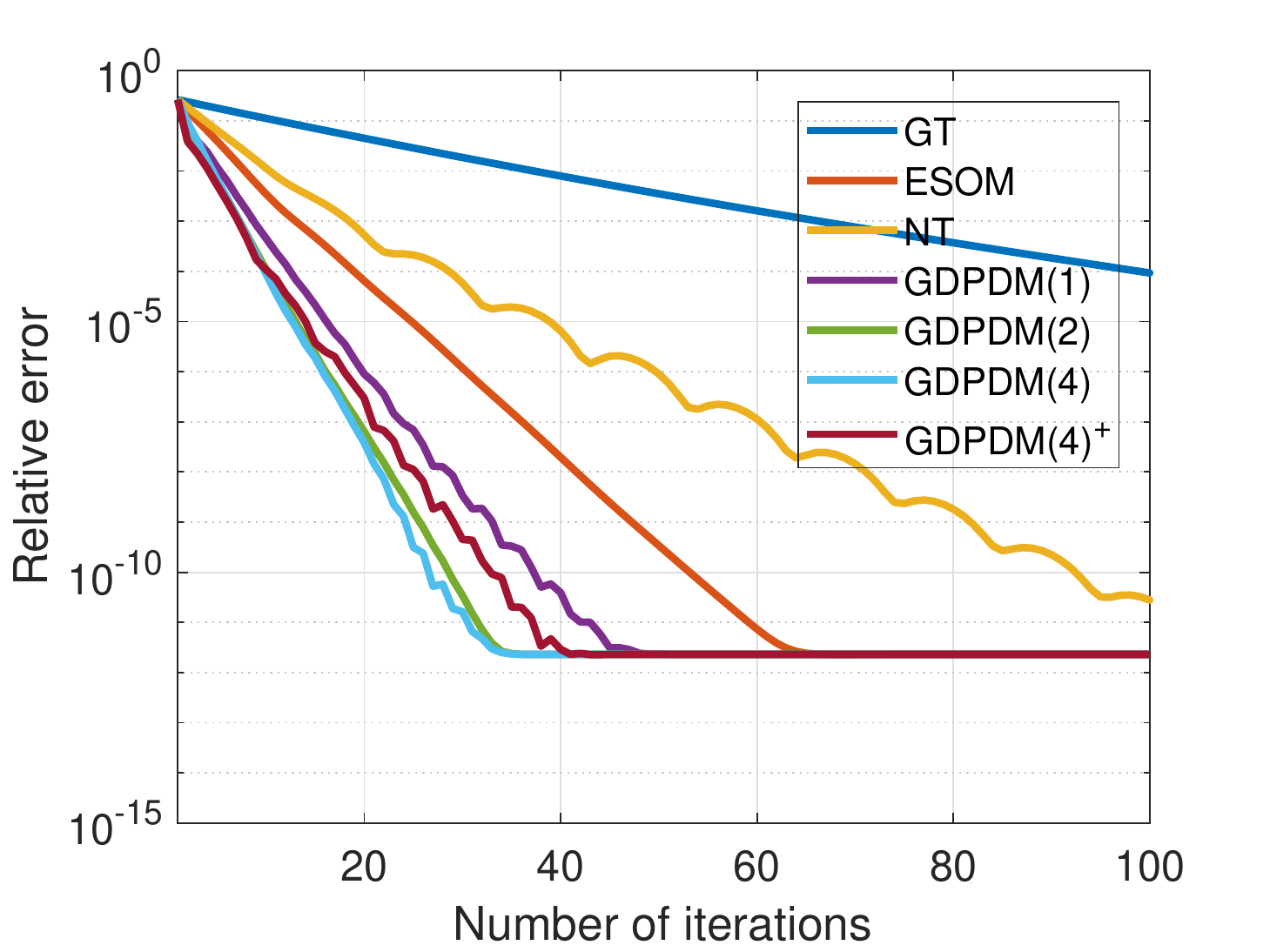}%
			\label{fig_7_a}}
		\hfil
		\subfloat[]{\includegraphics[width=3.5in]{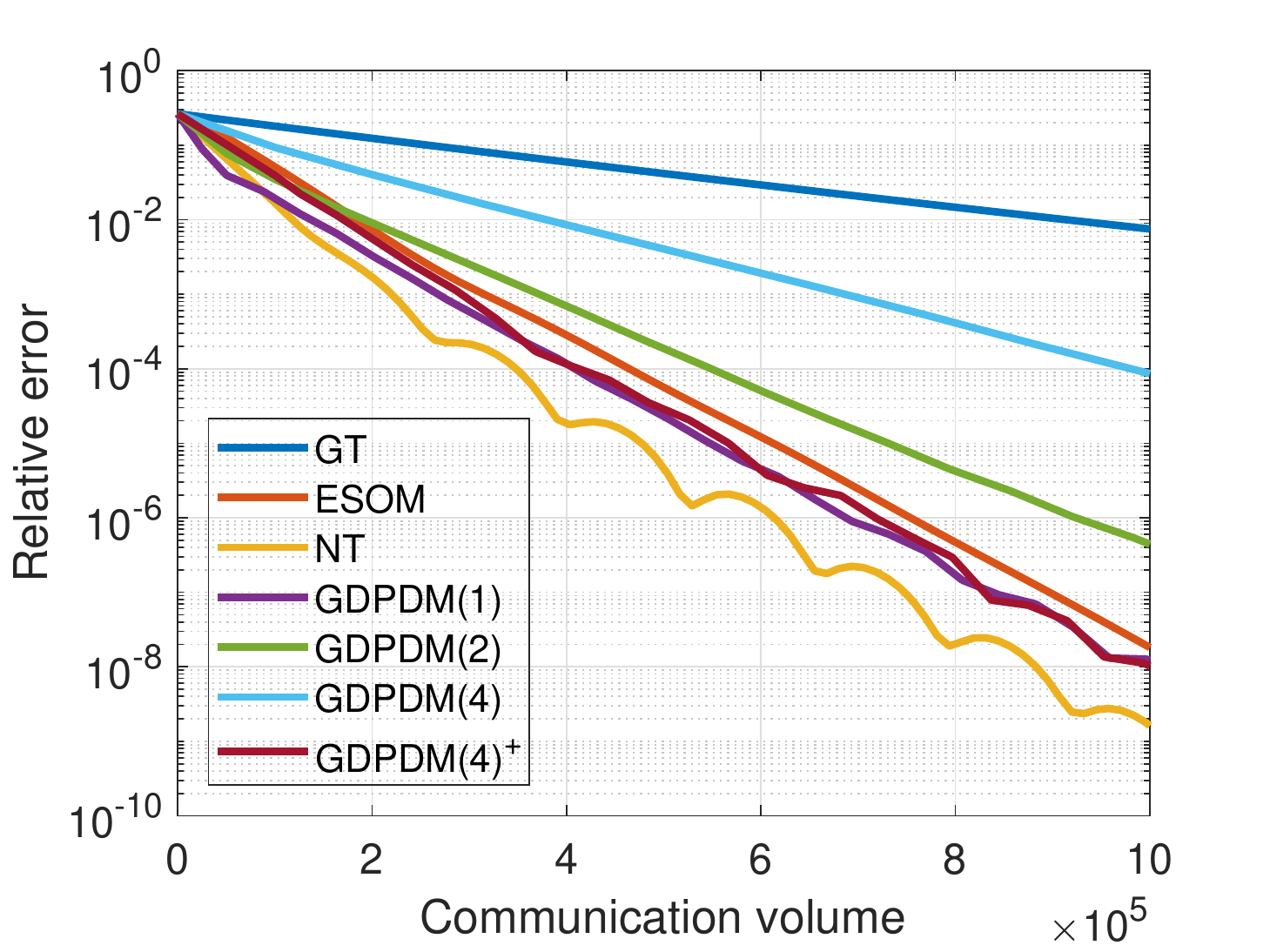}%
			\label{fig_7_b}}
		\hfil
		\subfloat[]{\includegraphics[width=3.5in]{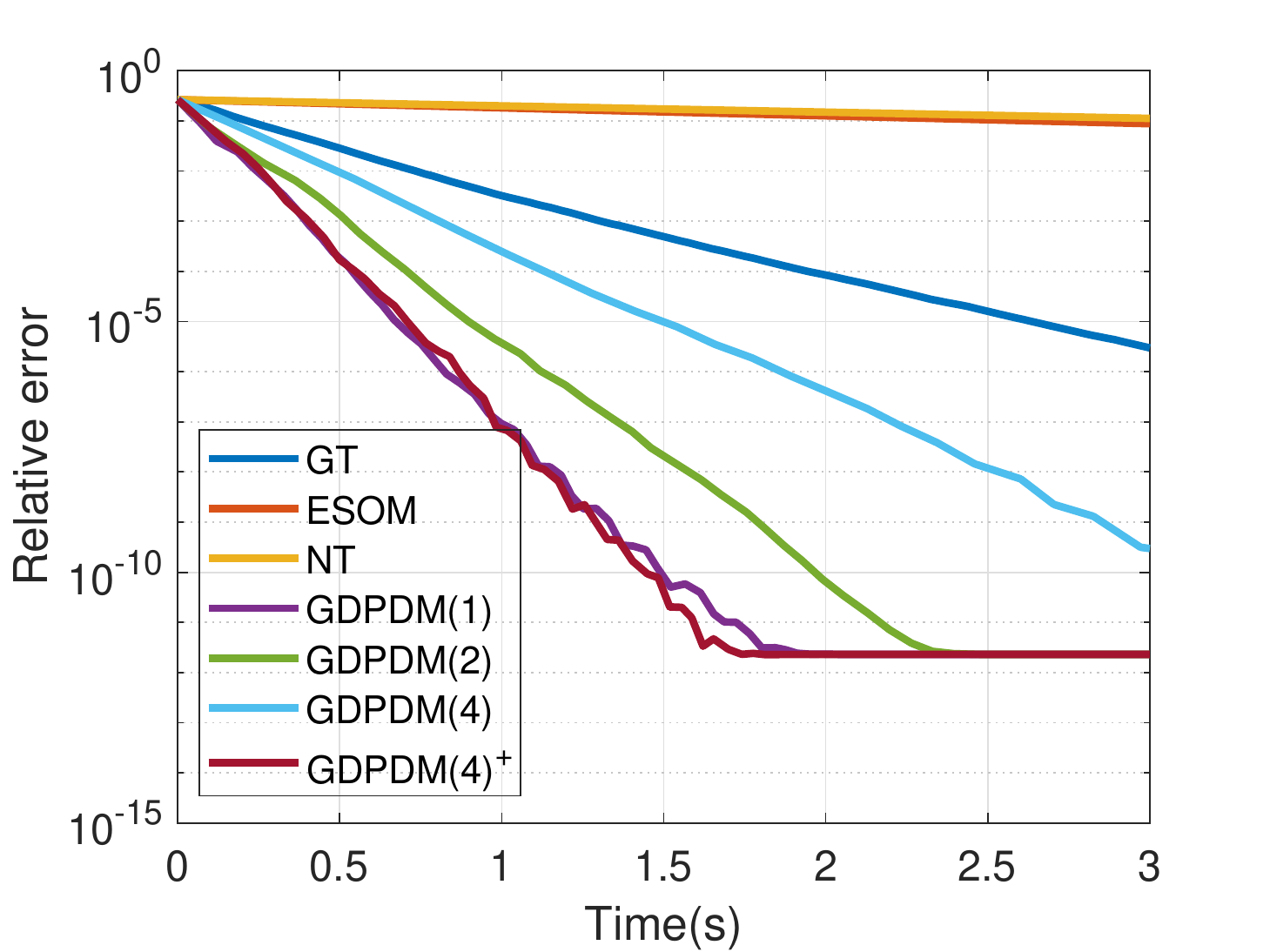}%
			\label{fig_7_c}}
		\hfil
		\subfloat[]{\includegraphics[width=3.5in]{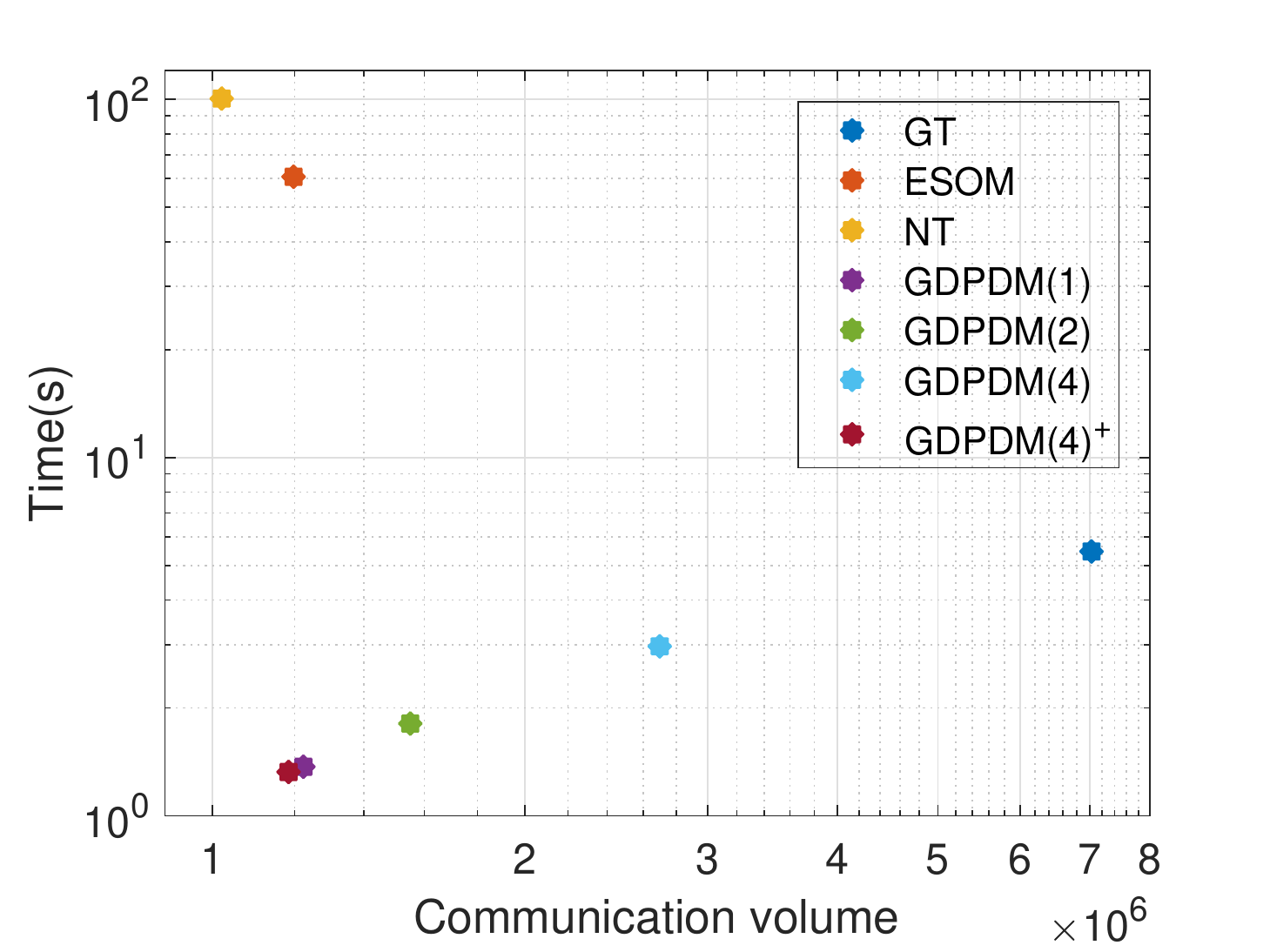}
			\label{fig_7_d}}
		\caption{(a-c) Comparisons with decentralized second-order algorithms in terms of the iteration number, communication volume, and CPU time (in seconds) using \textbf{w8a} dataset. (d) Balance between time and communication volume.}
		\label{w8a}
	\end{figure}
	
	\begin{figure}[H]
		\centering
		\subfloat[]{\includegraphics[width=3.5in]{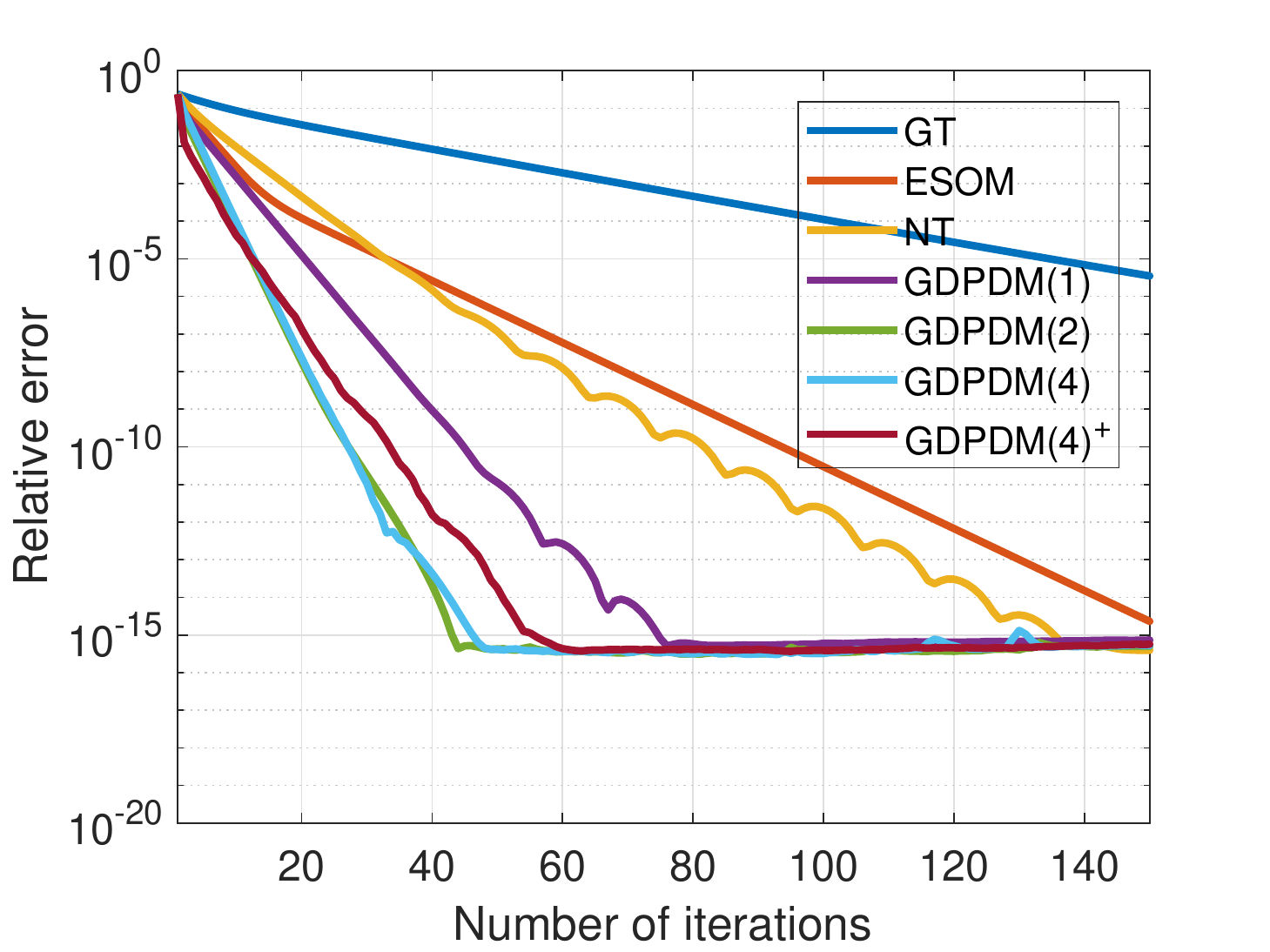}%
			\label{fig_8_a}}
		\hfil
		\subfloat[]{\includegraphics[width=3.5in]{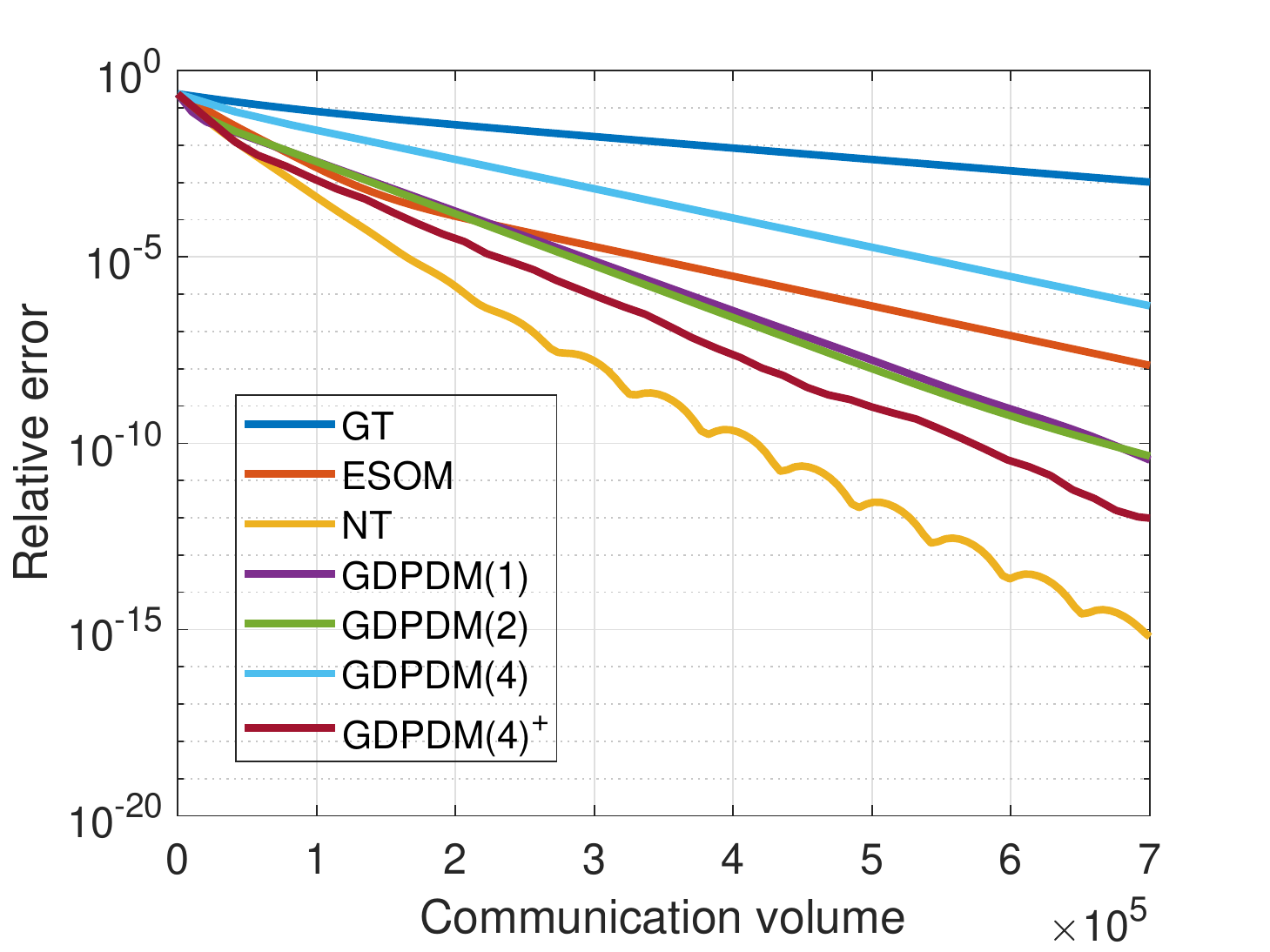}%
			\label{fig_8_b}}
		\hfil
		\subfloat[]{\includegraphics[width=3.5in]{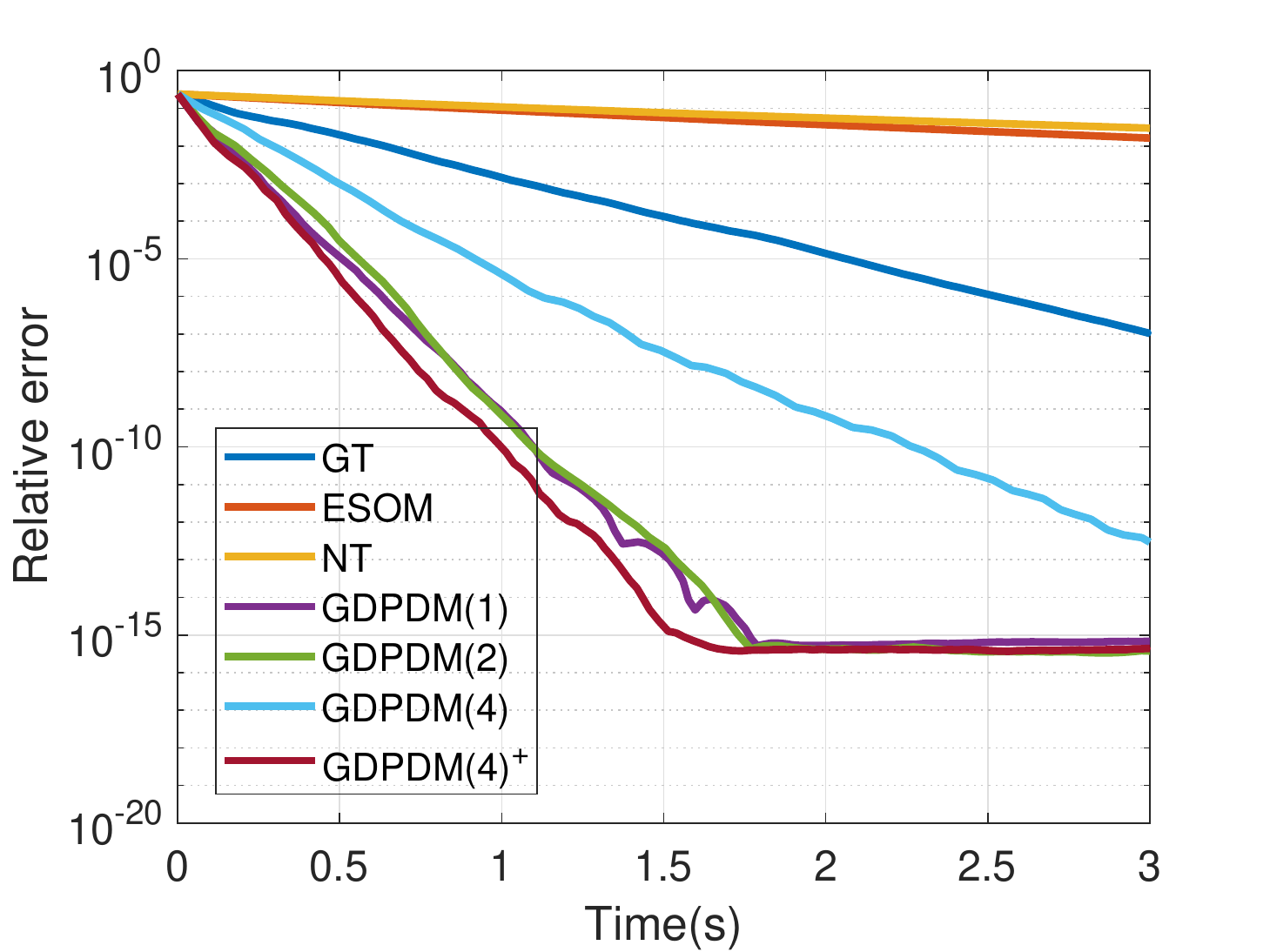}%
			\label{fig_8_c}}
		\hfil
		\subfloat[]{\includegraphics[width=3.5in]{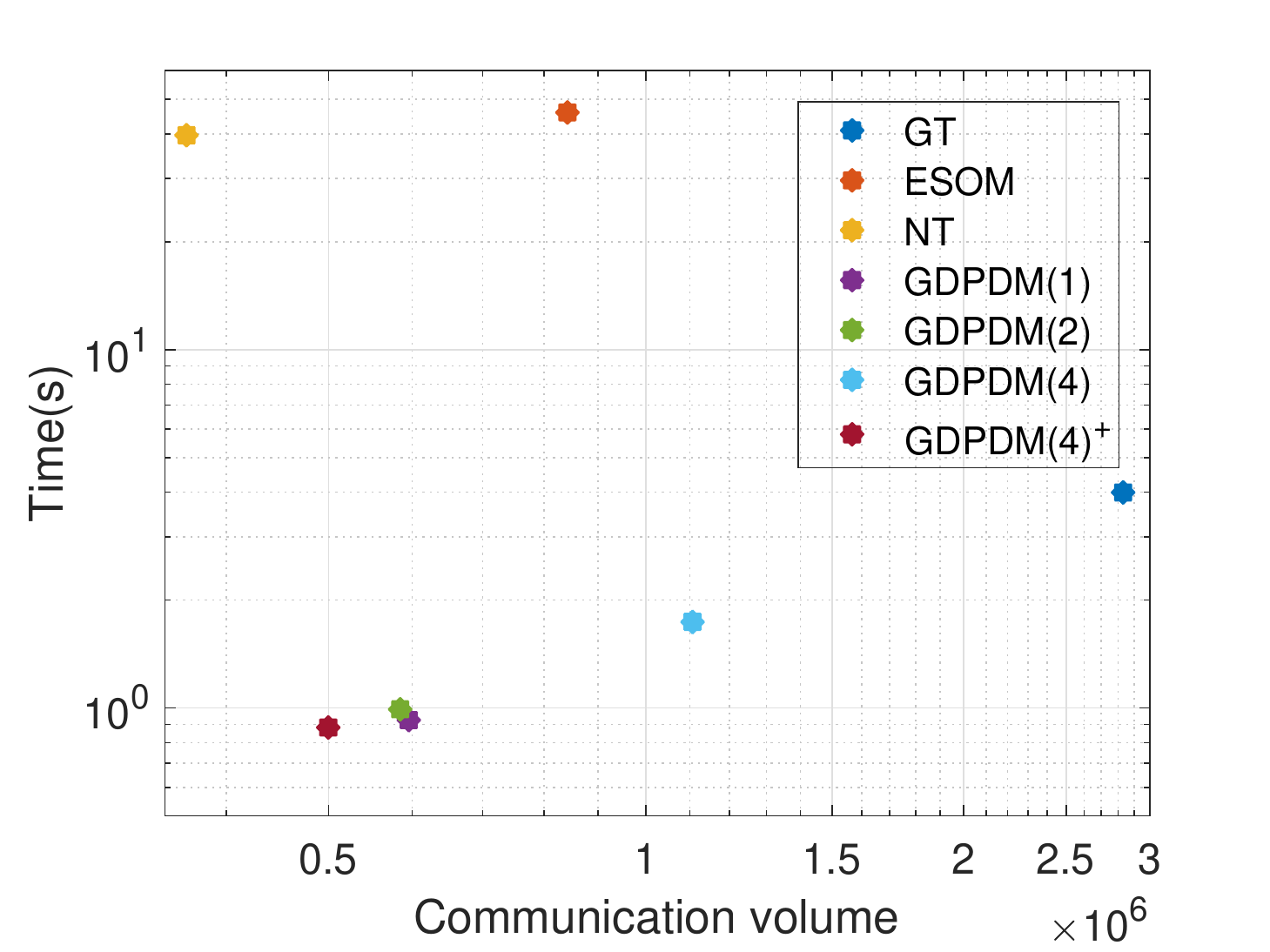}
			\label{fig_8_d}}
		\caption{(a-c) Comparisons with decentralized second-order algorithms in terms of the iteration number, communication volume, and CPU time (in seconds) using \textbf{a9a} dataset. (d) Balance between time and communication volume.}
		\label{a9a}
	\end{figure}
	
	We now compare our algorithms with several well-developed decentralized quasi-Newton algorithms: DBFGS \cite{eisen2017decentralized}, Damped regularized limited-memory DFP(called DR-LM-DFP below) \cite{zhang2023variance}, Damped limited-memory BFGS(called D-LM-BFGS below) \cite{zhang2023variance}.  All the algorithm parameters are set for their better performance and are listed in Table \ref{table3} where parameter notations follow the source papers.
		\begin{table}[H]
		\caption{Parameter settings for logistic regression with decentralized quasi-Newton algorithms}\label{table3}
		\centering
		\begin{tabular}{c|c|c|c}
			\hline
			&DBFGS  &DR-LM-DFP  &D-LM-BFGS  \\
			\hline
			\multirow{3}{*}{\textbf{ijcnn1}}&\multirow{3}{*}{\shortstack{$\alpha=0.01$, $\epsilon=0.02$,\\$\gamma=0.1$, $\Gamma=0.1$}}   &\multirow{3}{*}{\shortstack{$\alpha=0.04$, $\rho=0.03$, $\epsilon=10^{-3}$,\\$\beta=1$, $\C{B}=10^4$, $\tilde{L}=1$, $M=50$}}   &\multirow{3}{*}{\shortstack{$\alpha=0.48$, $\epsilon=10^{-3}$, $\beta=10^{-3}$,\\ $\C{B}=10^4$, $\tilde{L}=20$, $M=8$}}  \\
			& & &\\
			& & &\\
			\hline
			\multirow{3}{*}{\textbf{a9a}}&\multirow{3}{*}{\shortstack{$\alpha=0.01$, $\epsilon=0.03$,\\$\gamma=0.25$, $\Gamma=0.05$}}   &\multirow{3}{*}{\shortstack{$\alpha=0.04$, $\rho=0.03$, $\epsilon=5 \times 10^{-3}$,\\$\beta=10^{-2}$, $\C{B}=10^4$, $\tilde{L}=1$, $M=50$}}  &\multirow{3}{*}{\shortstack{$\alpha=0.45$, $\epsilon=10^{-3}$, $\beta=10^{-3}$,\\ $\C{B}=10^4$, $\tilde{L}=10$, $M=3$}}  \\
			& & &\\
			& & &\\
			\hline
		\end{tabular}
	\end{table}

	From figures \ref{quasi-ijcnn} and \ref{quasi-a9a}, we see that our algorithm is more efficient than other algorithms. The reason why DR-LM-DFP and D-LM-BFGS are very slow is that their quasi-Newton matrices are constructed using some significant regularization or damping techniques which could badly affect the approximation to the Hessian for capturing the second-order information. Furthermore, DBFGS is an inexact method which only converges to a small neighborhood of the solution.
	
	To sum up, numerical experiments show our new GDPDMs perform better than currently well-developed decentralized methods due to the applications of 
	quasi-Newton techniques to reduce the computational cost and the explorations of second-order information in both primal and dual updates to accelerate the convergence.

	\begin{figure}[H]
		\centering
		\subfloat[]{\includegraphics[width=3.5in]{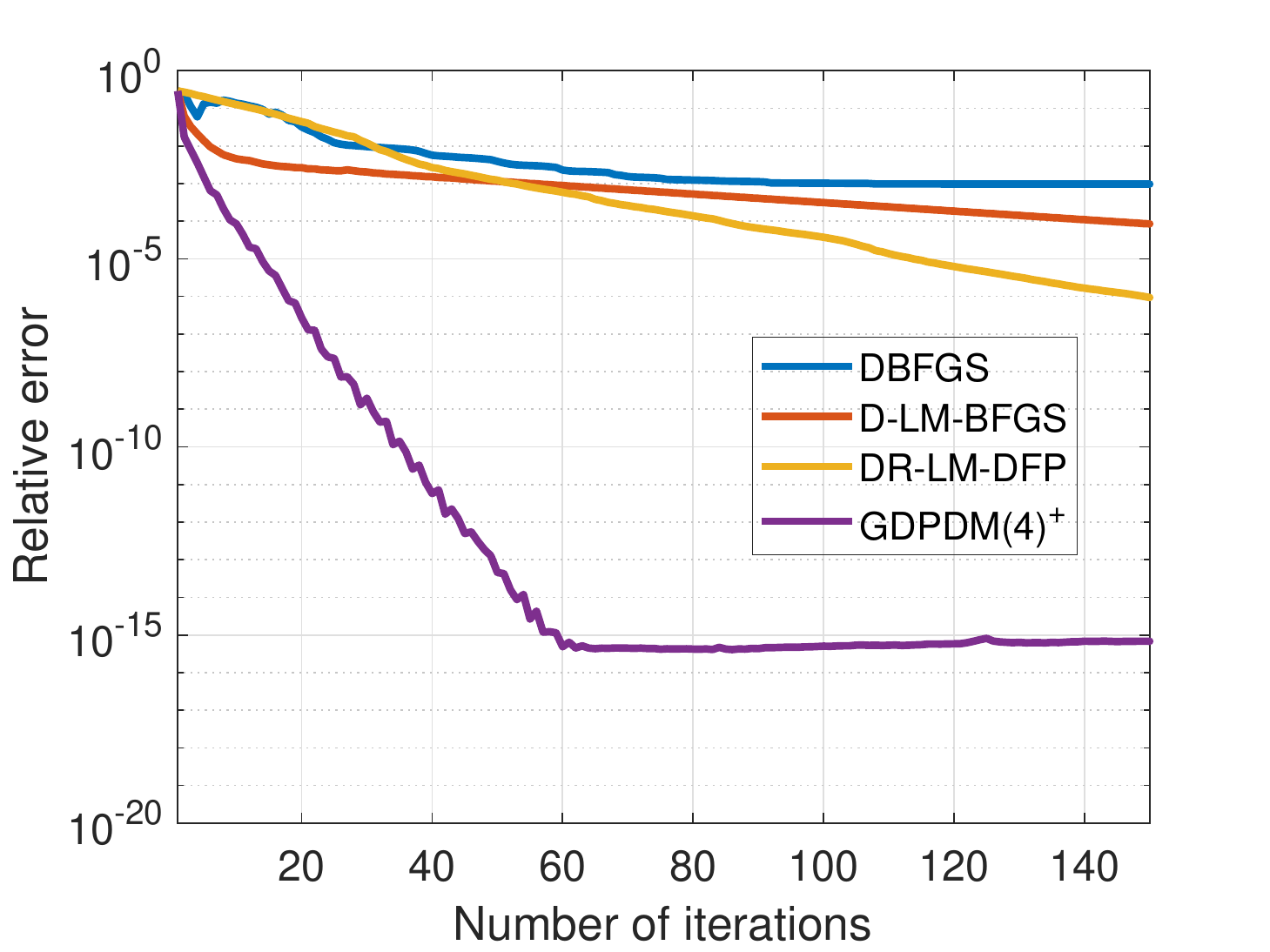}%
			\label{fig_9_a}}
		\hfil
		\subfloat[]{\includegraphics[width=3.5in]{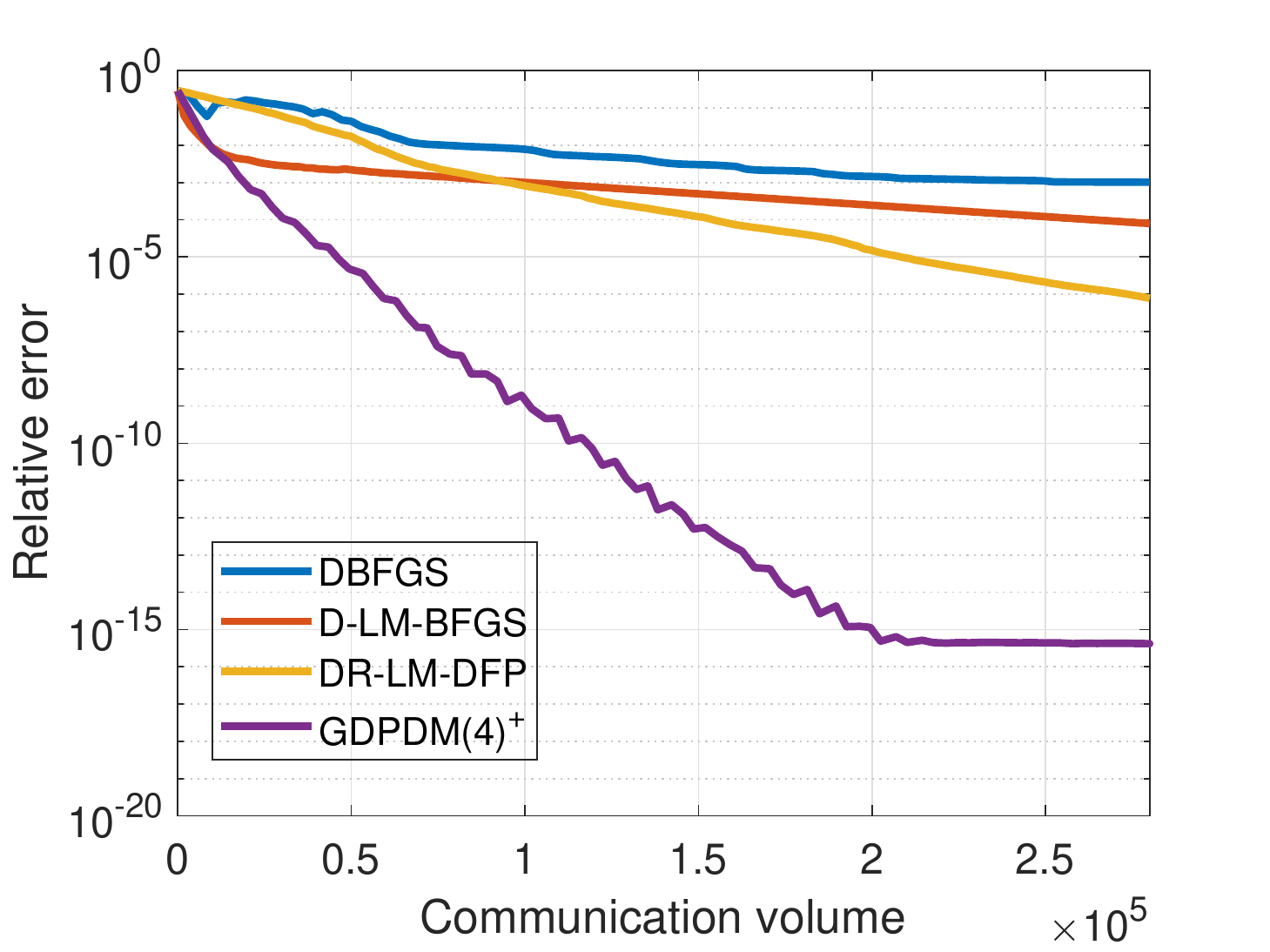}%
			\label{fig_9_b}}
		\hfil
		\subfloat[]{\includegraphics[width=3.5in]{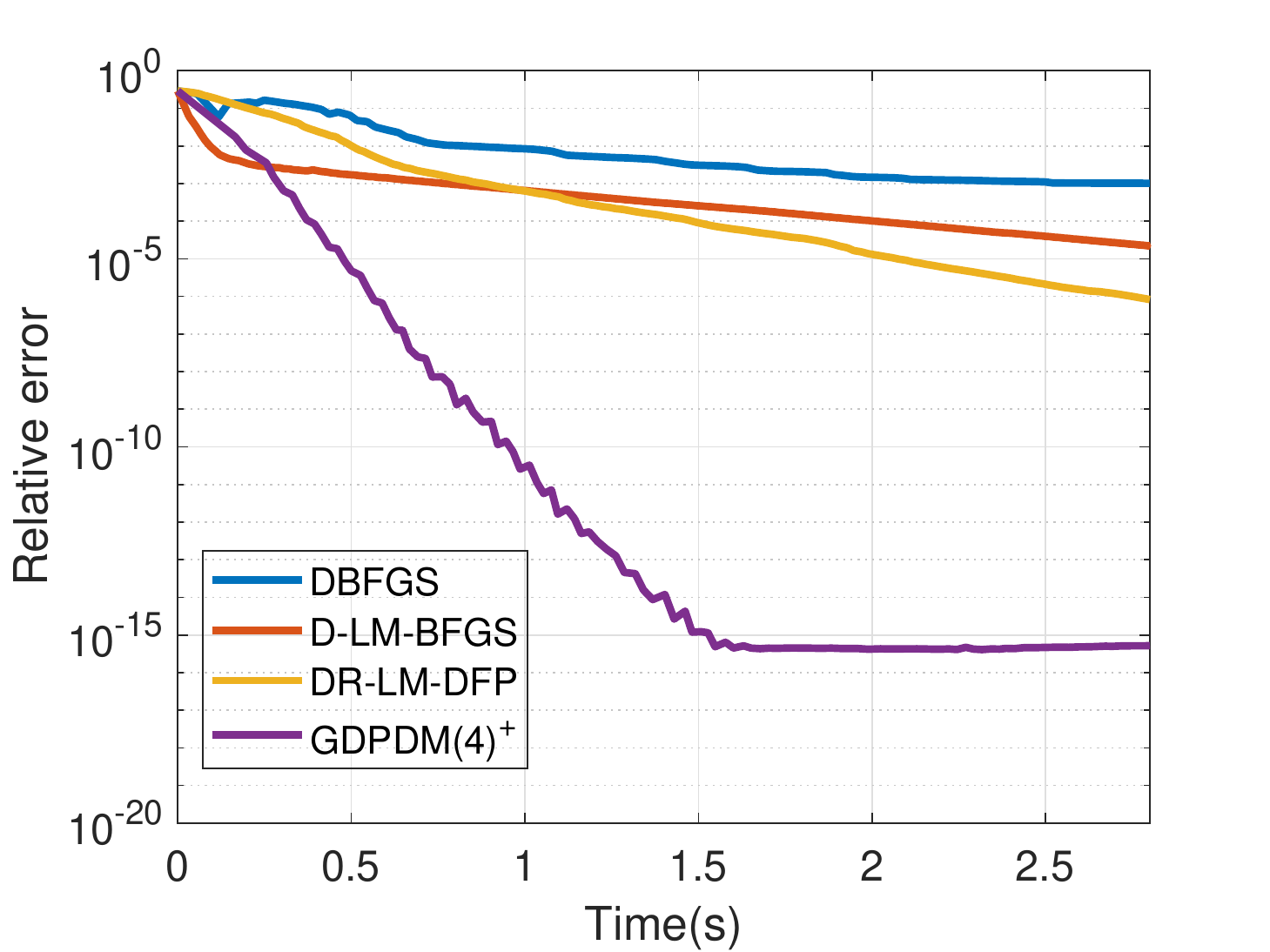}%
			\label{fig_9_c}}
		\caption{Comparisons with decentralized quasi-Newton algorithms in terms of the iteration number, communication volume, and CPU time (in seconds) using \textbf{ijcnn1} dataset.}
		\label{quasi-ijcnn}
	\end{figure}
	
	\begin{figure}[H]
		\centering
		\subfloat[]{\includegraphics[width=3.5in]{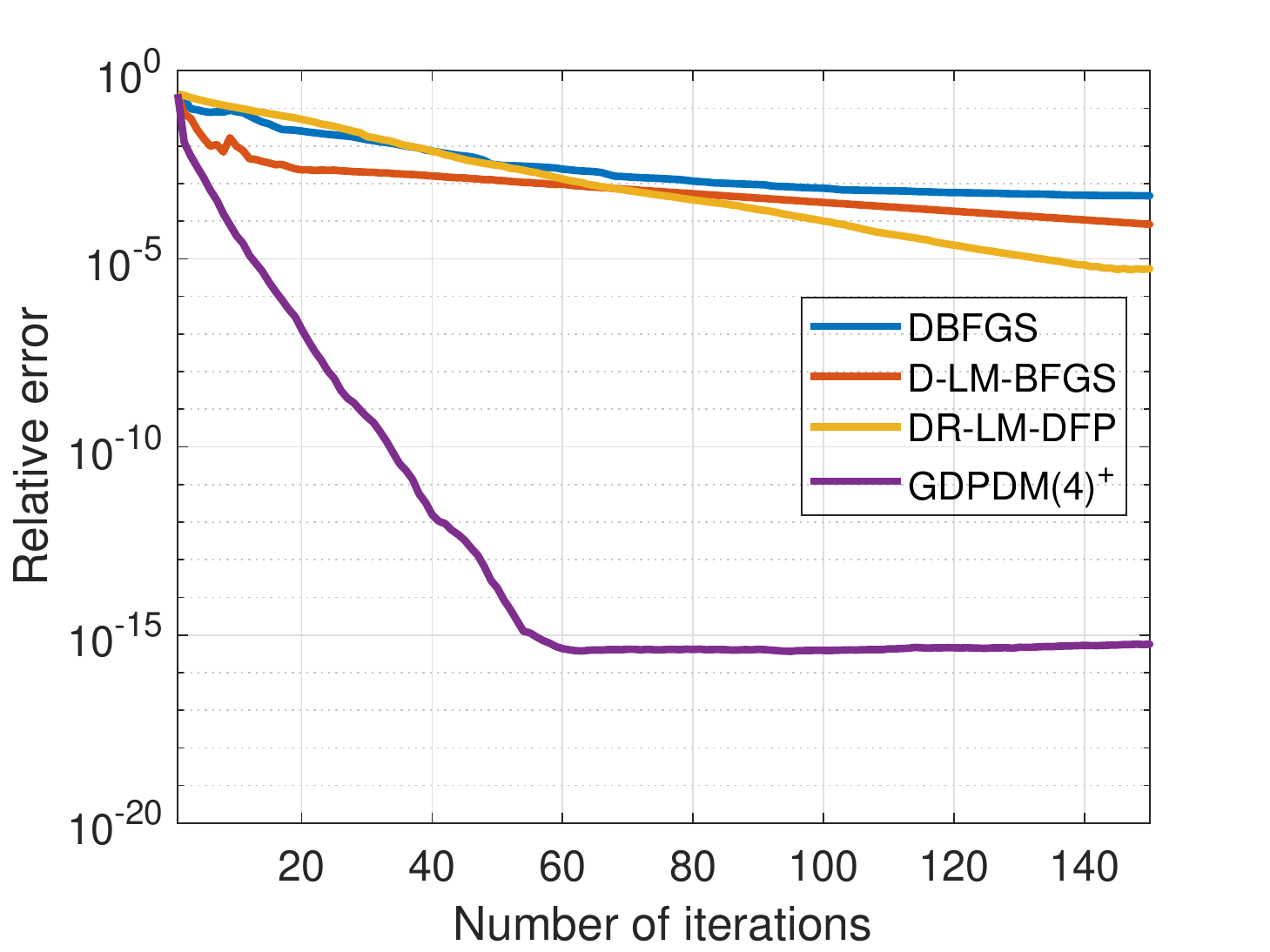}%
			\label{fig_10_a}}
		\hfil
		\subfloat[]{\includegraphics[width=3.5in]{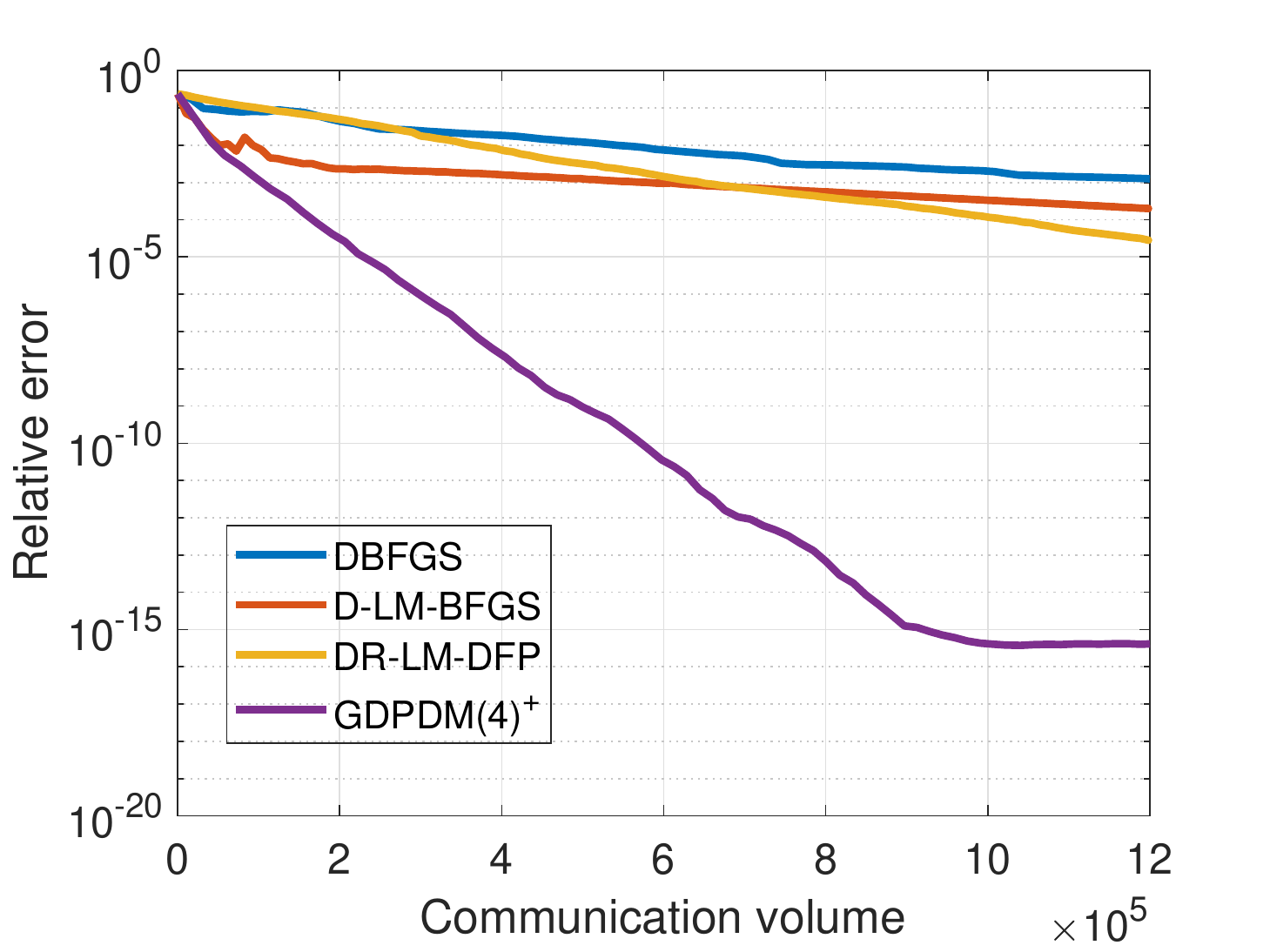}%
			\label{fig_10_b}}
		\hfil
		\subfloat[]{\includegraphics[width=3.5in]{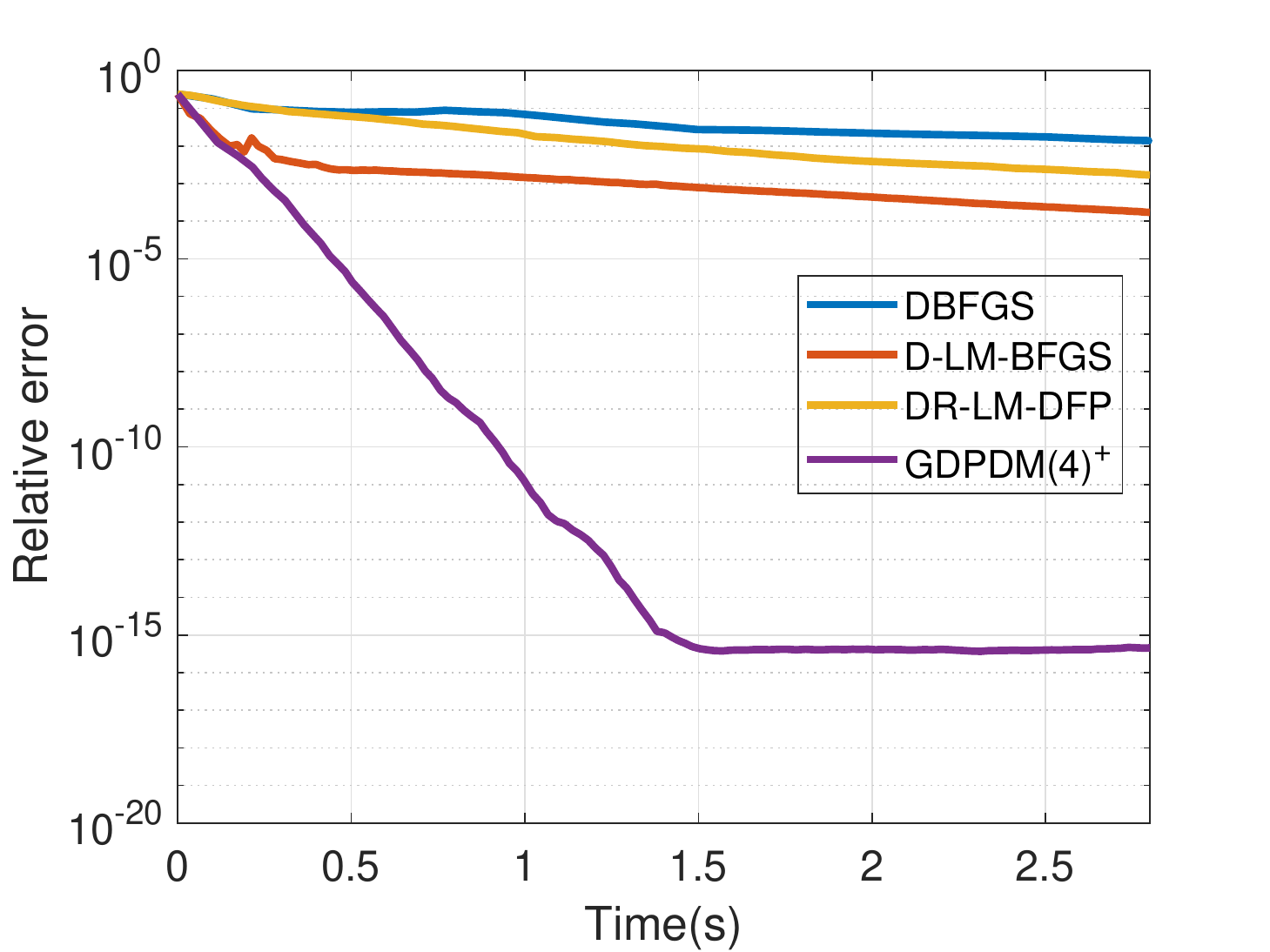}%
			\label{fig_10_c}}
		\caption{Comparisons with decentralized quasi-Newton algorithms in terms of the iteration number, communication volume, and CPU time (in seconds) using \textbf{a9a} dataset.}
		\label{quasi-a9a}
	\end{figure}
	
	\subsection{Effect of Network Topology}
	In the final subsection, we explore the impacts of the condition number of the network, denoted as $\kappa_g$, on our GDPDMs. 
	We choose GDPDM(1) as a demonstration and similar observations can be obtained for other GDPDMs.
	We conduct experiments on the linear regression problem \eqref{linear_problem} for different edge density $d= \{0.2, 0.36, 0.51, 0.67, 0.82, 1 \}$. 
	$d= 0.2$ corresponds to one of the sparsest networks (line graph) composed of 10 nodes, while $d=1$ corresponds to the complete graph. 
	For these mixing matrices with different edge density, $\kappa_g = \{39.8, 9.8, 4.2, 2.5, 1.9, 1.0\}$.
	
	We set $p=50$, $\kappa_f=100$. For $d =$ 0.2 (0.36; 0.51; 0.67; 0.82; 1), 
	we set $\omega =$ 0.28 (0.28; 0.26; 0.25; 0.26; 0.25) and $\beta=$ 0.1 (0.23; 0.42; 0.43; 0.46; 0.47). The remaining parameters are the same as those used in \textit{Linear Regression}.
	
	Figure \ref{graph} indicates the decentralized algorithm converges faster when the condition number of the network decreases, namely, the connectivity of the network becomes strong. 
	The complete graph whose condition number is 1 corresponds to the fastest convergence. However, the algorithm under the complete graph does not necessarily require the least communication volume 
	to reach a given relative error since the cost of each round of communication is also the highest. 
	It is important to note that real-world networks in applications, such as the Internet of Things and social networks, tend to be sparse.

	\begin{figure}[H]
		\centering
		\subfloat[]{\includegraphics[width=3.5in]{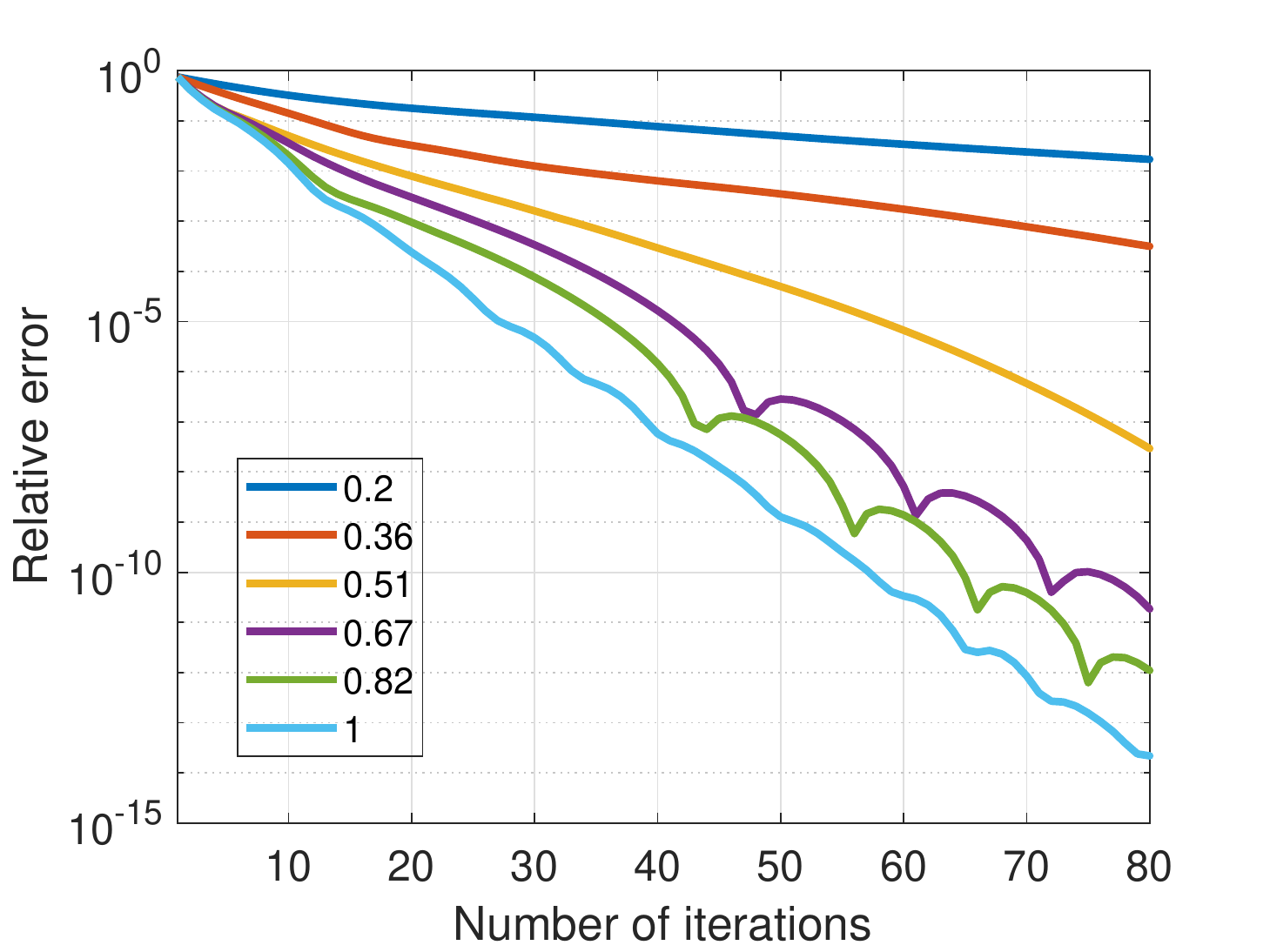}%
			\label{fig_12_a}}
		\hfil
		\subfloat[]{\includegraphics[width=3.5in]{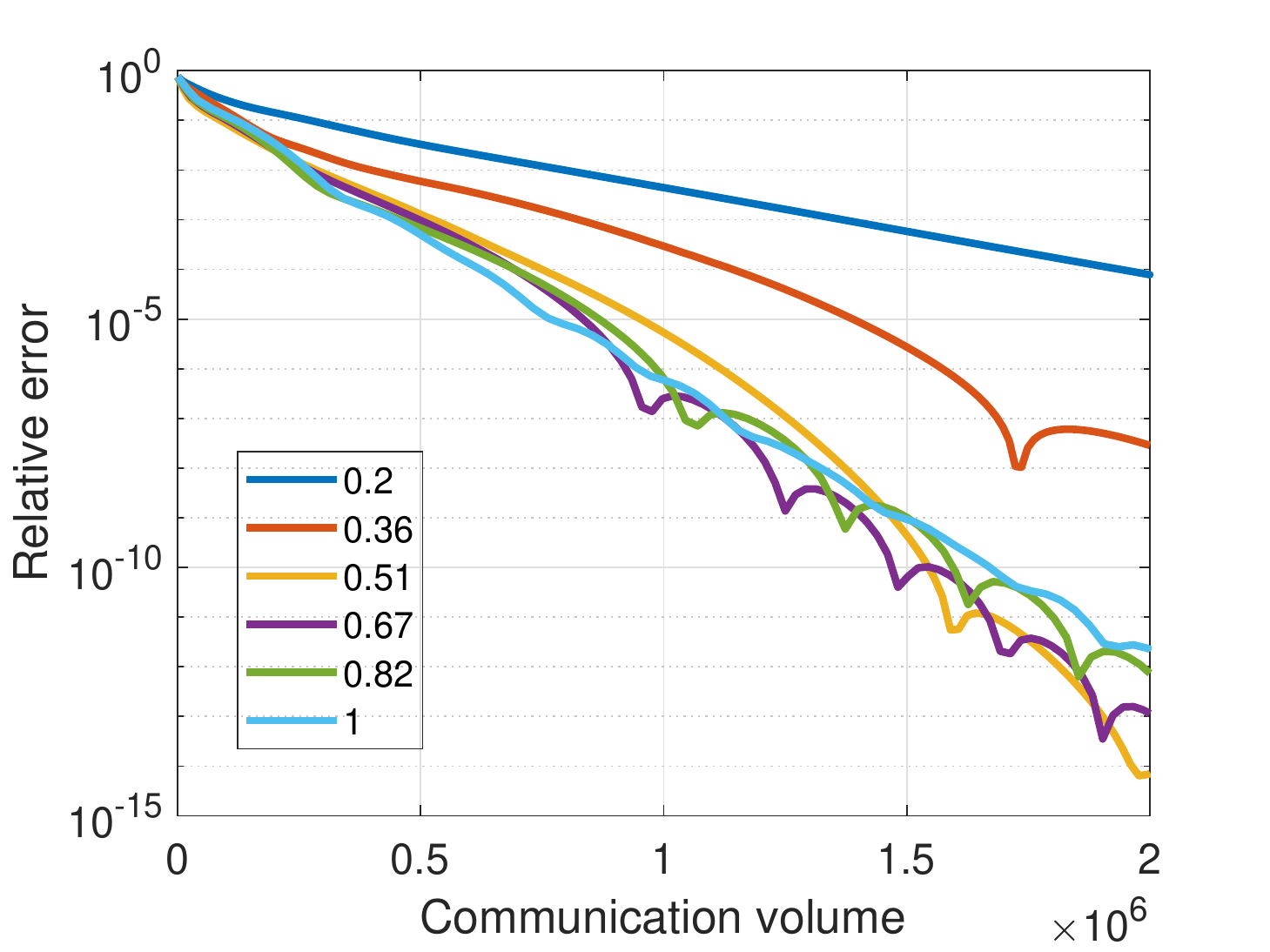}%
			\label{fig_12_b}}
		\caption{Relative error of GDPDM(1) versus the iteration number, communication volume for edge density 0.2, 0.36, 0.51, 0.67, 0.82 and 1.}
		\label{graph}
	\end{figure}
	
	\section{CONCLUSIONS}
	
	This paper considers a decentralized consensus optimization problem and proposes a decentralized primal-dual method(DPDM) and its generalization(GDPDM) with multiple primal updates per iteration. Both primal and dual updates effectively explore second-order information, where Hessian approximations are constructed with at most matrix-vector products. The single Jacobi loop with block-wise BFGS approximation is conducted in the primal domain. 
	Based on a new approximate dual gradient variation, the dual ascent step with a novel second-order correction is implemented in the dual domain. We show the update direction of each node asymptotically tends towards the global quasi-Newton direction. The relationship between GDPDM and some first or second-order methods is also established. 
	Under proper assumptions, GDPDMs have globally linear convergence. 
	The numerical results indicate that GDPDMs are not only very robust on the condition number of the problem,
	but also perform significantly better in terms of iteration number and CPU time than all the comparison decentralized methods, including EXTRA, GT, ESOM, NT, DBFGS, DR-LM-DFP, and D-LM-BFGS.

	
	\bibliographystyle{IEEEtran}
	\bibliography{refs}

	\end{document}